\documentclass[11pt]{amsart}
\usepackage[utf8]{inputenc}

\title{An explicit presentation for asymptotically rigid mapping class groups}
\author{Sergio Domingo-Zubiaga}
\date{March 2026}

\usepackage{enumitem}
\usepackage{mathrsfs}
\usepackage[export]{adjustbox}
\usepackage{array}
\usepackage[T1]{fontenc}
\usepackage[utf8]{inputenc}
\usepackage{amsfonts, amssymb, amsmath, amsthm}
\usepackage{graphicx}
\usepackage{parskip}
\usepackage{color}
\usepackage{float}
\usepackage{dirtytalk}
\usepackage{mathtools}
\usepackage{tikz-cd}
\usepackage{thm-restate}
\usepackage{ amssymb }
\usepackage{hyperref}
\usepackage{xcolor}
\usepackage{mathtools}
\usepackage{subcaption}
\usepackage{caption}

\begin{document}

\makeatletter
\renewcommand\subsubsection{\@startsection{subsubsection}{3}%
  {\z@}%
  {-3.25ex\@plus -1ex \@minus -.2ex}%
  {-0.5em}%
  {\normalfont\bfseries}}%
\makeatother

\newtheorem{theor}{Theorem}[section]
\newtheorem{lem}[theor]{Lemma}
\newtheorem{cla}[theor]{Claim}
\newtheorem{sublem}[theor]{Sublemma}
\newtheorem{cor}[theor]{Corollary}
\newtheorem{prop}[theor]{Proposition}

\newtheorem{theorem}{Theorem}
\renewcommand*{\thetheorem}{\Roman{theorem}}

\renewcommand{\thesubfigure}{\Alph{subfigure}}

\theoremstyle{remark}
\newtheorem{ex}[theor]{Example}
\newtheorem{rem}[theor]{Remark}
\theoremstyle{definition}
\newtheorem{defi}[theor]{Definition}

\def\Cent{\textnormal{Cent}}
\def\g{\textnormal{genus}}
\def\Map{\textnormal{Map}}
\def\In{\textnormal{Cap}}
\def\QA{\textnormal{QAut}}
\def\Aut{\textnormal{Aut}}
\def\So{\textnormal{Sym}_o}
\def\AI{\textnormal{AIsom}}
\def\capp{\textnormal{cap}}
\def\push{\textnormal{push}}
\def\forget{\textnormal{forget}}
\def\twist{\textnormal{twist}}
\def\Homeo{\textnormal{Homeo}^+}
\def\X{\chi}
\def\Xc{\mathfrak{X}}
\def\R{\mathfrak{R}}
\def\P{\mathcal P}
\def\Pc{\mathcal P_d}
\def\B{\mathcal B}
\def\C{\mathcal C}
\def\G{\mathcal G}
\def\T{\mathcal T}
\def\D{\mathcal D}
\def\K{\mathcal K}
\def\S{\mathcal S}
\def\v{\mathbf v}

\def\dl{\textnormal{lk}^{\downarrow}}

\def\b{\beta}
\def\a{\alpha}
\def\c{\gamma}
\def\t{\theta}
\newcommand{\nt}[1]{\textcolor{blue}{#1}}
\newcommand{\ch}[1]{\textcolor{violet}{#1}}

\def\BB{\mathbb B}
\def\OD{\textnormal{OD}_d}
\def\RR{\mathbb R}
\def\ZZ{\mathbb Z}
\def\SS{\mathbb S}
\def\NN{\mathbb N}
\def\DD{\mathbb D}
\def\H{\mathcal{H}}

\def\HTB{\textnormal{HT}\BB}
\def\HTD{\textnormal{HTD}_d}
\def\DTD{\textnormal{DT}\D}
\def\TCD{\textnormal{TC}\D}
\def\CD{\textnormal{C}\D}
\def\TH{\textnormal{TH}}
\def\Hc{\textnormal{H}}
\def\Sym{\textnormal{Sym}}
\def\id{\textnormal{id}}
\def\int{\textnormal{int}}
\def\Lk{\textnormal{lk}}
\def\Star{\textnormal{star}}
\def\Image{\textnormal{Im}}
\def\Stab{\textnormal{Stab}}

\newcommand{\rom}[1]{\uppercase\expandafter{\romannumeral #1\relax}}

\begin{abstract}
We show that several families of asymptotically rigid mapping class groups arise as explicit quotients of the fundamental group of a graph of groups, with mapping class groups as vertex and edge stabilizers. Using this description, and building on the work of Labruère and Paris \cite{Labruere-Paris}, we compute explicit presentations for asymptotically rigid mapping class groups of surfaces.
\end{abstract}

\maketitle

In recent years, asymptotically rigid mapping class groups of manifolds have been in the spotlight for a variety of reasons (see \cite{Fun-Kap1, Funar, Fun-Kap2, BrThom}). First, they sometimes appear as extensions of Higman-Thompson groups by direct limits of mapping class groups of compact manifolds. In certain cases, this implies that their homology coincides with the stable homology of the underlying manifold (see \cite[Section 10]{Asymp24} and \cite[Section 5]{me}). Additionally, in the context of surfaces, these groups form countable subgroups of the so-called big mapping class groups, an active area of research at the intersection of geometric group theory, low-dimensional topology, topological groups, etc.

One of the main sources of motivation for the study of asymptotically rigid mapping class groups stems from the investigation of their finiteness properties (see \cite{Asymp24, SurfHoughton, Ara-Fun, me, Fun-Kap2, BrThom, hill2025asymptoticallyrigidmappingclass}). In particular, these properties are often closely related to the space of ends of the underlying manifold. For instance, for the families of groups considered in \cite{Asymp24, SurfHoughton, me, hill2025asymptoticallyrigidmappingclass}, the associated asymptotically rigid mapping class group is finitely presented whenever the space of ends contains at least three points. These results are non-constructive, which naturally raises the question of how to obtain explicit finite presentations for this class of groups.

For surfaces of genus zero, this had previously been settled by the work of Funar and Kapoudjian \cite{Fun-Kap1}. More recently, Genevois, Lonjou and Urech \cite[Theorem 3.20]{genevoisIII} gave a finite presentation for the asymptotically rigid mapping class group of the planar surfaces with punctures defined in \cite{BrThom}; and Hill, Kwan, Udall, and West \cite{hill2025asymptoticallyrigidmappingclass} provided a finite presentation for the \textit{pure graph Houghton group}, a finite-index subgroup of the asymptotically rigid mapping class group of the graph. 

\subsection{Statement of results.}\label{MainResults}

Our first result shows that several families of asymptotically rigid mapping class groups are an explicit quotient of the fundamental group of a graph of groups, with mapping class groups as vertex and edge stabilizers. Before stating our result, a graph is a \textit{snake} if it is isomorphic to the graph in Figure \ref{Snake}, for any finite number of vertices. The precise statement of Theorem \ref{Main1} is somewhat involved; we offer here only an abridged version, postponing the details until Section \ref{Sect1}.

\begin{theorem}\label{Main1}
Let $\B$ be a group in one of the following families:
\begin{enumerate}
    \item asymptotic mapping class groups of Cantor manifolds from \cite[Theorem 1.7, Theorem 1.8, Theorem 1.11]{Asymp24},
    \item surface Houghton groups $\B_r$ from \cite{SurfHoughton}, for $r\geq 3$,
    \item graph Houghton groups $B(g, h, r)$ with $r\geq 3$, and asymptotically rigid mapping class groups $B_{d,r}(1, 1)$ with $d\geq 2$ from \cite{hill2025asymptoticallyrigidmappingclass},
    \item asymptotically rigid handlebody groups $\H_{d,r}(O,Y)$ from \cite{me}, for either $d\geq 2$ of $r\geq 3$.
\end{enumerate}
Then $\B$ is isomorphic to an explicit quotient of the fundamental group of a graph of groups $\G_\B$ over a snake graph, where vertex groups are explicit finite extensions of mapping class groups of compact manifolds, and the extra relations correspond to identifying mapping class groups of common submanifolds.

\end{theorem}

\begin{figure}[H]
\begin{center}
\begin{tikzpicture}
  \filldraw (0,0) circle (2pt) ;
  \filldraw (1,0.5) circle (2pt) ;
  \filldraw (0,1) circle (2pt) ;
  \filldraw (1,1.5) circle (2pt);

  \filldraw (0,2) circle (2pt);
  \filldraw (1,2.5) circle (2pt);
  \filldraw (0,3) circle (2pt) ;
  \filldraw (1,3.5) circle (2pt) ;

  \draw (0,0) --  (1,0.5);
  \draw (0,0) --  (0,1);
  \draw[dashed] (1,0.5) --  (1,1);
  \draw (1,0.5) --  (0,1);
  \draw (1,1.5) --  (0,1);
  \draw (1,1.5) --  (1,0.5);
  \draw[dashed] (0,1) --  (0,1.5);
  \draw[dashed] (1,1.5) --  (0.7,1.65);
  \draw[dashed] (0,2) --  (0.3,1.85);
  
  \draw (1,2.5) --  (1,3.5);
  \draw (1,2.5) --  (0,3);
  \draw (0,3) --  (1,3.5);
  \draw (0,2) --  (0,3);
  \draw (0,2) --  (1,2.5);
  \draw[dashed] (1,2) --  (1,2.5);

\end{tikzpicture}

\end{center}
\caption{A snake graph.}
\label{Snake}
\end{figure}
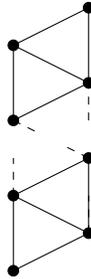

Theorem \ref{Main1} is proven in Section \ref{Sect1}, where we also describe vertex and edge groups.

In Section \ref{Sect2}, starting from the presentation for the mapping class group of a finite type surface of Labruère and Paris \cite{Labruere-Paris}, we determine finite presentations for the vertex groups and finite generating sets for the edge groups of $\G_\B$ when $\B$ is the asymptotically rigid mapping class group of a surface, which we use to compute explicit finite presentations for $\B$ in the surface case (see Sections \ref{Sect4} and \ref{Sect5}). Here, a \textit{boundary swap} is the boundary swapping analog of a half twist in a sense that will be made clear in Subsection \ref{Gens-bpmcg}.

\begin{theorem}\label{Main2}
 Let $O$ and $Y$ be homeomorphic to $\SS^1\times\SS^1$. 
 \begin{enumerate}[label=\alph*)]
     \item The surface Houghton group $\B_{1,r}(O, Y)$ with $r\geq 3$ of \cite{SurfHoughton} has an explicit finite presentation with  $27+4r$ Dehn twists and $4r-4$ boundary swaps as generators, and relations given in Theorem \ref{Presentation1}.
     \item The asymptotically rigid mapping class group $\B_{2,1}(O, Y)$ of \cite{Asymp24} has an explicit finite presentation with  $93$ Dehn twists and $21$ boundary swaps as generators, and relations given in Theorem \ref{Presentation2}.
 \end{enumerate}

\end{theorem}

As will become apparent, the method of Sections \ref{Sect4} and \ref{Sect5} may be used to compute an explicit presentation of $\B_{d,r}(O,Y)$ for any surface $O$, a torus $Y$, and any $d$ and $r$ such that either $d \geq 2$ or $r \geq 3$, although the technical details of the calculations vary in each particular case.

Furthermore, the techniques in this paper can in principle be used to compute explicit presentations for any group $\B$ in Theorem \ref{Main1}. The procedure is as follows: given presentations for the\textit{ boundary permuting mapping class group} of the underlying \textit{suited} manifolds (resp. graphs), one first obtains finite presentations for the vertex groups and finite generating sets for the edge groups of the associated graph of groups $\G_\B$. An explicit presentation for $\B$ is then obtained by reasoning as in Sections \ref{Sect4} and \ref{Sect5}.

\textbf{Acknowledgments.} The author is indebted to Javier Aramayona for proposing the research questions, providing valuable insights into asymptotically rigid mapping class groups, and offering guidance on various aspects of this work. The author acknowledges financial support from the grants CEX2019-000904-S, PGC2018-101179-B-I00, and PID2024-155800NB-C31, all funded by MCIN/AEI.

\section{Preliminaries}\label{Sect0}

Mapping class groups are defined differently depending on the underlying space. For an orientable manifold $M$, $\Map(M)$ is the group of orientation-preserving diffeomorphisms of $M$ fixing $\partial M$, up to isotopy; for a locally finite graph $\Gamma$, $\Map(\Gamma)$ consists of homotopy equivalences of $\Gamma$ modulo homotopy; and for a handlebody $H$, $\Map(H)$ is the group of homeomorphisms of $\partial H$ that extend to $H$, up to isotopy.

Throughout this paper, we assume the reader is familiar with the standard literature on mapping class groups, and refer to \cite{Farb-Margalit} for further background. For the sake of consistency and clarity, throughout this section we define the relevant objects in the surface case. The constructions for the other cases are analogous, and we refer the reader to \cite{Asymp24,SurfHoughton,me,hill2025asymptoticallyrigidmappingclass} for the precise definitions. Nevertheless, when there are key differences we will say so explicitly.

Henceforth, all surfaces are assumed to be connected and orientable, and for such a surface $\Sigma$ we define $$\Map(\Sigma) = \Homeo(\Sigma,\partial \Sigma)/\Homeo_0(\Sigma,\partial \Sigma),$$ where $\Homeo(\Sigma,\partial \Sigma)$ is the group of orientation-preserving homeomorphisms fixing the boundary pointwise, and $\Homeo_0(\Sigma,\partial \Sigma)$ is the group of isotopies fixing the boundary pointwise.

\subsection{Boundary permuting mapping class group}\label{mapo}

Given a surface $\Sigma$ with $b$ boundary components $\{A_1,...,A_b\}$, fix a set of parametrizations for the $A_i$, i.e. a set of orientation-preserving homeomorphisms $\{\mu_i:\SS^1\rightarrow A_i\}_{0\leq i\leq b}$. The \textit{boundary permuting mapping class group} $\Map_{\{\mu_i\}}(\Sigma)$ is the group of isotopy classes (relative to $\partial \Sigma$) of homeomorphisms of $\Sigma$ that respect the $\mu_i$, meaning that for any $\phi\in\Map_{\{\mu_i\}}(\Sigma)$, there exists a permutation $\sigma$ of $\{1,...,b\}$ such that $\phi\circ \mu_i=\mu_{\sigma(i)}$ $\forall i$. For any two sets of parametrizations $\{\mu_i:\SS^1\rightarrow A_i\}_{0\leq i\leq b}$ and $\{\nu_i:\SS^1\rightarrow A_i\}_{0\leq i\leq b}$ there is an isomorphism between $\Map_{\{\mu_i\}}(\Sigma)$ and $\Map_{\{\nu_i\}}(\Sigma)$ given by conjugation by a homeomorphism $\psi$ of $\Sigma$ satisfying $\psi(\mu_i(s))=\nu_i(s)$ for all $s\in\SS^1$ and $i\in\{1,...,b\}$. Hence we denote $\Map_o(\Sigma):=\Map_{\{\mu_i\}}(\Sigma)$ to simplify notation. Note that $\Map(\Sigma)$ is a finite index subgroup of $\Map_o(\Sigma)$.

\subsection{Tree surfaces and rigid structures}

Let $d\geq1$ and $r\geq1$ be integers, $O$ and $Y$ compact surfaces without boundary. Remove $r$ open discs from $O$, giving rise to a surface $O^r$ with $r$ boundary components which we denote by $\{A_1,...,A_r\}$. Similarly, remove $d+1$ open discs from $Y$, giving rise to a surface $Y^d$ with $d+1$ boundary components which we denote by $\{B_0,B_1,...,B_d\}$. Fix orientation-reversing homeomorphisms $\mu_i:A_i\rightarrow B_0$ for $i\in[1,...,r]$ and $\lambda_j:B_j\rightarrow B_0$ for $i\in[1,...,d]$. We define a set of surfaces $\{M_i\}_{i\in\NN}$ as follows:

\begin{itemize}
    \item $M_1=O^r$,
    \item $M_2$ is the result of gluing a copy of $Y^d$ to each $\{A_1,...,A_r\}$ using the maps $\mu_i:A_i\rightarrow B_0$.
    \item For $k\geq 3$, $M_k$ is the result of gluing a copy of $Y^d$ to each boundary component in $M_{
    k-1}\setminus M_{k-2}$ using the maps  $\lambda_j:B_j\rightarrow B_0$.
\end{itemize}

\begin{figure}[H]  
\centering
\includegraphics[width=10.5cm]{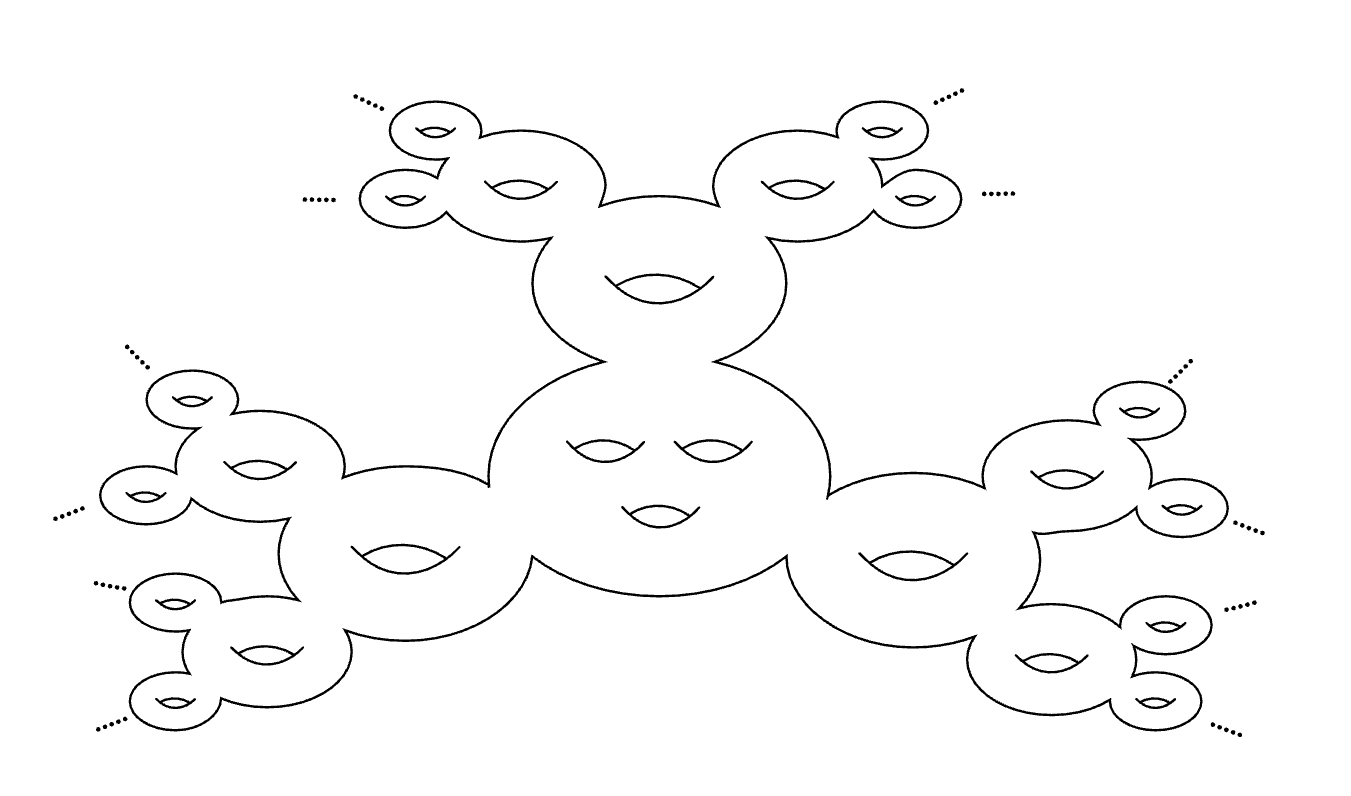}
\caption{The tree surface $\S_{2,1}(O,Y)$, with $O$ of genus 3 and $Y$ of genus 1.}
\label{TreeSurf}
\end{figure}

We define the \textit{tree surface} $\S_{d,r}(O,Y)$ as $\cup_{k=1}^\infty M_k$. Each of the connected components of $\overline{M_k\setminus M_{k-1}}$ for $k\geq 1$ is called a \textit{piece}. By construction, each piece is homeomorphic to $Y^d$. A subsurface $\Sigma\subset\S_{d,r}(O,Y)$ is \textit{suited} if it is connected and is the union of $O^r$ and finitely many pieces. A piece $P$ is \textit{adjacent} to a suited subsurface $\Sigma$ if $P\not\subset \Sigma$ and $P\cap \Sigma\neq \emptyset$. For each piece $P$ of $\S_{d,r}(O,Y)$, choose a homeomorphism $i_P:P\rightarrow Y^d$ that agrees with the $\mu_i$'s and the $\lambda_i$'s used in defining $\S_{d,r}(O,Y)$.  

\begin{rem}\label{Remark1}
When $d\geq2$, the surfaces $\S_{d,r}(O,Y)$ correspond to \textit{Cantor surfaces} as introduced in \cite{Asymp24}, whose end space is a Cantor set. On the other hand, the surfaces $\S_{1,r}(O,Y)$ include the surfaces $\Sigma_r$ in \cite{SurfHoughton}, whose end space consists on $r$ isolated points. In this sense, tree surfaces encompass both constructions.
\end{rem}

\begin{defi}\label{rigstru} (Preferred rigid structure). The set $\{i_P\: : \: P \text{ is a piece}\}$ is called the \textit{preferred rigid structure} on $\S_{d,r}(O,Y)$.
\end{defi}

\subsection{Asymptotically rigid mapping class groups} Let $\S_{d,r}(O,Y)$ be a tree surface equipped with its preferred rigid structure. A homeomorphism $f:\S_{d,r}(O,Y)\rightarrow \S_{d,r}(O,Y)$ is \textit{asymptotically rigid} if there exists a suited subsurface $\Sigma\subset \S_{d,r}(O,Y)$ such that

\begin{enumerate}
    \item $f(\Sigma)$ is suited,
    
    \item $f$ is \textit{rigid outside of $\Sigma$}, that is, for every piece $P$ outside of $\Sigma$, $f(P)$ is also a piece and $i_{f(P)}\circ f_{|P}\circ i^{-1}_P=\id_{Y^d}$.
\end{enumerate}

\begin{defi}\label{def-Asy}(Asymptotically rigid mapping class group)
The \textit{asymptotically rigid mapping class group} $\B_{d,r}(O,Y)$ is the group of asymptotically rigid maps of $\S_{d,r}(O,Y)$, up to isotopy.
\end{defi}

\begin{rem}

The asymptotically rigid mapping class group is analogously defined for every group in Theorem \ref{Main1}, and we still denote it by $\B_{d,r}(O,Y)$, where $O$ and $Y$ are either specific compact manifolds without boundary, finite graphs or compact handlebodies. Although the asymptotically rigid mapping class group of a manifold is generally defined up to \textit{proper isotopy} (see \cite[Subsection 3.3, page 14]{Asymp24}), we use isotopy as both definitions are equivalent in the surface case by \cite[Proposition 3.11]{Asymp24}.
\end{rem}

\subsection{Stein-Farley cube complex}

Consider all ordered pairs $(\Sigma,f)$ where $\Sigma$ is a suited subsurface of $\S_{d,r}(O, Y)$ and $f\in\B_{d,r}(O, Y)$. We deem two such pairs $(\Sigma_1,f_1)$ and $(\Sigma_2,f_2)$ to be equivalent, and write $(\Sigma_1,f_1)\sim (\Sigma_2,f_2)$, if there are representing homeomorphisms (abusing notation) $f_1$ and $f_2$ such that $f_2^{-1}\circ f_1$ maps $\Sigma_1$ onto $\Sigma_2$ and is rigid outside of $\Sigma_1$. We denote by $[(\Sigma,f)]$ the equivalence class of the pair $(\Sigma,f)$ with respect to this relation, and write $\P$ for the set of equivalence classes. Observe that $\B_{d,r}(O, Y)$ acts on $\P$ by left multiplication, namely $g\cdot[(\Sigma,f)]=[(\Sigma,g\circ f)]$.

Consider a pair $(\Sigma,f)$. Since $\Sigma$ is a suited subsurface, it is the union of finitely many pieces and $O^r$. We define the\textit{ height} $h((\Sigma,f))$ of the pair $(\Sigma,f)$ as the number of pieces in $\Sigma$. Note that if $(\Sigma_1,f_1)\sim (\Sigma_2,f_2)$ then $h((\Sigma_1,f_1))=h((\Sigma_2,f_2))$, and thus $h$ descends to a map (abusing notation) $h:\P\rightarrow\NN$, setting the height $h([(\Sigma,f)])$ to be the height of any representative.

We introduce a relation $\preceq$ on the elements of $\P$ by declaring $x_1\preceq x_2$ if and only if $x_1=[\Sigma_1,f]$ and $x_2=[\Sigma_2,f]$ for suited subsurfaces $\Sigma_1\subset \Sigma_2$ such that $\overline{\Sigma_2\setminus \Sigma_1}$ is a disjoint union of pieces.

Define \textit{a closed interval} $[x,y]$ to be the set of elements $z\in\P$ such that $x\preceq z\preceq y$. The relation $\preceq$ can be used to construct a cube complex $\Xc_{d,r}(O, Y)$ with $\P$ as its 0-skeleton and every $[x,y]$ as a $d$-cube with $d=h(y)-h(x)$ (See \cite[Proposition 5.11]{Asymp24} or \cite[Section 4]{SurfHoughton} for details). We will refer to the cube complex $\Xc_{d,r}(O, Y)$ as the \textit{Stein-Farley cube complex} associated to the tree surface $\S_{d,r}(O, Y)$. 

The Stein-Farley cube complex can be analogously defined for any of the asymptotically rigid groups in Theorem \ref{Main1}. The fact that $\Xc_{d,r}(O, Y)$ is a cube complex is not obvious, and we want to remark that it relies on the fact that the corresponding underlying manifolds satisfy the \textit{inclusion} and \textit{intersection} properties. We refer the reader to \cite[Subsection 3.1]{Asymp24} for the precise definitions of these properties.

\subsection{Graphs of groups and fundamental groups}\label{1.5} Let $\Gamma$ be a finite undirected graph with vertex set $V(\Gamma)$ and edge set $E(\Gamma)$. A \textit{graph of groups} $\G$ over a finite graph $\Gamma$ consists of \begin{itemize}
    \item a group $G_v$ for each vertex $v\in V(\Gamma)$,
    \item a group $G_e$ for each directed edge $e=(v_0,v_1)\in E(\Gamma)$, together with a pair of monomorphisms: $$\varphi_{e,0}:G_e\rightarrow G_{v_0}, \quad \varphi_{e,1}:G_e\rightarrow G_{v_1}.$$
\end{itemize} 

Let $T$ be a tree formed by edges in $E(\Gamma)$ containing every vertex in $V(\Gamma)$. Assign a symbol $y_e$ to each $e=(v_0,v_1)\in E(\Gamma)$. The \textit{fundamental group} of $\G$ is the free product $$(\ast_{v \in V(\Gamma)}\; G_v)*(\ast_{e \in E(\Gamma)}\;\langle y_e\rangle)$$ modulo the relations:
\begin{itemize}
\item $y_{\bar{e}} = y^{-1}_e$ for all $e \in E(\Gamma)$, where if $e=(v_i,v_j)$ then $\bar{e}=(v_j,v_i)$,
\item $y_e=1$ for all $e\in T$,
\item $y_e\varphi_{e,0}(x)y_e^{-1}=\varphi_{e,1}(x)$ for all $e\in T$ and every $x\in G_e$.
\end{itemize}

This definition does not depend on the choice of $T$, see \cite[Proposition 20]{Trees}.

\section{Proof of Theorem \ref{Main1}}\label{Sect1}

To prove Theorem \ref{Main1}, we introduce a theorem of Brown which provides a method for expressing a group as the quotient of the fundamental group of a graph of groups, given that the group acts on a simplicial complex under suitable conditions.

\begin{theor}\label{teoBrown}\cite[Theorem $1'$]{BROWN19841}
Let $G$ be a group acting on a simplicial complex $K$, and $W_K$ a subcomplex of $K$ such that:
\begin{enumerate}
    \item $K$ is simply-connected,
    \item Every simplex of $K$ is equivalent modulo $G$ to a unique simplex of $W_K$,
    \item $W_K$ has finite 2-skeleton,
    \item The stabilizer $G_v$ of every vertex $v\in K$ is finitely presented,
    \item The stabilizer $G_e$ of every edge $e\in K$ is finitely generated.
\end{enumerate}

Let $\hat{G}$ be the fundamental group of the graph of groups over the 1-skeleton of $W_K$, with stabilizers of vertices $G_v$ as vertex groups, stabilizers of edges $G_e$ as edge groups, and with the inclusion as the maps $\varphi_{e,0}$ and $\varphi_{e,1}.$ Then $G$ is isomorphic to the quotient of $\hat{G}$ by the relations: $$y_e=1 \textit{ for all }e\in E(W_k).$$ 
\end{theor}

Whenever a subcomplex $W_K\subset K$ satisfies hypotheses (2) and (3) in the theorem, we say that $W_K$ is a \textit{fundamental domain} of the action of $G$ on $K$. We dedicate the rest of the section to introducing a simplicial complex where each of groups in Theorem \ref{Main1} act satisfying every hypothesis of Theorem \ref{teoBrown}, hence proving Theorem \ref{Main1}. In contrast to the previous section, the definitions and proofs presented are formulated in the general setting, in order to be as precise as possible.

\subsection{A cocompact action}\label{1.1}
Let $\B$ be one of the asymptotically rigid groups in Theorem \ref{Main1}, and $\Xc$ its associated Stein-Farley complex. The complex $\Xc$ serves as a starting point in the search for a complex satisfying every hypothesis in Theorem \ref{teoBrown}. We will modify $\Xc$ to obtain a complex where $\B$ acts cocompactly, this is necessary for the existence of a fundamental domain $W_K$ as it must satisfy hypothesis (3) in Theorem \ref{teoBrown}. 

The \textit{descending link} $\dl(v)$ of $v\in \Xc$ is a simplicial complex with $n$-simplices corresponding to $n+1$-dimensional simplices $\sigma\in \Xc$ with $v\in \sigma$ such that every vertex $w\in\sigma$ satisfies $h(w)\leq h(v)$. Denote by $\Xc^{h_0}$ to the subcomplex of $\Xc$ spanned by vertices of height at most $h_0$.

\begin{rem}
Although the tree surfaces in \cite{SurfHoughton} are of the form $\S_{1,r}(O,Y)$ for $O$ of genus 0, the definition of the Stein-Farley cube complex $\Xc_{1,r}(O,Y)$ and the properties it enjoys do not depend on the genus of $O$, since the inclusion and intersection properties are satisfied independently of $O$. Hence we will freely use the arguments and results of \cite{SurfHoughton}, regardless of the genus of $O$.

\end{rem}

\begin{lem}\label{LemDeslink}
Let $\B$ be one of the asymptotically rigid groups in Theorem \ref{Main1}. There exists $h_0\in\NN$ such that the complex $\Xc^{h_0}$ is simply connected.
\end{lem}

\begin{proof}[\sc Proof of Lemma {\rm \ref{LemDeslink}}]

According to \cite[Proposition A.4]{Asymp24}, if there exists $h_0\in \NN$ such that $\dl(v)$ is simply connected for every vertex of height $h(v)> h_0$, then the inclusion $\Xc^{h_0}\hookrightarrow \Xc$ induces an isomorphism $$f:\pi_1(\Xc^{h_0})\rightarrow\pi_1(\Xc).$$ Since the complex $\Xc$ is contractible  (see \cite[Proposition 5.7]{Asymp24}, \cite[Theorem 4.1]{SurfHoughton}, \cite[Theorem 2.4]{me}, \cite[Theorem 5.11]{hill2025asymptoticallyrigidmappingclass}), and hence simply connected, if there exists such $h_0$, the complex $\Xc^{h_0}$ will be simply connected because of the isomorphism $f$.

For each group in Theorem \ref{Main1}, there exists $h_0\in\NN$ such that every $\dl(v)$ with $h(v)\geq h_0$ is simply connected:
\begin{itemize}
\item For asymptotically rigid mapping class group of Cantor manifolds because of \cite[Corollary 6.7, Theorem 7.1, Theorem 8.1, Theorem 9.1]{Asymp24}.
\item For surface Houghton groups $\B_r$ because of \cite[Corollary 4.4, Corollary 5.2]{SurfHoughton}, as long as $r\geq 3$.
\item For Houghton graph groups $B(g,h,r)$ because of \cite[Corollary 5.22]{hill2025asymptoticallyrigidmappingclass}, as long as $r\geq 3$, and for groups $B_{d,r}(1, 1)$ because of \cite[Subsection 7.1]{hill2025asymptoticallyrigidmappingclass}, given that $d\geq2$.

\item For asymptotically rigid handlebody groups $\H_{d,r}(O,Y)$ because of \cite[Theorem 3.9]{me}, as long as either $r\geq 3$ or $d\geq 2$.
\end{itemize}

\end{proof}

In the particular case where $\B$ is the asymptotically rigid mapping class group of a tree surface, we use \cite[Theorem 9.1]{Asymp24} and \cite[Theorem 5.1]{SurfHoughton} to determine that the descending link $\dl(v)$ of a vertex $v=[\Sigma,f]$ with $h(v)\geq r$ is $m$-connected for $$m= \left\lfloor\min\left\{ \frac{g-3}{2}, \frac{|A|+1}{2d-1}-2, |A|-2\right\}\right\rfloor,$$ with $|A|$ the number of boundary components of $\Sigma$, and $g$ its genus. By computing the minimum $h(v)$ with the given equalities, we obtain the following results, which we will use in Sections \ref{Sect4} and \ref{Sect5}.

\begin{cor}\label{connected}
Let $O$ and $Y$ be tori. The complex $\Xc^{6}_{2,1}(O,Y)$ is simply connected.
\end{cor}

\begin{cor}\label{connected0}
Let $O$ and $Y$ be tori. For any $r\geq 3$, the complex $\Xc^{3}_{1,r}(O,Y)$ is simply connected.
\end{cor}

\subsection{Constructing a fundamental domain}\label{2.2}

Let $\B$ be one of the groups in Theorem \ref{Main1}, and $\Xc^{h_0}$ the corresponding simply connected complex from Lemma \ref{LemDeslink}. In order to construct a simplicial complex and a fundamental domain satisfying the hypotheses of Theorem \ref{teoBrown} for the action of $\B$ on $\Xc^{h_0}$, we modify the complex $\Xc^{h_0}$ as follows: 

Take the $2$-skeleton of $\Xc^{h_0}$, which is a square complex, and subdivide each of its squares into two simplices by connecting the vertices of highest and lowest height by an edge. This gives rise to a simplicial complex $\K$, which is a subcomplex of the Stein simplicial complex mentioned in \cite[Subsection 5.1]{Asymp24}. The Stein simplicial complex is itself a subdivision of the Stein-Farley cube complex. Observe that $\K$ is simply connected since $\Xc^{h_0}$ is. We dedicate the rest of the subsection to finding a fundamental domain $W_{\K}$ for the action of $\B$ on $\K$, i.e. a subcomplex of $\K$ satisfying conditions (2) and (3) in Theorem \ref{teoBrown}.

\begin{rem}
An alternative would be to consider the 2-skeleton of the Stein simplicial complex (up to height $h_0$), rather than the complex $\K$, but the resulting fundamental domain would have more edges. This would translate into a different, non universal family of graphs in Theorem \ref{Main1}, and more relations to consider when computing the presentations in Sections \ref{Sect4} and \ref{Sect5}. To avoid these complications, we work with the complex $\K$ instead.
\end{rem}

As a first step, we have:

\begin{lem}\label{trans}
Let $\B$ be one of the groups in Theorem \ref{Main1}. The action of $\B$ on $\K$ is transitive on $n$-simplices with vertices of the same height.
\end{lem}

\begin{proof}[\sc Proof of Lemma {\rm \ref{trans}}]
Let $[\Sigma_1,\phi_1]$, and $[\Sigma_2, \phi_2]$ be two vertices of the same height. The suited submanifolds $\Sigma_1$ and $\Sigma_2$ have the same number of pieces, hence there exists a homeomorphism sending $\Sigma_1$ to $\Sigma_2$, which can be extended to an asymptotically rigid mapping class $\psi\in \B$. The composition $\phi_2\circ\psi\circ\phi_1^{-1}$ sends $[\Sigma_1,\phi_1]$ to $[\Sigma_2, \phi_2]$, which means that $\B$ acts transitively on vertices $v\in\K$ of the same height. The same argument proves that $\B$ acts transitively on edges between vertices of the same height, and on $2$-simplices with vertices of the same height. 
\end{proof}

We are in a position to find a fundamental domain $W_{\K}$ for the action of $\B$ on $\K$. We say that a simplicial complex $K$ is \textit{flag} if every set of vertices pairwise connected by edges span a simplex. A graph is a \textit{tongued snake} if it is isomorphic to the graph in Figure \ref{tSnake}, for any finite number of vertices. Given a group $\B$ in Theorem \ref{Main1}, let $h_0$ be the bound associated to it in Theorem $\ref{LemDeslink}$

\begin{figure}[h]
\begin{center}
\begin{tikzpicture}

  \filldraw (0,0) circle (2pt) ;
  \filldraw (1,0.5) circle (2pt) ;
  \filldraw (0,1) circle (2pt) ;
  \filldraw (1,1.5) circle (2pt);

  \filldraw (0,2) circle (2pt);
  \filldraw (1,2.5) circle (2pt);
  \filldraw (0,3) circle (2pt) ;
  \filldraw (1,3.5) circle (2pt) ;

  \draw (0,0) --  (1,0.5);
  \draw (0,0) --  (0,1);
  \draw[dashed] (1,0.5) --  (1,1);
  \draw (1,0.5) --  (0,1);
  \draw (1,1.5) --  (0,1);
  \draw (1,1.5) --  (1,0.5);
  \draw[dashed] (0,1) --  (0,1.5);
  \draw[dashed] (1,1.5) --  (0.7,1.65);
  \draw[dashed] (0,2) --  (0.3,1.85);
  
  \draw (1,2.5) --  (1,3.5);
  \draw (1,2.5) --  (0,3);
  \draw (0,3) --  (1,3.5);
  \draw (0,2) --  (0,3);
  \draw (0,2) --  (1,2.5);
  \draw[dashed] (1,2) --  (1,2.5);
  \filldraw (0,-0.75) circle (2pt);

  \draw (0,0) --  (0,-0.75);

\end{tikzpicture}

\end{center}
\caption{A tongued snake.}
\label{tSnake}
\end{figure}
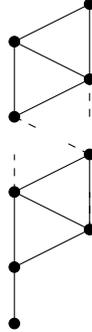

\begin{lem}\label{fudom}
Let $\B$ be one of the groups in Theorem \ref{Main1}. There exists a fundamental domain $W_\K$ for the action of $\B$ on $\K$, isomorphic to the flag complex with either a snake graph or a tongued snake graph as 1-skeleton, and with $h_0+1$ vertices. The fundamental domain $W_\K$ has one vertex of each height $0\leq h\leq h_0$, and the manifolds defining the vertices are obtained by adding two collections of pieces in an alternating fashion to $O^r$ (see Figure \ref{OPQR}).
\end{lem}

\begin{proof}[\sc Proof of Lemma {\rm \ref{fudom}}]

By Lemma \ref{trans} it suffices to choose a subcomplex $W_\K$ with one vertex of each height up to $h_0$, such that every edge and $2$-simplex of $\K$ is in $W_\K$ up to the action of $\B$.

Let $\B$ be a group in Theorem \ref{Main1}. We first consider the case where $O^r$ has $r>1$ adjacent pieces, i.e. $\B$ is different from the following cases:

\begin{enumerate}
    \item asymptotic mapping class groups of Cantor manifolds $\B_{d,1}(O,Y)$,
    \item asymptotically rigid mapping class groups of graphs $B_{d,1}(1, 1)$,
    \item asymptotically rigid handlebody groups $\H_{d,1}(O,Y)$.
\end{enumerate}

By Lemma \ref{trans}, we can choose the following vertices of $W_\K$:
\begin{itemize}
    \item $\v_0=[O^r,\id]$,
    \item $\v_1=[O^r\cup P_1,\id]$,
    \item $\v_2=[O^r\cup P_1 \cup Q_1,\id]$,
    \item $\v_3=[O^r\cup P_1 \cup Q_1 \cup P_2,\id],$
    \item[] ...
    \item $\v_{2k-1}=[O^r\cup P_1 \cup Q_1\cup...\cup P_{k},\id]$,
    \item $\v_{2k}=[O^r\cup P_1 \cup Q_1\cup...\cup P_{k}\cup Q_{k},\id],$
    \item[] ...
    \item $\v_{h_0},$
\end{itemize}

where each $P_i$ is adjacent to $P_{i-1}$ and $P_{i+1}$, each $Q_i$ is adjacent to $Q_{i-1}$ and $Q_{i+1}$, and $P_1$ and $Q_1$ are adjacent to $O^r$, as in Figure \ref{OPQR} left. This complex has a vertex of each height between 0 and $h_0$.

\begin{figure}[H]  
\centering
\includegraphics[width=8.5cm]{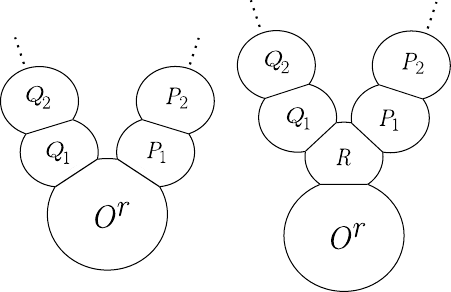}
\caption{On the left, the manifolds defining $W_\K$ when $r\geq1$. On the right, the manifolds defining $W_\K$ when $r=1$.}
\label{OPQR}
\end{figure}

Let $\v_h\in W_{\K}$ denote the unique vertex of height $0\leq h\leq h_0$. There is an edge connecting $\v_h$ and $\v_{h+1}$, since if $\v_h=[\Sigma,\id]$, then either $\v_{h+1}=[\Sigma\cup P_i,\id]$ or $\v_{h+1}=[\Sigma\cup Q_i,\id]$ for some $i$. There is also an edge connecting $\v_h$ and $\v_{h+2}$, since if $\v_h=[\Sigma,\id]$, then either $\v_{h+1}=[\Sigma\cup P_i \cup Q_i,\id]$ or $\v_{h+1}=[\Sigma\cup Q_i\cup P_{i+1},\id]$ for some $i$. Hence $\v_h$ has an edge connecting it to the vertices of heights $h-2$, $h-1$, $h+1$ and $h+2$ (whenever those exist). This means that the 1-skeleton of $W_{\K}$ will be a snake graph (see Figure \ref{Snake}), and we can obtain $W_{\K}$ by adding a $2$-simplex for every triangle.

On the other hand, orbits of edges in $W_{\K}$ by the action of $\B$ cover the 2-skeleton of $\K$, since edges of $\K$ only connect vertices $\v_i$ and $\v_j$ with $|h(\v_i)-h(\v_j)|\leq2$, and the action is transitive on edges between vertices of the same height. We deduce that $W_{\K}$ is a fundamental domain for the action of $\B$ on $\K$.

Now suppose $r=1$, and hence we are in one of the cases (1)-(3) above. We take the simplicial complex $W_\K$ spanned by the vertices: $$\v_{2k-1}=[O^r\cup R \cup P_1 \cup Q_1\cup...\cup Q_{k-1},\id], \quad\v_{2k}=[O^r\cup R \cup P_1 \cup Q_1\cup...\cup Q_{k-1} \cup P_{k},\id],$$ where $R$ es the only piece adjacent to $O^r$, both $P_1$ and $Q_1$ are adjacent to $R$ (see Figure \ref{OPQR} right), and $0\leq k\leq \lfloor \frac{h_0}{2} \rfloor$.

The above reasoning remains true except for $\v_0$, which is not adjacent to $\v_2$ in $W_\K$. However, there is no edge in $\K$ connecting two vertices $v,w\in\K$ of heights $h(v)=0$ and $h(w)=2$, since $O^r$ only has 1 adjacent piece. Hence if $r=1$, the orbits of edges in $W_{\K}$ cover the $2$-skeleton of $\K$ up to the action of $\B$.

We deduce that the 1-skeleton of $W_{\K}$ will be a tongued snake graph (see Figure \ref{tSnake}) and we obtain $W_{\K}$ by adding a $2$-simplex for every triangle.

\end{proof}

We now determine the vertex stabilizers for the action of $\B$ on $\K$.

\begin{lem}\label{finftystab}
Let $\B$ be one of the groups in Theorem \ref{Main1}. The vertex and edge stabilizers for the action of $\B$ on $\K$ are finite index subgroups of boundary permuting mapping class groups of suited manifolds. Particularly, vertex and edge stabilizers are of type $F_\infty$.
\end{lem}

\begin{proof}[\sc Proof of Lemma {\rm \ref{finftystab}}]
Vertex stabilizers of $\K$ are vertex stabilizers of $\Xc$. Edges of $\K$ are either edges of $\Xc$, or diagonals of squares of $\Xc$, hence edge stabilizers of $\K$ correspond to either edge or square stabilizers of $\Xc$. 

The fact that the $k$-cube stabilizers are of type $F_\infty$ follows from the results listed below:
\begin{itemize}

\item For asymptotically rigid mapping class group of Cantor manifolds because of \cite[Lemma 2.4, Subsection 2.1.2, Theorem 2.6]{Asymp24}.

\item For surface Houghton groups $\B_r$ because of \cite[Theorem 4.1]{SurfHoughton}.

\item For Houghton graph groups $B(g,h,r)$ and groups $B_{d,r}(1, 1)$ because of \cite[Theorem 5.11]{hill2025asymptoticallyrigidmappingclass}.

\item For asymptotically rigid handlebody groups $\H_{d,r}(O,Y)$ because of \cite[Corollary 2.8]{me}.

\end{itemize}

\end{proof}

We are now in a position to prove Theorem \ref{Main1}. 

\begin{proof}[\sc Proof of Theorem {\rm \ref{Main1}}]
Let $\B$ be one of the groups in Theorem \ref{Main1}, $\K$ its associated complex from Subsection \ref{2.2}, and $W_\K$ the fundamental domain for the action of $\B$ on $\K$ from Lemma \ref{fudom}. We will check that the action satisfies every condition in Theorem \ref{teoBrown}, which we will in turn use to prove Theorem \ref{Main1}.

Condition (1) holds because $\K$ is simply connected, as $\Xc^{h_0}$ is because of Lemma \ref{LemDeslink}. Conditions (2) and (3) follow from the fact that $W_\K$ is a fundamental domain for the action, see Lemma \ref{fudom}. Finally, Lemma \ref{finftystab} ensures that Conditions (4) and (5) are satisfied. Consequently, all hypotheses of Theorem \ref{teoBrown} are met, and each group $\B$ in Theorem \ref{Main1} is a quotient of the fundamental group of a graph of groups $\G_\B$, whose vertex and edge groups are finite extensions of mapping class groups.

According to Theorem \ref{teoBrown}, $\B$ is isomorphic to the fundamental group of a graph of groups $\G_\B$ defined over the 1–skeleton of $W_\K$. If $r>1$, the 1-skeleton of $W_\K$ is a snake graph, and Theorem \ref{Main1} is directly true. If $r=1$, the 1-skeleton of $W_\K$ is a tongued snake graph (Figure \ref{tSnake}), in this case we modify $\G_\B$. There is only one piece adjacent to $O^r$, which means that every element in the stabilizer of $\v_0$ also stabilizes the edge $(\v_1,\v_2)$, since it stabilizes the only adjacent piece. Hence $\G_\B$ has isomorphic fundamental group to the result of removing the vertex $\v_0$, which is a snake graph. 
\end{proof}

\section{Towards an explicit presentation for $\B$}

Let $B$ be any of the groups in Theorem \ref{Main1}, and let $W_\K$ be the fundamental domain for the action of $\B$ on $\K$ in Lemma \ref{fudom}. Because of Theorem \ref{teoBrown}, $\B$ is the quotient of the fundamental group of the graph of groups $\G_{\B}$ defined over the 1-skeleton of $W_{\K}$, with vertex stabilizers as vertex groups, and edge stabilizers as edge groups. Let $V_i$ be the stabilizer of $\v_i\in V(W_\K)$ and $E_{i,j}$ the stabilizer of $\{\v_i,\v_j\}$. Let $V_i=\langle X_i|R_i\rangle$, and let $X_{i,j}$ be a finite generating set of $E_{i,j}$. By computing the explicit quotient in Theorem \ref{teoBrown}, we deduce the following:

\begin{lem}\label{presentationB}
The group $\B$ has a finite presentation of the form: 
$$\B=\langle\cup X_i|(\cup R_i)\cup(\cup R_{i,j})\rangle,$$ where  $$R_{i,j}=\{o(s)t(s)^{-1}|s\in X_{i,j}\},$$ with each $o(s)$ an expression of $s$ as a word in $X_i$ and $t(s)$ an expression as a word in $X_j$.
\end{lem}

\begin{proof}[\sc Proof of Lemma {\rm \ref{presentationB}}]
Because of the definition of the fundamental group of a graph of groups (see Subsection \ref{1.5}), the fundamental group of the graph of groups $\G_{\B}$ is the quotient of $$(\ast_i\; \langle X_i|R_i\rangle)*(\ast_{e \in E(W_\K)}\;\langle y_e\rangle)$$ by the sets of relations:

\begin{enumerate}
\item $y_{\bar{e}} = y^{-1}_e$ for all $e\in E(W_\K)$, where if $e=(\v_i,\v_j)$, then $\bar{e}=(\v_j,\v_i)$,
\item $y_e=1$ for all $e\in T$,
\item $y_e\varphi_{e,0}(x)y_e^{-1}=\varphi_{e,1}(x)$ for all $e=(\v_i,\v_j)\in T$ and every $x\in E_{i,j}$.
\end{enumerate}

After quotienting by the extra relations $$y_e=1 \textit{ for all }e\in E(W_\K)$$ in  Theorem \ref{teoBrown}, the first two sets of relations become trivial, and the resulting group is the quotient of $(\ast_i\; \langle X_i|R_i\rangle)$ by the relations 

\begin{center}$\varphi_{e,0}(x)=\varphi_{e,1}(x)$ for all $e=\{\v_i,\v_j\}\in E(W_\K)$ and every $x\in E_{i,j}$.
\end{center}
These relations are precisely the sets $R_{i,j}$.

\end{proof}

Theorem \ref{presentationB} provides a method for computing an explicit presentation for the groups in Theorem \ref{Main1}, as long as we have presentations for the vertex stabilizers and generating sets for the edge stabilizers. This is what we will do in the tree surface case, where we will build on work of Labruère and Paris \cite{Labruere-Paris} in order to compute an explicit presentation for $\B$. To this end, in the next section we determine which subgroups of the boundary permuting mapping class group are the vertex and edge stabilizers.

\section{Vertex and edge stabilizers in the surface case}

Let $\B$ be the asymptotically rigid mapping class group of a tree surface and $\Sigma$ be a suited subsurface. Recall from Subsection \ref{mapo} that $\Map_o(\Sigma)$ denotes the boundary permuting mapping class group of $\Sigma$. For any $A_1,A_2\subset\partial \Sigma$ boundary components, let $\Map^{A_1}_o(\Sigma)$ denote the subgroup of $\Map_o(\Sigma)$ of elements fixing $A_1$, and $\Map^{\{A_1,A_2\}}_o(\Sigma)$ the subgroup of elements fixing $\{A_1,A_2\}$ as a set. For a simplex $\sigma\in \K$, we call $\B_\sigma$ the stabilizer of $\sigma$ by the action of $\B$. The next lemma goes along the lines of \cite[Lemma 6.3]{Asymp24}, and describes vertex and edge stabilizers for the action of $\B$ on $\K$ in the tree surface case.

\begin{lem}\label{typesstab}

Consider the action of $\B$ on $\K$:
    \begin{enumerate}
        \item  $\B_{[\Sigma,\phi]}\simeq\Map_o(\Sigma)$. 
        \item Let $P_1$ be a piece adjacent to $\Sigma$. Then $\B_{\{[\Sigma,\phi],[\Sigma\cup P_1,\phi]\}}\simeq\Map_o^{A_1}(\Sigma)$, for $A_1$ any boundary of $\Sigma$.
        \item Let $P_1,P_2$ be distinct pieces, both adjacent to $\Sigma$. Then $\B_{\{[\Sigma,\phi],[\Sigma\cup P_1\cup P_2,\phi]\}}\simeq\Map_o^{\{A_1,A_2\}}(\Sigma)$, for $A_1,A_2$ any couple of distinct boundaries of $\Sigma$.

    \end{enumerate}
\end{lem}

\begin{proof}[\sc Proof of Lemma {\rm \ref{typesstab}}]
We start by proving point (1). After multiplying by $\phi^{-1}$ we may suppose the vertex is $[\Sigma,\id]$. If $\psi\in\B_{[\Sigma,\id]}$ then it must satisfy $\psi(\Sigma)=\Sigma$, and must be rigid away from $\Sigma$. This can be used to define maps:
$$\Phi:\B_{[\Sigma,\id]}\rightarrow \Map_o(\Sigma),$$
\vspace{-1.3em}$$\Psi:\Map_o(\Sigma)\rightarrow\B_{[\Sigma,\id]},$$

where, for $\psi\in B_{[\Sigma,\id]} $, its image $\Phi(\psi)$ is defined by choosing a representative that is rigid away from $\Sigma$ and restricting it to $\Sigma$, which defines an element of $\Map_o(\Sigma)$; similarly, for $f\in \Map_o(\Sigma)$, the image $\Psi(f)$ is the unique rigid extension of any representative of $f$ to a homeomorphism of $\S_{d,r}(O, Y)$. We will first check that these maps are well defined, and then see that they are inverses of each other.

The map $\Psi$ is well defined, since representatives of $f\in \Map_o(\Sigma)$ are isotopic so are their extensions. The map $\Phi$ is well defined: if $\psi_1,\psi_2$ homeomorphisms of $\S_{d,r}(O, Y)$ are isotopic, and are rigid away from $\Sigma$, then restriction to any suited surface $Y$ containing the support of the isotopy defines the same element of $\Map_o(Y)$, and because of the \textit{inclusion property} of tree manifolds (see the proof of \cite[Proposition 3.6]{Asymp24}, which works for tree manifolds), restriction to $\Sigma$ defines the same element of $\Map_o(\Sigma)$. Hence the maps are well defined.

The maps satisfy $\Psi\circ\Phi=\id$, since taking a homeomorphism $\psi$ of $\S_{d,r}(O, Y)$ that is rigid away from $\Sigma$, restricting it to $\Sigma$ and then extending it, gives back $\psi$. Analogously, $\Phi\circ\Psi=\id$. Hence the maps establish an isomorphism, since composition is preserved, and we conclude that $\B_{[\Sigma,\phi]}\simeq\Map_o(\Sigma)$.

To compute the stabilizer of the edge $\{[\Sigma,\id],[\Sigma\cup P_1,\id]\}$, we remark that $$\B_{\{[\Sigma,\id],[\Sigma\cup P_1,\id]\}}\leq\B_{[\Sigma,\id]}. $$ 

Hence we simply have to study the image of $\B_{\{[\Sigma,\id],[\Sigma\cup P_1,\id]\}}$ under $\Phi$, which corresponds to the image of elements $\psi$ that fix $P_1$. The image $\Phi(\psi)$ is an element of $\Map^{A_1}_o(\Sigma)$ for $A_1$ the boundary component that is in $P_1$, and any element of $\Map^{A_1}_o(\Sigma)$ is extended to an element in the stabilizer, hence $\B_{\{[\Sigma,\id],[\Sigma\cup P_1,\id]\}}\simeq\Map_o^{A_1}(\Sigma)$. 

Analogously it can be proved that $\B_{\{[\Sigma,\id],[\Sigma\cup P_1\cup P_2,\id]\}}\simeq\Map_o^{\{A_1,A_2\}}(\Sigma)$.

\end{proof}

Given a presentation for boundary permuting mapping class groups, Theorems \ref{presentationB} and \ref{typesstab} provide a recipe for determining a presentation of $\B_{d,r}(O,Y)$. To this end, the next subsection is dedicated to calculating a presentation of the boundary permuting mapping class group in the tree surface case.

\section{Presentation of the boundary permuting mapping class group of surfaces}\label{Sect2}

In \cite{Labruere-Paris}, Labruère and Paris find a presentation of the mapping class group of a finite type surface, possibly with boundary. Building on this work, in this section we will find a presentation of the boundary permuting mapping class group of a surface, and a set of generators of some of its subgroups.

We first give some definitions. By a \textit{curve} on a surface $\Sigma$ we mean the isotopy class of a closed simple curve on $\Sigma$, and by an \textit{arc} we mean the isotopy class of an simple arc that connects either two boundary components of $\Sigma$, or two punctures of $\Sigma$. We will blur the difference between isotopy classes and their representatives.

Let $S_{g,b,p}$ be a surface of genus $g$, with $p$ punctures and $b$ boundary components. For the rest of the paper, we adopt the notation: $$S=S_{g,n,0} \quad \textnormal{  and  } \quad S'=S_{g,0,n}$$ for $g\geq1$ and $n\geq 0$. 

By selecting $n$ curves each bounding a punctured disc $D_i\subset S'$, we determine an inclusion $S\hookrightarrow S'$ where the image of each of the $n$ boundary components of $S$ is one of the chosen curves. In other words, the surface $S'$ is the result of \textit{capping} every boundary component of $S$ with a punctured disc. This inclusion in turn induces a surjective homomorphism $$\In:\Map_o(S)\rightarrow \Map(S')$$ by, after permuting the boundary components, extending each homeomorphism as the identity in each $D_i$. More precisely, fix a set of parametrizations $\xi_i:\DD^2\rightarrow D_i$. Each map $\phi\in\Map_o(S)$, induces a permutation $\sigma=\sigma(\phi)$ of the $n$ boundary components of $S$. We define $$\In(\phi):=\begin{cases}
  \phi & \text{in } S, \\
 \xi_{\sigma(i)}\circ\xi_i^{-1}& \text{in each }D_i,
\end{cases}$$

The notation $\In(\cdot)$ reflects the fact that the map is the result of iterating the capping homeomorphism introduced in \cite[Proposition 3.19]{Farb-Margalit}, adapted to the boundary‑permuting setting.

Let $U_{S}$ be the subgroup of $\Map_o(S)$ generated by the Dehn twists $\{u_1,...,u_n\}$ in Figure \ref{Generators}, where each $u_i$ is the Dehn twist along $\partial D_i$. The group $U_{S}$ is abelian and has a presentation $$\langle u_1,u_2,...,u_n|R_{U}:=\{u_iu_j=u_ju_i \; \forall i,j\}\rangle.$$ The following lemma relates the groups $\Map_o(S)$ and $\Map(S')$ through a short exact sequence.

\begin{lem}\label{SecMap}

There is an exact sequence: $$1\rightarrow U_{S}\rightarrow\Map_o(S)\overset{\In}{\rightarrow}\Map(S')\rightarrow1.$$

\end{lem}

\begin{proof}[\sc Proof of Lemma {\rm \ref{SecMap}}]
We verify that the kernel of $\In$ is precisely $U_{S}$. The inclusion $U_{S}\subset \ker(\In)$ is direct. For the inclusion $\ker(\In)\subset U_{S}$, we observe that elements in the kernel do not permute boundary components, and hence belong to $\Map(S)$. Because of \cite[Theorem 3.18]{Farb-Margalit}, any element in $\ker(\In)\cap \Map(S)$ is in $U_{S}$.

\end{proof}

We will use this short exact sequence in order to compute a presentation of the boundary permuting mapping class group. The text subsection introduces one of the various tools needed for this.

\subsection{Exact sequences and presentations}
Through their work, Labruère and Paris use a well-known method for crafting a presentation from a short exact sequence (see \cite[Subsection 2.2]{Labruere-Paris}). 

Let $1\rightarrow K \hookrightarrow G\overset{\pi}{\rightarrow}  H\rightarrow 1$ be a short exact sequence of groups. Given presentations $H=\langle X_H|R_H\rangle$ and $K=\langle X_K|R_K\rangle$,  for each $x\in X_H$, we choose $\tilde{x}\in \pi^{-1}(x)$, and write $$ \tilde{X}_H:=\{\tilde{x}:x\in X_H\}.$$ Given any relation $r=x_1^{\varepsilon_1}...\,x_l^{\varepsilon_l}\in R_H$, if we substitute each $x_i$ for $\tilde{x}_i\in\tilde{X}_H$ we get an element $\tilde{r}=\tilde{x_1}^{\varepsilon_1}...\,\tilde{x_l}^{\varepsilon_l}\in G$, which is in the kernel of $\pi$ as $\pi(\tilde{r})=r$ is a relation of $H$. Treat the group $K$ as a subgroup of $G$. By the exactness of the sequence, one can choose a word $w_r$ with letters in $X_K\subset G$ that represents the same element as $\tilde{r}$. Set $$R_1:=\{\tilde{r}w_r^{-1}:r\in R_H\}.$$ Since $K$ is a normal subgroup of $G$, for each $\tilde{x}\in \tilde{X}_H$ and $y\in X_K$ the conjugate $\tilde{x}y\tilde{x}^{-1}$ is also in $K$ and hence one may choose a word $v(x,y)$ with letters in $X_K$ representing the same element of $G$. Set $$R_2:=\{\tilde{x}y\tilde{x}^{-1}v(x,y)^{-1}:\tilde{x}\in \tilde{X}_H \textnormal{ and } y\in X_K\}.$$

\begin{lem}\label{Lema-present}\cite[Lemma 2.5]{Labruere-Paris}
$G$ admits a presentation $$G=\langle \tilde{X}_H\cup X_K|R_1\cup R_2 \cup R_K\rangle.$$

\end{lem}

\subsection{Generators of the boundary permuting mapping class group}\label{Gens-bpmcg}

We can compute a set of generators of $\Map_o(S)$ from the exact sequence of Lemma \ref{SecMap} using Lemma \ref{Lema-present}. In this vein, $\Map_o(S)$ will be generated by the Dehn twists $\{u_1,...,u_n\}$, and preimages through $\In(\cdot)$ of each of the generators of $\Map(S')$. 

According to \cite[Corollary 2.11]{Labruere-Paris}, $\Map(S')$ is generated by the Dehn twists \\ $\{x_0,x_1,z,y_1,...,y_{2g-1}\}$, together with the half twists $\{h_1,...,h_{n-1}\}$ in Figure \ref{gen-in}. To find preimages through $\In(\cdot)$ of each $h_i$, we need to introduce \textit{boundary swaps}, which are certain preimages of half twists under $\In(\cdot)$, as we now describe.

\begin{figure}[H]  
\centering
\includegraphics[width=10.5cm]{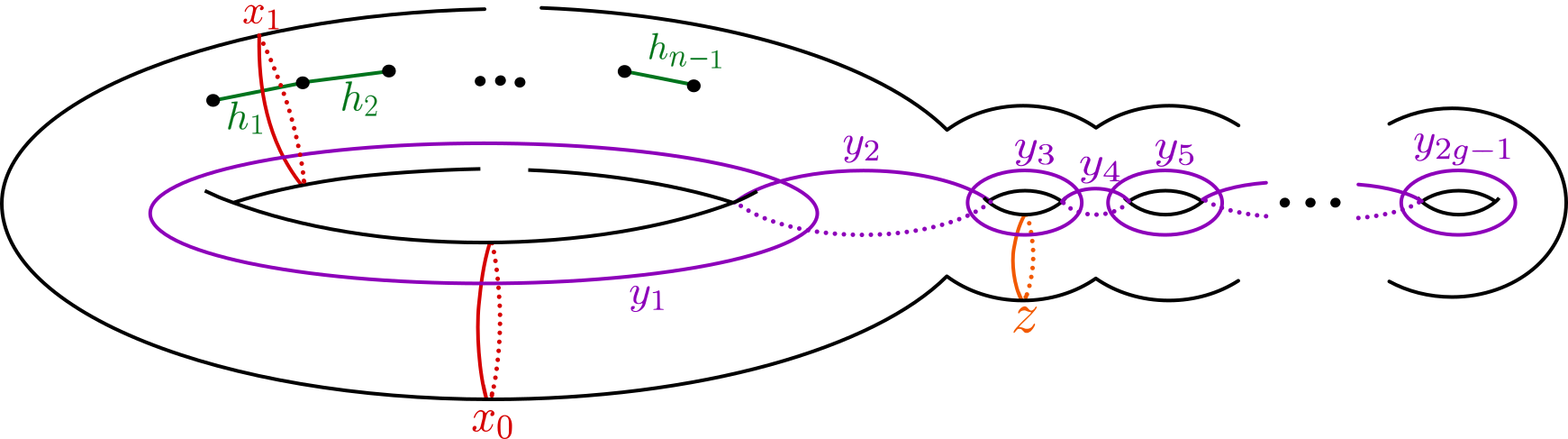}
\caption{}
\label{gen-in}
\end{figure}

Let $\a$ be any arc between two boundary components $A_1,A_2\subset S$. Let $N_\a$ be a closed, regular neighborhood of $\a\cup A_1 \cup A_2$. Any homeomorphism $\psi$ of $N_\a$ fixing the boundary component $C=\partial N_\a$ (see Figure \ref{twistex}) can be extended to an element of $\psi\in\Map_o(S)$ by the identity. We say that $\psi$ is a \textit{boundary swap} along $\a$ if $\In(\psi)$ is a half twist. The following result describes the information that determines a boundary swap.

\begin{figure}[H]  
\centering
\includegraphics[width=9.5cm]{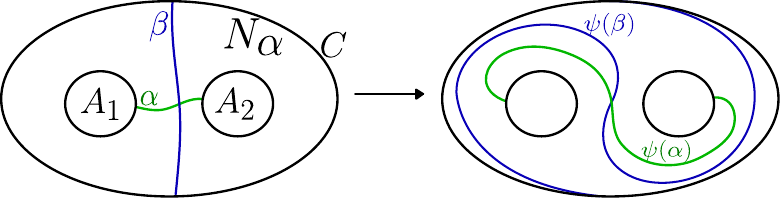}
\caption{An example of a boundary swap.}
\label{twistex}
\end{figure}

\begin{lem}\label{lemswap0}
    A boundary swap $\psi$ along $\a$ is uniquely determined by $\In(\psi)$ and $\psi(\a)$.
\end{lem}

\begin{proof}[\sc Proof of Lemma {\rm \ref{lemswap0}}]
Suppose there are two boundary swaps $\psi, \;\psi'$ such that $\In(\psi)=\In(\psi')$ and $\psi(\a)=\psi'(\a)$. Because $\psi(\a)=\psi'(\a)$, the composition $\psi^{-1}\circ\psi'$ fixes $\a$. The mapping class $\psi^{-1}\circ\psi'$ also fixes the arc $\b$, since the kernel of $\In(\cdot)$ fixes $\b$ and $\In(\psi^{-1}\circ\psi')=\In(\psi^{-1})\circ\In(\psi')=\id$. Hence because of the Alexander method (see \cite[Proposition 2.8]{Farb-Margalit}), $\psi^{-1}\circ\psi'$ is the identity, since it fixes $\a$ and $\b$.

\end{proof}

\begin{figure}[H]
\includegraphics[width=13.5cm]{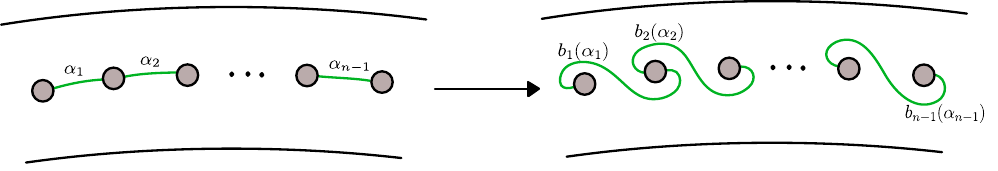}
\centering
\caption{}
\label{election}
\end{figure}

Lemma \ref{lemswap0} implies that a preimage through $\In(\cdot)$ of one of the half-twists $h_i$ is uniquely determined by the image $\In(\a_i)$ of the arc $\a_i$ in Figure \ref{election}. We define $b_i$ to be the only preimage of $h_i$ that sends $\a_i$ to the arc $b_i(\a_i)$ in Figure \ref{election}.

\begin{theor}\label{teogenSPMCG}
The boundary swaps $\{b_1,...,b_{n-1}\}$, together with the Dehn twists \\ $\{x_0,x_1,z,y_1,...,y_{2g-1},u_1,...,u_n\}$ in Figure \ref{Generators}, generate $\Map_o(S)$.
\end{theor}

\begin{proof}[\sc Proof of Theorem {\rm \ref{teogenSPMCG}}]
Lemma \ref{Lema-present} implies that $\Map_o(S)$ is generated by the Dehn twists $\{u_1,...,u_n\}$, together with a set of preimages of the generators $$\{x_0,x_1,z,y_1,...,y_{2g-1}, h_1,...,h_{n-1}\}$$ of $\Map(S')$ through $\In(\cdot)$.  For each Dehn twist $T_{\c}\in\{x_0,x_1,z,y_1,...,y_{2g-1}\}$ we choose $T_{\c'}$ as a preimage, for $\c'\subset S'$ any curve that the inclusion $S\hookrightarrow S'$ maps to $\c$. We call these preimages $\{x_0,x_1,z,y_1,...,y_{2g-1}\}$, abusing notation. For each $h_i$, we choose $b_i$ as a preimage. Hence $\Map_o(S)$ is generated by the mapping classes in Figure \ref{Generators}.

\end{proof}

\begin{figure}[H]
\includegraphics[width=14.7cm]{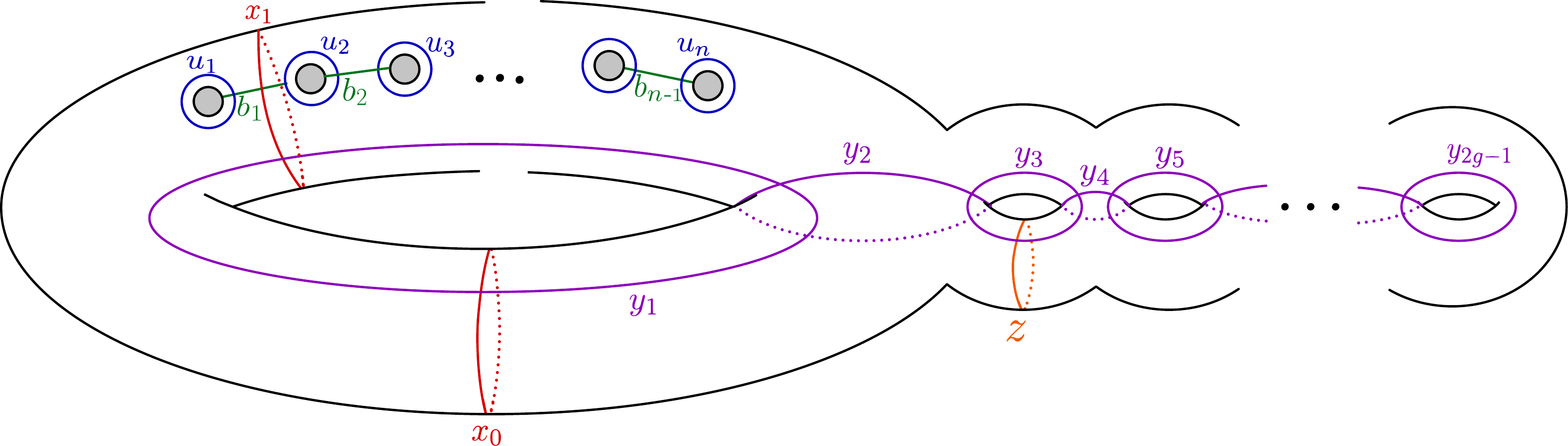}
\centering
\caption{A set of generators of $\Map_o(S)$.}
\label{Generators}
\end{figure}

As mentioned in Section \ref{Sect1}, in order to compute a presentation for $\B_{d,r}(O, Y)$, we also need a set of generators for the edge stabilizers. Particularly, we will need generators for the stabilizers of edges of the form $\{[\Sigma,\id],[\Sigma\cup P_1,\id]\}$ which, according to Lemma \ref{typesstab}, are the groups $\Map_o^{A_1}(S)$. In Sections \ref{Sect4} and \ref{Sect5} we will use these presentations to determine the edge groups $\B_{\{[\Sigma,\id],[\Sigma\cup P_1,\id]\}}$. The piece $P_1$ can attach to $\Sigma$ in two different ways, as will be made precise in Sections \ref{Sect4} and \ref{Sect5}. To account for these differences, we introduce two alternative generating sets that will be more convenient in different parts of our arguments.

Following the same strategy as for calculating generators of $\Map_o(S)$, we can deduce, from \cite[Proposition 2.10]{Labruere-Paris}, the following generating set:

\begin{lem}\label{lemSPMCG-A1}
The boundary swaps $\{b_1,...,b_{n-1}\}$, together with the Dehn twists \\ $\{x_0,x_1,z,y_1,...,y_{2g-1},u_1,...,u_n\}$ in Figure \ref{Gen-1}, generate $\Map_o^{A_1}(S_{g,n+1,0})$.
\end{lem}

\begin{figure}[H]
\includegraphics[width=11.7cm]{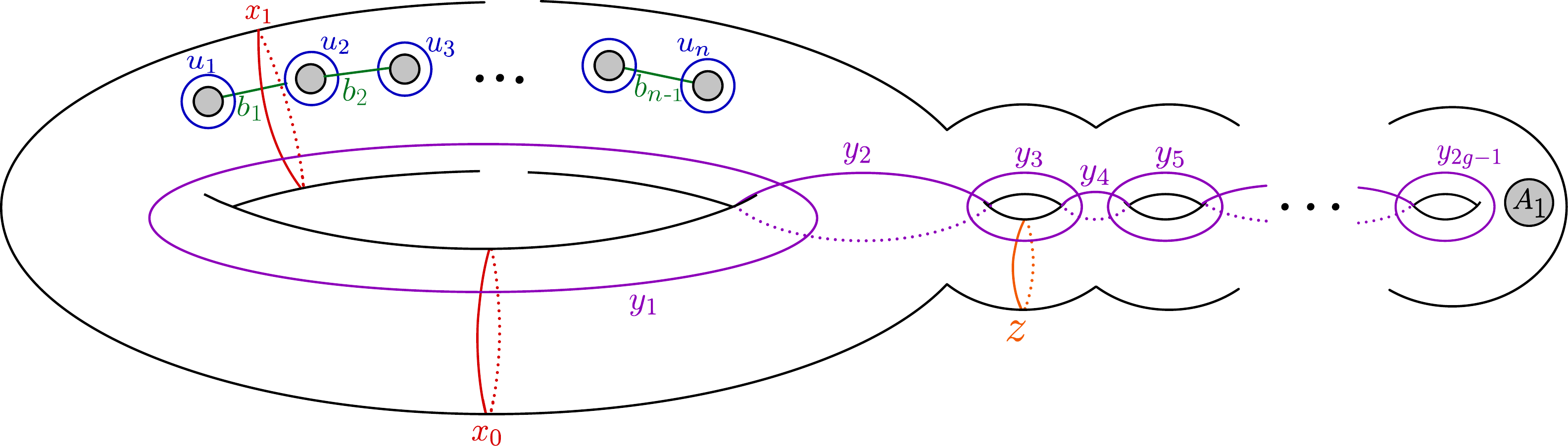}
\centering
\caption{A set of generators of $\Map_o^{A_1}(S_{g,n+1,0})$.}
\label{Gen-1}
\end{figure}

The following generating set of $\Map(S_{g,n+1,0})$ can be deduced from \cite[Proposition 2.10]{Labruere-Paris}: an argument can be made by first capping by a disc one of the boundary components, and then following an analogous reasoning as for deducing Theorem \ref{teogenSPMCG}.

\begin{lem}\label{lemSPMCG-A2}
The boundary swaps $\{b_1,...,b_{n-1}\}$, together with the Dehn twists \\ $\{w_0,w_1,w_2,z,y_1,...,y_{2g-1},u_1,...,u_n\}$ in Figure \ref{Gen-2}, generate $\Map_o^{A_1}(S_{g,n+1,0})$.
\end{lem}

\begin{figure}[H]
\includegraphics[width=11.7cm]{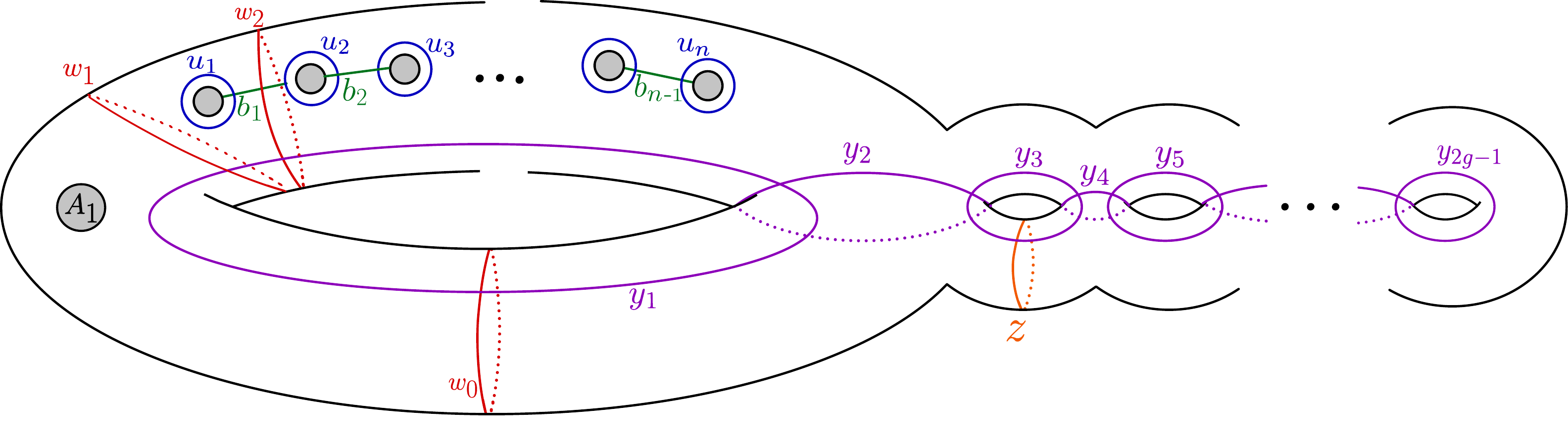}
\centering
\caption{A set of generators of $\Map_o^{A_1}(S_{g,n+1,0})$.}
\label{Gen-2}
\end{figure}

Notice that although the generating sets appear similar, there is a key difference: in Figure \ref{Gen-1}, the boundary component $A_1$ lies in the genus 0 component of $S_{g,n+1,0}\setminus \{x_1,x_0\}$, while in Figure \ref{Gen-2} it lies in the positive genus component of $S_{g,n+1,0}\setminus \{w_2,w_0\}$.

 \subsection{Artin groups and Artin group homomorphisms}\label{2.3} 
The work of Labruère and Paris (see Theorem \cite[Theorem 3.2]{Labruere-Paris}) expresses $\Map(S')$ as the quotient of an Artin group, where the extra relations from the quotient are expressed in terms of elements of Artin groups. This subsection is dedicated to introducing the terminology and tools about Artin groups we need for our presentation of $\Map_o(S)$.
 
An Artin group is given by the following presentation: $$
\left\{ x_1,...,x_n \,\middle|\,
\underbrace{x_ix_jx_i\ldots}_{m_{i,j}\ \text{times}} = \underbrace{x_jx_ix_j\ldots}_{m_{i,j}\ \text{times}}\,
\quad m_{i,j} \in \NN_{\geq 2}\right\}.
$$

An Artin group uniquely defines a labeled graph
$\Gamma$ with vertices $V(\Gamma):=\{x_1,...,x_n\}$, edges $E(\Gamma):=\{[x_{i},x_{j}]\}_{m_{i,j}\geq 3}$, and a labeling $$l:E(\Gamma)\rightarrow \NN_{\geq 3},$$ 
\vspace{-20pt}
$$l([x_{i},x_{j}])=m_{i,j}.$$

Reciprocally, any such graph $\Gamma$ uniquely defines an Artin group $A(\Gamma)$. 

Let $X_\Gamma:=\{x_1,...,x_n\}$ be the generators of $A(\Gamma)$ in the presentation above, and $R_\Gamma$ the set of relations. The \textit{quasi-center}  of $A(\Gamma)$ is the subgroup of elements $a\in A(\Gamma)$ satisfying $a X_\Gamma a^{-1}=X_\Gamma$. The quasi-center is a cyclic group generated by the \textit{fundamental element}, denoted by $\Delta(\Gamma)$ (see \cite[Section 5]{Brieskorn-Saito} and \cite[Notation 1.13]{Deligne}). Certain equalities expressing powers of the fundamental element in terms of the generators of the Artin group are necessary for understanding the relations in Theorem \ref{Teo-L-P}, we present them in Lemma \ref{Fundamental}.

The presentation of the mapping class group of a punctured surface that Labruère and Paris give in Theorem \cite[Theorem 3.2]{Labruere-Paris} is expressed in terms of powers of the fundamental element of \textit{parabolic subgroups}, which are subgroups generated by a subset of the $x_i\in X_\Gamma$. The parabolic subgroups in Theorem \ref{Teo-L-P} are defined by graphs in the families in Figure \ref{Graphs}.

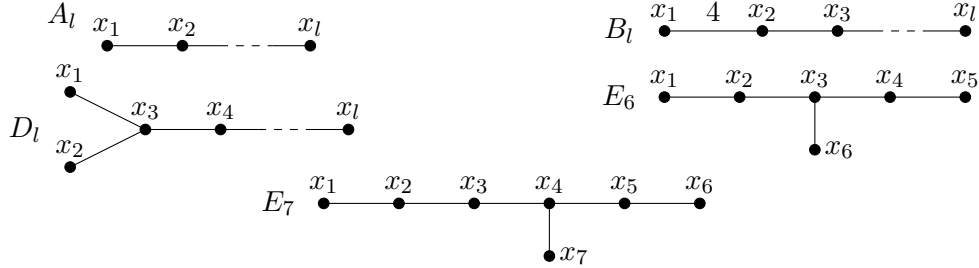
\begin{figure}[H]

\begin{minipage}{0.45\textwidth}
\centering
\begin{tikzpicture}
  \node at (-0.6, 0.4) {$A_l$};
  \filldraw (0,0) circle (2pt) node[above] {$x_1$};
  \filldraw (1,0) circle (2pt) node[above] {$x_2$};
  \filldraw (2.7,0) circle (2pt) node[above] {$x_l$};

  \draw (0,0) -- (1,0);
  \draw (1,0) -- (1.5,0);
  \draw[dashed] (1.5,0) -- (2.2,0);
  \draw (2.2,0) -- (2.7,0);
\end{tikzpicture}
\end{minipage}
\hfill
\begin{minipage}{0.45\textwidth}
\centering
\begin{tikzpicture}
  \node at (-0.6, 0) {$B_l$};
  \filldraw (0,0) circle (2pt) node[above] {$x_1$};
  \filldraw (1.3,0) circle (2pt) node[above] {$x_2$};
  \filldraw (2.3,0) circle (2pt) node[above] {$x_3$};
  \filldraw (4,0) circle (2pt) node[above] {$x_l$};

  \draw (0,0) -- node[above]{4} (1.3,0);
  \draw (1.3,0) -- (2.3,0);
  \draw (2.3,0) -- (2.8,0);
  \draw[dashed] (2.8,0) -- (3.5,0);
  \draw (3.5,0) -- (4,0);
\end{tikzpicture}
\end{minipage}

\begin{minipage}{0.45\textwidth}
\centering
\begin{tikzpicture}
  \node at (-0.6, 0) {$D_l$};
  \filldraw (0,0.5) circle (2pt) node[above] {$x_1$};
  \filldraw (0,-0.5) circle (2pt) node[above] {$x_2$};
  \filldraw (1,0) circle (2pt) node[above] {$x_3$};
  \filldraw (2,0) circle (2pt) node[above] {$x_4$};
  \filldraw (3.7,0) circle (2pt) node[above] {$x_l$};

  \draw (0,0.5) -- (1,0);
  \draw (0,-0.5) -- (1,0);
  \draw (1,0) -- (2,0);
  \draw (2,0) -- (2.5,0);
  \draw[dashed] (2.5,0) -- (3.2,0);
  \draw (3.2,0) -- (3.7,0);
\end{tikzpicture}
\end{minipage}
\hfill
\begin{minipage}{0.45\textwidth}
\centering
\begin{tikzpicture}
  \node at (-0.6, 0) {$E_6$};
  \filldraw (0,0) circle (2pt) node[above] {$x_1$};
  \filldraw (1,0) circle (2pt) node[above] {$x_2$};
  \filldraw (2,0) circle (2pt) node[above] {$x_3$};
  \filldraw (3,0) circle (2pt) node[above] {$x_4$};
  \filldraw (4,0) circle (2pt) node[above] {$x_5$};
  \filldraw (2,-0.7) circle (2pt) node[right] {$x_6$};

  \draw (0,0) -- (1,0);
  \draw (1,0) -- (2,0);
  \draw (2,0) -- (3,0);
  \draw (3,0) -- (4,0);
  \draw (2,0) -- (2,-0.7);
  
\end{tikzpicture}
\end{minipage}

\vspace{-0.3em}

\begin{minipage}{1\textwidth}
\centering
\begin{tikzpicture}
  \node at (-0.6, 0) {$E_7$};
  \filldraw (0,0) circle (2pt) node[above] {$x_1$};
  \filldraw (1,0) circle (2pt) node[above] {$x_2$};
  \filldraw (2,0) circle (2pt) node[above] {$x_3$};
  \filldraw (3,0) circle (2pt) node[above] {$x_4$};
  \filldraw (4,0) circle (2pt) node[above] {$x_5$};
  \filldraw (5,0) circle (2pt) node[above] {$x_6$};
  \filldraw (3,-0.7) circle (2pt) node[right] {$x_7$};

  \draw (0,0) -- (1,0);
  \draw (1,0) -- (2,0);
  \draw (2,0) -- (3,0);
  \draw (3,0) -- (4,0);
  \draw (4,0) -- (5,0);
  \draw (3,0) -- (3,-0.7);
\end{tikzpicture}
\end{minipage}

\caption{Families of graphs.}
\label{Graphs}
\end{figure}

The following expressions are due to Brieskorn and Saito \cite[Lemma 5.8]{Brieskorn-Saito}, we use notation from Labruère and Paris \cite[Proposition 2.8]{Labruere-Paris}.

\begin{lem}\label{Fundamental} The following equalities hold.

\begin{minipage}{0.50\textwidth}
\centering

\begin{itemize}
    \item $\Delta^2(A_{l})=(x_1x_2\dots x_l)^{l+1},$
    \item $\Delta(B_{l})=(x_1x_2\dots x_l)^{l},$
    \item $\Delta(D_{2p})=(x_1x_2\dots x_{2p})^{2p-1},$
\end{itemize}

\end{minipage}
\hfill
\begin{minipage}{0.50\textwidth}
\centering
\begin{itemize}
    \item $\Delta^2(D_{2p+1})=(x_1x_2\dots x_{2p+1})^{4p},$
    \item $\Delta^2(E_6)=(x_1x_2\dots x_6)^{12},$
    \item $\Delta(E_7)=(x_1x_2\dots x_7)^{15}.$
\end{itemize}

\end{minipage}
    
\end{lem}

The following Lemma presents a set of well known relations involving Dehn twists and half twists that can be used to define a homomorphism from certain Artin groups to the mapping class groups of a surface. See \cite[Lemmas 2.1, 2.2, 2.3]{Labruere-Paris}.

\begin{lem}\label{relDehnHalf} Let $x$ and $x'$ be Dehn twists along curves $\c$ and $\c'$. Let $h$ and $h'$ be half twists along arcs $\a$ and $\a'$.

\begin{itemize}
    \item If $\c$ and $\c'$ intersect once then $xx'x=x'xx'$. If $\c$ and $\c'$ are disjoint then $xx'=x'x$.
    \item If $\a$ and $\a'$ only intersect at a common endpoint then $hh'h=h'hh'$. If $\a$ and $\a'$ are disjoint then $hh'=h'h$.
    \item If $\a$ and $\c$ intersect once then $xhxh=hxhx$. If $\a$ and $\c$ are disjoint then $hx=xh$.
\end{itemize}
    
\end{lem}

These relations will be used to find homomorphisms $\rho:A(\Gamma)\rightarrow \Map(S')$ for the Artin groups defined by the families of graphs in Figure \ref{Graphs}. The images of the expressions in Lemma \ref{Fundamental} are written as a product of generators of $\Map(S')$, which can be used for finding relations of $\Map(S')$. This is used in the proof of Theorem \ref{Teo-L-P} in \cite{Labruere-Paris}, and we will use it in proving Theorem \ref{ThMain}.

We will abuse notation by identifying each generator $x_i$ with its image under the homomorphism $\rho(\cdot)$. The following Lemma is a direct consequence of \cite[Proposition 2.12]{Labruere-Paris}, \cite[Subsection 9.2]{Farb-Margalit}, and Lemma \ref{relDehnHalf}.

\begin{figure}[htbp]
\centering

\begin{subfigure}{0.48\textwidth}
    \centering
    \includegraphics[width=\linewidth]{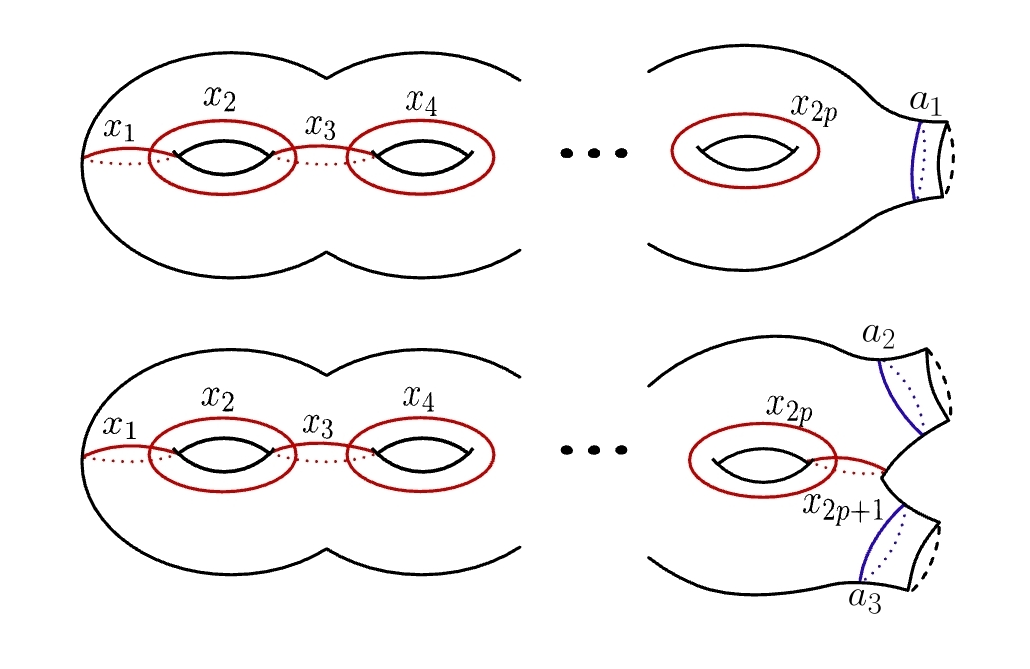}
    \caption{}
    \label{fig:seis-A}
\end{subfigure}
\hfill
\begin{subfigure}{0.48\textwidth}
    \centering
    \includegraphics[width=\linewidth]{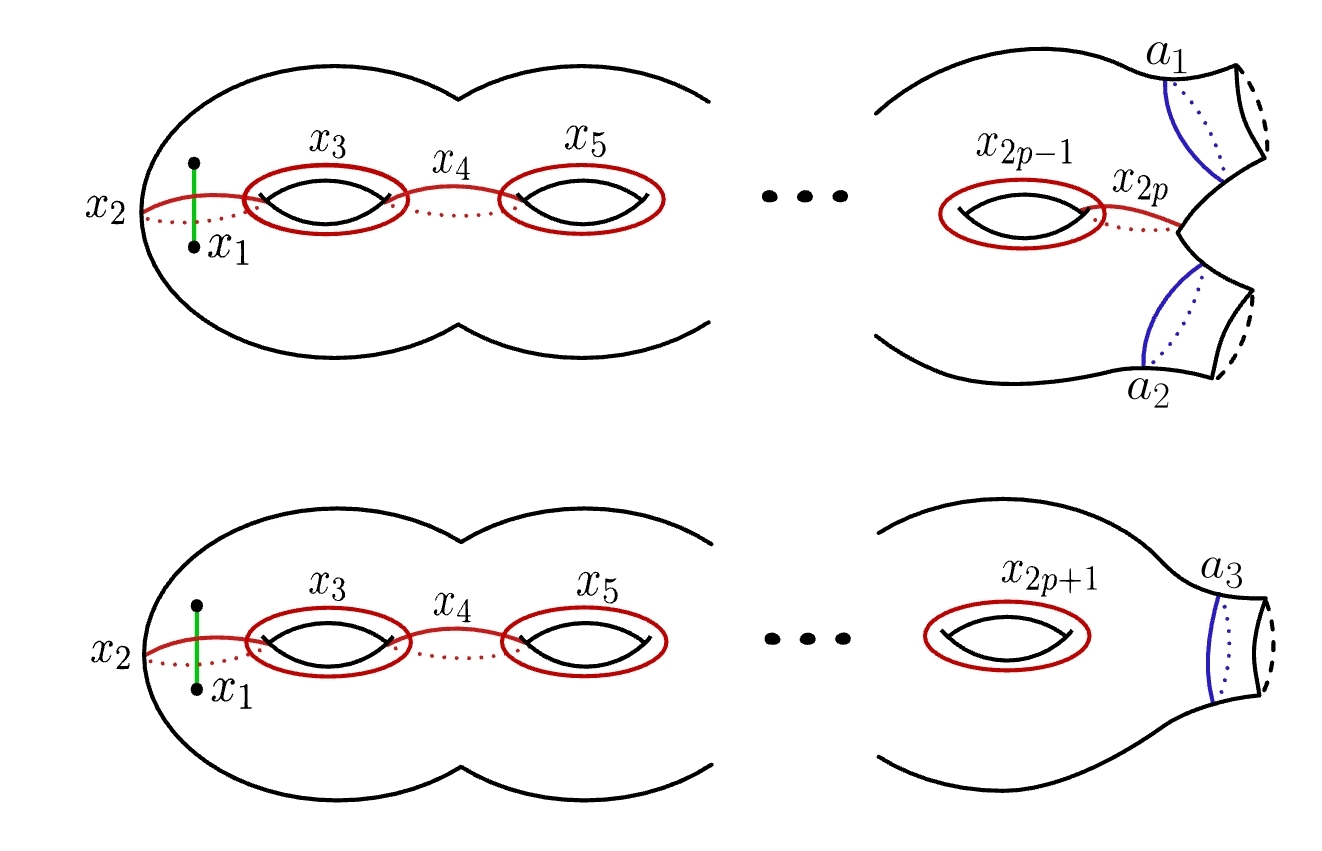}
    \caption{}
    \label{fig:seis-B}
\end{subfigure}

\vspace{0.5cm}

\begin{subfigure}{0.48\textwidth}
    \centering
    \includegraphics[width=\linewidth]{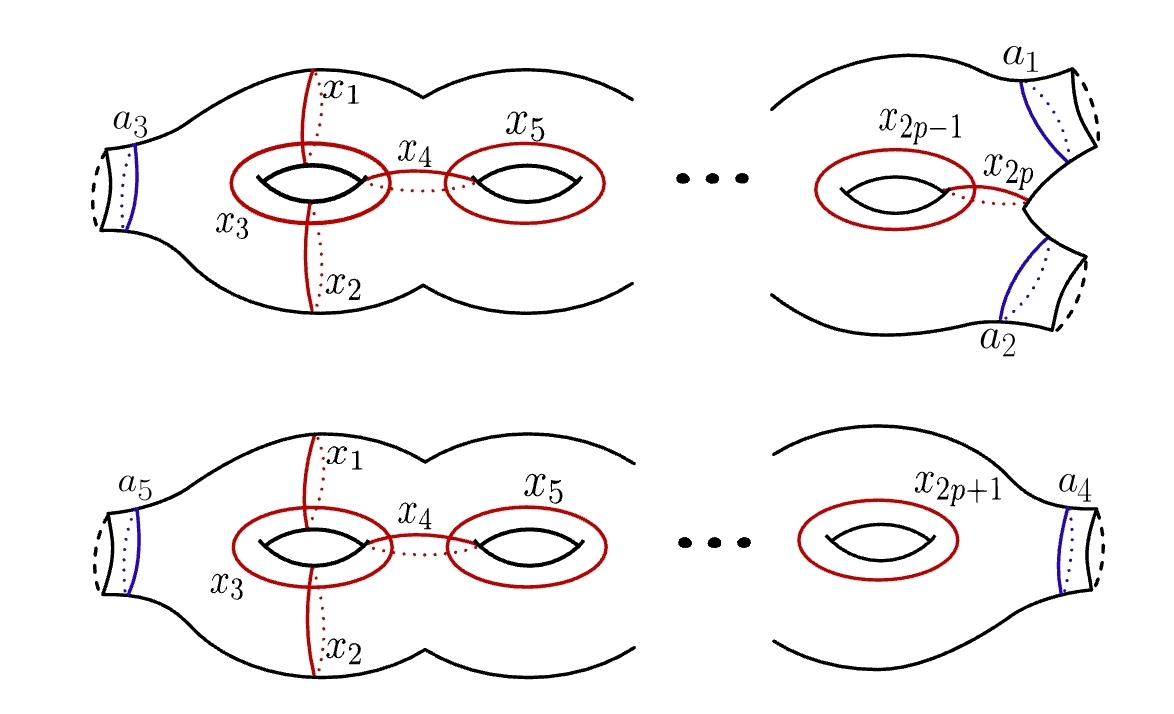}
    \caption{}
    \label{fig:seis-C}
\end{subfigure}
\hfill
\begin{subfigure}{0.5\textwidth}
    \centering
    \includegraphics[width=\linewidth]{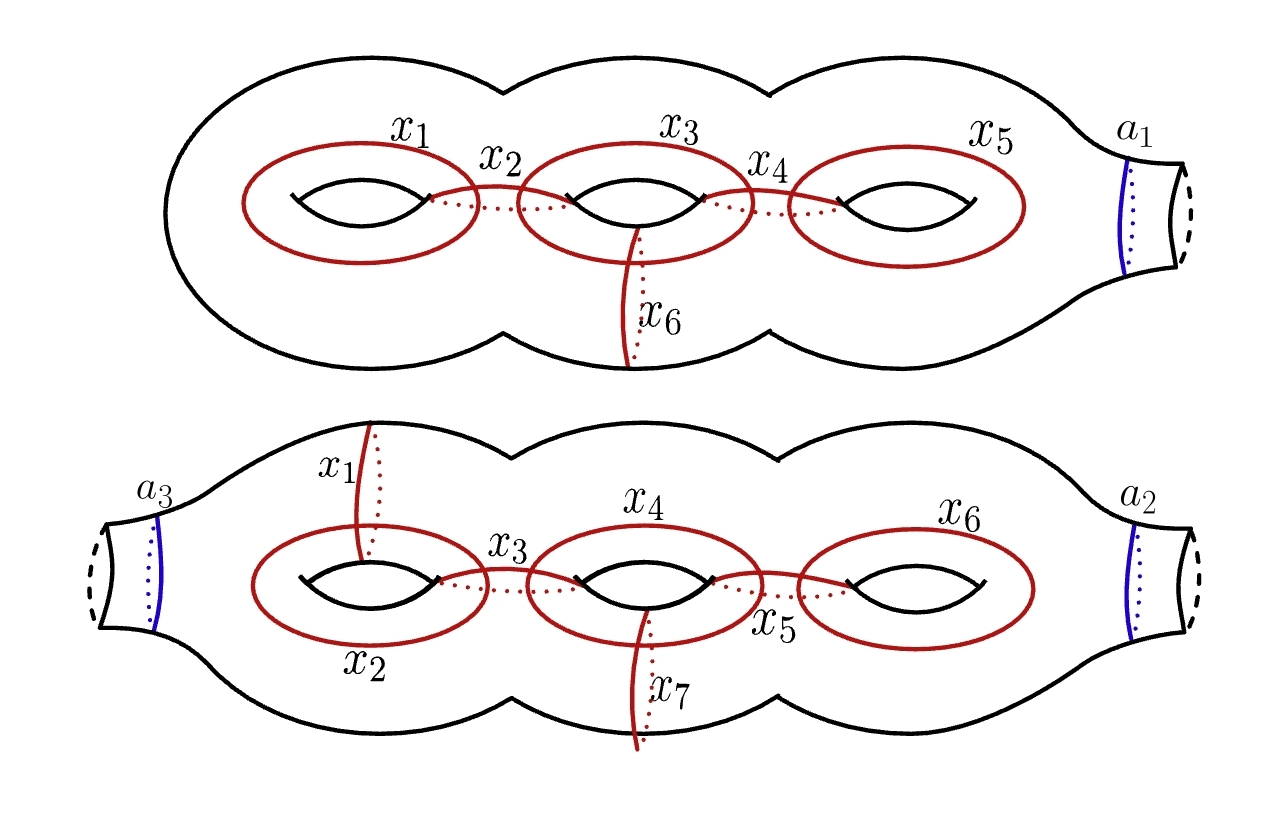}
    \caption{}
    \label{fig:seis-D}
\end{subfigure}

\vspace{0.5cm}

\begin{subfigure}{0.48\textwidth}
    \centering
    \includegraphics[width=\linewidth]{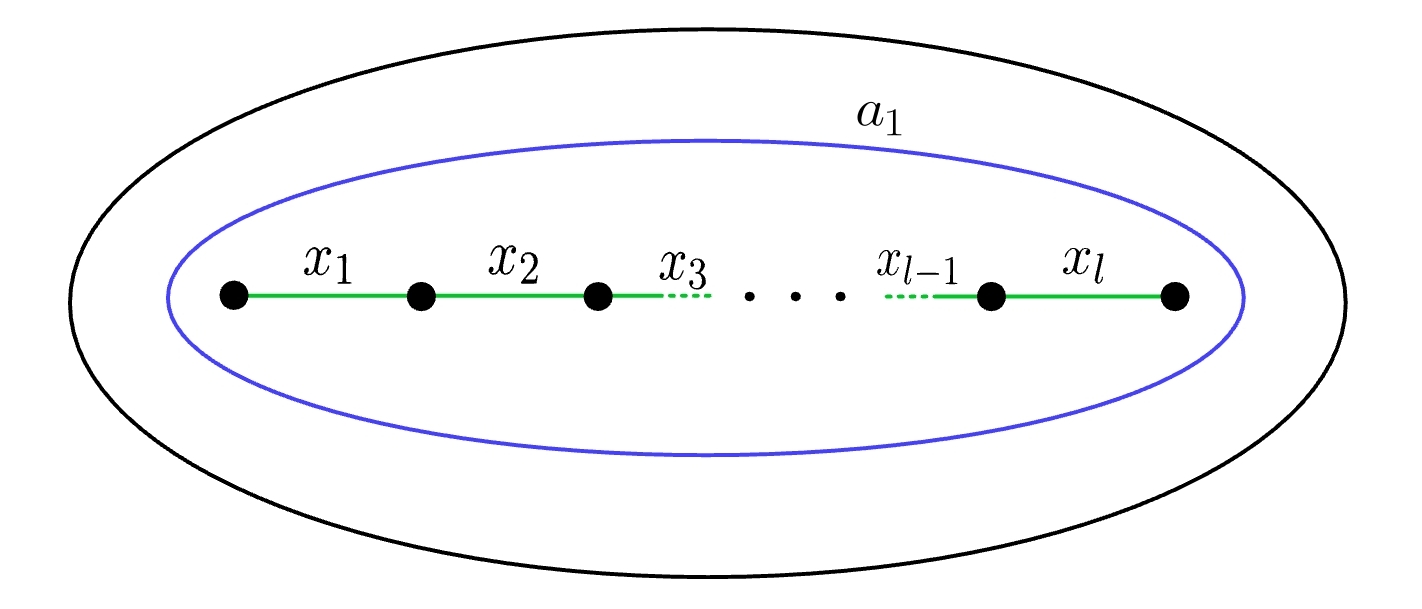}
    \caption{}
    \label{fig:seis-E}
\end{subfigure}
\hfill
\begin{subfigure}{0.48\textwidth}
    \centering
    \includegraphics[width=\linewidth]{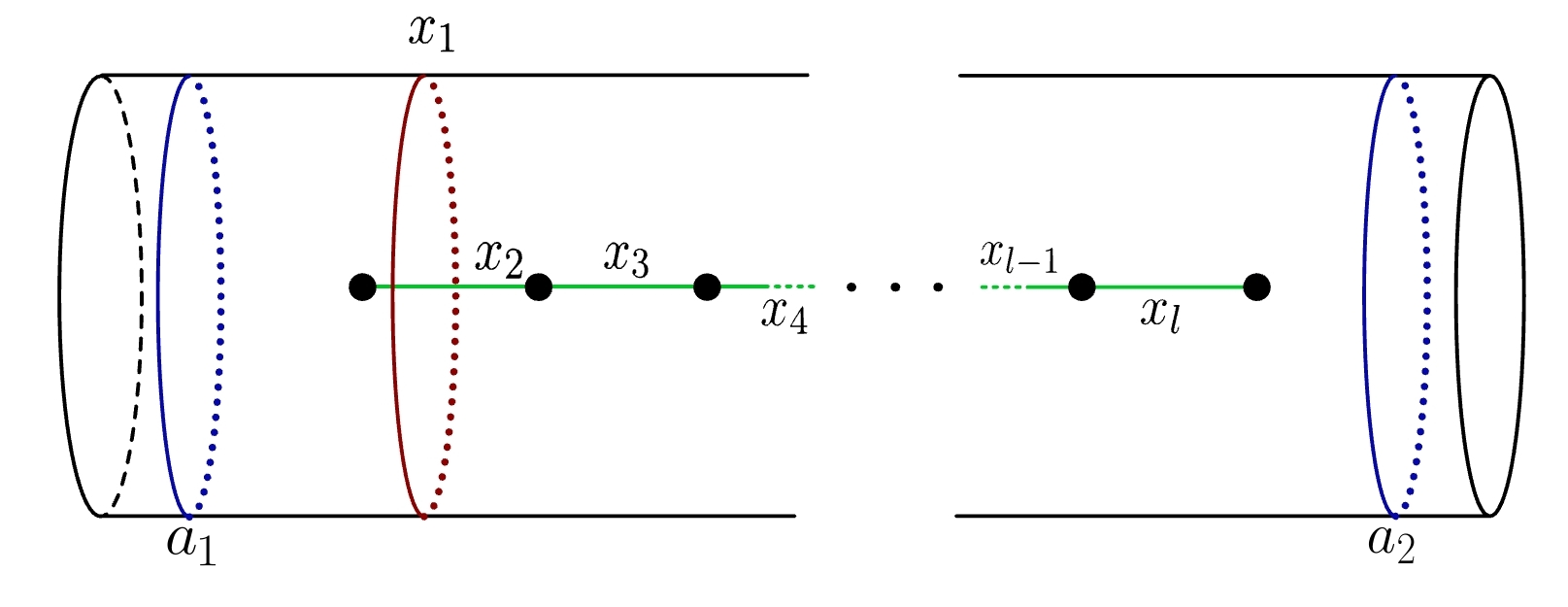}
    \caption{}
    \label{fig:seis-F}
\end{subfigure}

\caption{Homomorphisms $\rho:A(\Gamma)\rightarrow \Map(S')$.}
\label{fig:seis}
\end{figure}

\begin{lem}\label{PVRepresentations} The following maps are homomorphisms:

\begin{enumerate}[label=(\arabic*),leftmargin=*]

\item The map $\rho:A(A_l)\rightarrow \Map(S')$ defined by mapping each generator $x_i$ of $A(A_l)$ to the Dehn twist $x_i$ in Figure \ref{fig:seis}(\subref{fig:seis-A}). Moreover, the following equalities hold, where $a_1$, $a_2$ and $a_3$ are the Dehn twists in Figure \ref{fig:seis}(\subref{fig:seis-A}):

$$\rho(\Delta^4(A_{2p}))=a_1,\quad\rho(\Delta^2(A_{2p+1}))=a_2a_3.$$

\item The map $\rho:A(B_l)\rightarrow \Map(S')$ defined by mapping the generator $x_1$ to the half twist $x_1$ in Figure \ref{fig:seis}(\subref{fig:seis-B}), and each $x_i$ for $i\geq 2$ to the Dehn twist $x_i$ in Figure \ref{fig:seis}(\subref{fig:seis-B}). Moreover, the following equalities hold, where $a_1$, $a_2$ and $a_3$ are the Dehn twists in Figure \ref{fig:seis}(\subref{fig:seis-B}):

$$\rho(\Delta(B_{2p}))=a_1a_2,\quad\rho(\Delta^2(B_{2p+1}))=a_3.$$

\item The map $\rho:A(D_l)\rightarrow \Map(S')$ defined by mapping each generator $x_i$ to the Dehn twist $x_i$ in Figure \ref{fig:seis}(\subref{fig:seis-C}). Moreover, the following equalities hold, where the $a_i$ are the Dehn twists in \ref{fig:seis}(\subref{fig:seis-C}):

$$\rho(\Delta(D_{2p}))=a_1a_2a_3^{p-1},\quad\rho(\Delta^2(D_{2p+1}))=a_4a_5^{2p-1}.$$

\item The maps $\rho:A(E_6)\rightarrow \Map(S')$ and $\rho:A(E_7)\rightarrow \Map(S')$ defined by mapping each generator $x_i$ to the Dehn twist $x_i$ in Figure \ref{fig:seis}(\subref{fig:seis-D}). Moreover, the following equalities hold, where $a_1$, $a_2$ and $a_3$ are the Dehn twists in Figure \ref{fig:seis}(\subref{fig:seis-D}):

$$\rho(\Delta^2(E_6))=a_1,\quad\rho(\Delta(E_7))=a_2a_3^2.$$

\item The map $\rho:A(A_l)\rightarrow\Map(S')$ defined by mapping each generator $x_i$ to the half twist $x_i$ in Figure \ref{fig:seis}(\subref{fig:seis-E}). Moreover, the following equality holds, where $a_1$ is the Dehn twist in Figure \ref{fig:seis}(\subref{fig:seis-E}):

$$\rho(\Delta^2(A_{l}))=a_1.$$

\item The map $\rho:A(B_l)\rightarrow\Map(S')$ defined by mapping the generator $x_1$ to the Dehn twist $x_1$ in Figure \ref{fig:seis}(\subref{fig:seis-F}) and each $x_i$ for $i\geq 2$ to the half twist $x_i$ in Figure \ref{fig:seis}(\subref{fig:seis-F}). Moreover, the following equality holds, where $a_1$ and $a_2$ are the Dehn twists in Figure \ref{fig:seis}(\subref{fig:seis-F}):

$$\rho(\Delta(B_{l})) = a_1^{l-1} a_2. $$

\end{enumerate}

\end{lem}

We are finally in a position to introduce the presentation of the mapping class group of a punctured surface calculated by Labruère and Paris. Denote by $\Delta(s_1,s_2,...,s_k)$ to the fundamental element of the parabolic subgroup generated by $ \{s_1,s_2,...,s_k\}$.

\begin{theor}\label{Teo-L-P}\cite[Theorem 3.2]{Labruere-Paris}:
    Let $\Gamma=\Gamma(g,n)$ be the graph drawn in Figure \ref{Graph1}. Then $\Map(S')$ is isomorphic with the quotient of $A(\Gamma)$ by the following relations:\\

\begin{tabular}{p{32pt} c l}

\rm (R1) \quad & $\Delta^4(y_1,y_2,y_3,z)=\Delta^2
(x_0,y_1,y_2,y_3,z)$ & \quad if $g\ge 2$,\\

\rm (R2) \quad & $\Delta^2(y_1,y_2,y_3,y_4,y_5,z)=
\Delta(x_0,y_1,y_2,y_3,y_4,y_5,z)$ & \quad if $g\ge 3$,\\

\rm (R3) \quad & $\Delta(x_0,x_1,y_1,h_1)=
\Delta^2(x_1,y_1,h_1)$ & \quad if $n\ge 2$,\\

\rm (R4) \quad & $\Delta(x_0,x_1,y_1,y_2,y_3,z)=
\Delta^2(x_1,y_1,y_2,y_3,z)$ & \quad if $n\ge 1$ and $g\ge 2$,\\

\rm (R5a) \quad & $ x_0^{2g-n-2}\Delta(x_1,h_1,\dots,
h_{n-1})=\Delta^2(z,y_2,\dots,y_{2g-1})$ & \quad if $g\ge 2$,\\

\rm (R5b) \quad & $x_0^n=\Delta(x_1,h_1,\dots, 
h_{n-1})$ & \quad if $g=1$,\\

\rm (R5c) \quad & $\Delta^4 (x_0,y_1)=\Delta^2(h_1,\dots, 
h_{n-1})$ & \quad if $g=1$.\\

\end{tabular}

\end{theor}

\begin{figure}[H]
\begin{tikzpicture}
  \node at (-1.4, 0) {$\Gamma(g,n)$};
  \filldraw (0,0) circle (2pt) node[left] {$y_1$};
  \filldraw (1,0) circle (2pt) node[above] {$y_2$};
  \filldraw (2,0) circle (2pt) node[above] {$y_3$};
  \filldraw (3,0) circle (2pt) node[above] {$y_4$};
  \filldraw (4.4,0) circle (2pt) node[above] {$y_{2g-1}$};
  \filldraw (0,-1) circle (2pt) node[left] {$x_0$};
  \filldraw (0,1) circle (2pt) node[above] {$x_1$};
  \filldraw (1,1) circle (2pt) node[above] {$h_1$};
  \filldraw (2,1) circle (2pt) node[above] {$h_2$};
  \filldraw (3.4,1) circle (2pt) node[above] {$h_{n-1}$};
  \filldraw (2,-1) circle (2pt) node[left] {$z$};
  \draw (0,0) -- (1,0);
  \draw (1,0) -- (2,0);
  \draw (2,0) -- (3,0);
  \draw (3,0) -- (3.3,0) ;
  \draw[dashed] (3.3,0) -- (4.1,0);
  \draw (4.1,0) -- (4.4,0);
  \draw (0,1) -- node[above]{4} (1,1);
  \draw (1,1) -- (2,1);
  \draw (2,1) -- (2.3,1);
  \draw[dashed] (2.3,1) -- (3.1,1);
  \draw (3.1,1) -- (3.4,1);
  \draw (0,0) -- (0,1);
  \draw (0,0) -- (0,-1);
  \draw (2,0) -- (2,-1);
\end{tikzpicture}
\centering
\caption{}
\label{Graph1}
\end{figure}

\vspace{-19pt}

From Theorem \ref{Teo-L-P}, we want to produce a presentation of the boundary permuting mapping class group using Lemma \ref{Lema-present} applied to the short exact sequence:
$$1\rightarrow U_{S}\rightarrow\Map_o(S)\overset{\In}{\rightarrow}\Map(S')\rightarrow1,$$ 

Lemma \ref{Lema-present} implies that $\Map_o(S)$ has a presentation with generators $$X_{\Map_o}=\{ x_0,x_1,z,y_1,...,y_{2g-1},b_1,...,b_{n-1},u_1,...,u_n\}.$$ and relations $R_{U}\cup R_1\cup R_2,$ with $$R_1=\{\tilde{r}w_r^{-1}:r\in R_{\Map}\}\textnormal{,}$$ $$R_2=\{\tilde{x}y\tilde{x}^{-1}v(x,y)^{-1}:\tilde{x}\in \tilde{X}_{\Map} \textnormal{ and } y\in X_{U}\}.$$

Computing $R_1$ and $R_2$ constitutes the core of the proof of Theorem \ref{ThMain}. Some of the elements in $R_{\Map}$ are relations between powers of the fundamental element whose images under the homomorphisms in Lemma \ref{PVRepresentations} include half twists as factors. So, before proving Theorem \ref{ThMain} we need to adapt these homomorphisms so that the half twists are substituted by boundary swaps. We also need to verify that boundary swaps satisfy some of the relations that half twists do, so that boundary swaps together with Dehn twists also define homomorphisms from Artin groups into $\Map_o(S)$.

\subsection{Artin group homomorphisms in the presence of boundary swaps} We study two relations involving the boundary swaps $\{b_1,...,b_{n-1}\}$ in $X_{\Map_o}$. The fact that these relations are satisfied are crucial for the homomorphisms later in this subsection.

\begin{lem}\label{lemswap1}
Any $b_i,b_{i+1}\in \{b_1,...,b_{n-1}\}$ satisfy: $$b_ib_{i+1}b_i=b_{i+1}b_ib_{i+1}.$$
    
\end{lem}

\begin{proof}[\sc Proof of Lemma {\rm \ref{lemswap1}}] 
The image of the boundary swaps under the $\In(\cdot)$ map are half twists, which satisfy the equality. Hence, when dealing with boundary swaps, if both sides of the equality are different, the difference $b_ib_{i+1}b_i(b_{i+1}b_ib_{i+1})^{-1}$ must be a product of the Dehn twists $\{u_1,...,u_{n}\}$ generating $U_{S}$, as $U_{S}$ is the kernel of $\In(\cdot)$. 

Since $b_ib_{i+1}b_i(b_{i+1}b_ib_{i+1})^{-1}$ fixes every arc $\alpha_1,...,\a_{i-1},\a_{i+2},...,\a_{n-1}$ in Figure \ref{election}, we deduce that the product must be of the form $u_{i}^{k_i}u_{i+1}^{k_{i+1}}u_{i+2}^{k_{i+2}}$, since those are the only twists in $\{u_1,...,u_{n}\}$ that fix those arcs. Hence, in order to determine the product $b_ib_{i+1}b_i(b_{i+1}b_ib_{i+1})^{-1}$, it suffices to study the action of $b_ib_{i+1}b_i(\cdot)$ and $b_{i+1}b_ib_{i+1}(\cdot)$ on the arcs $\a_i$ and $\a_{i+1}$, which is done in Figure \ref{swap1}.

\begin{figure}[H]
\hspace{-0.6cm}\includegraphics[width=9cm]{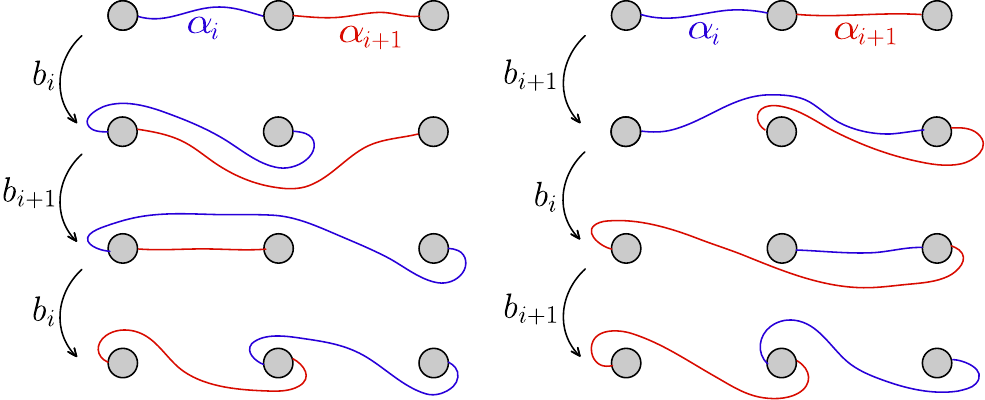}
\centering
\caption{}
\label{swap1}
\end{figure}

Because the result is the same, we conclude the equality in the statement of the lemma.

\end{proof}

\begin{lem}\label{lemswap2}
Let $b_i\in\{b_1,...,b_{n-1}\}$ be a boundary swap along an arc $\a\subset S$ and $x$ a Dehn twist along a curve $\c\subset S$ with $i(\a,\c)=1$. Then $b_i$ and $x$ satisfy: $$b_ixb_ix=xb_ixb_i.$$
\end{lem}

\begin{proof}[\sc Proof of Lemma {\rm \ref{lemswap2}}]

Following the same approach as in the proof of Lemma \ref{lemswap1}, we study the action of $b_ixb_ix(\cdot)$ and $xb_ixb_i(\cdot)$ on the arc $\alpha$, which we carry out in Figure \ref{swap2}.

\begin{figure}[H]
\includegraphics[width=11cm]{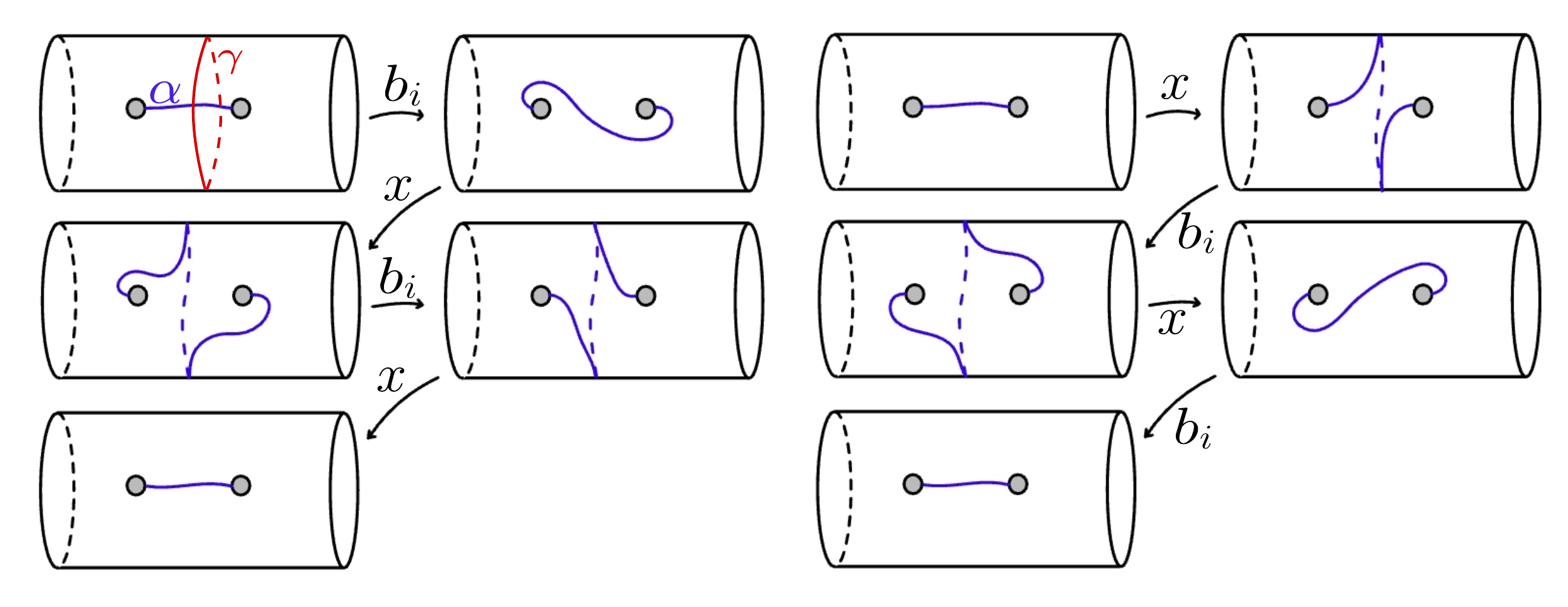}
\centering
\caption{}
\label{swap2}
\end{figure}

Because the result is the same, we conclude the equality in the statement of the lemma.

\end{proof}

We are now in a position to introduce the homomorphisms from Artin groups into $\Map_o(S)$ that we will need for our arguments. As will become apparent, the homomorphisms are a modification of those in Lemma \ref{PVRepresentations} to the case of surfaces with boundary.  

\begin{figure}[htbp]
\centering

\begin{subfigure}{0.48\textwidth}
    \centering
    \includegraphics[width=\linewidth]{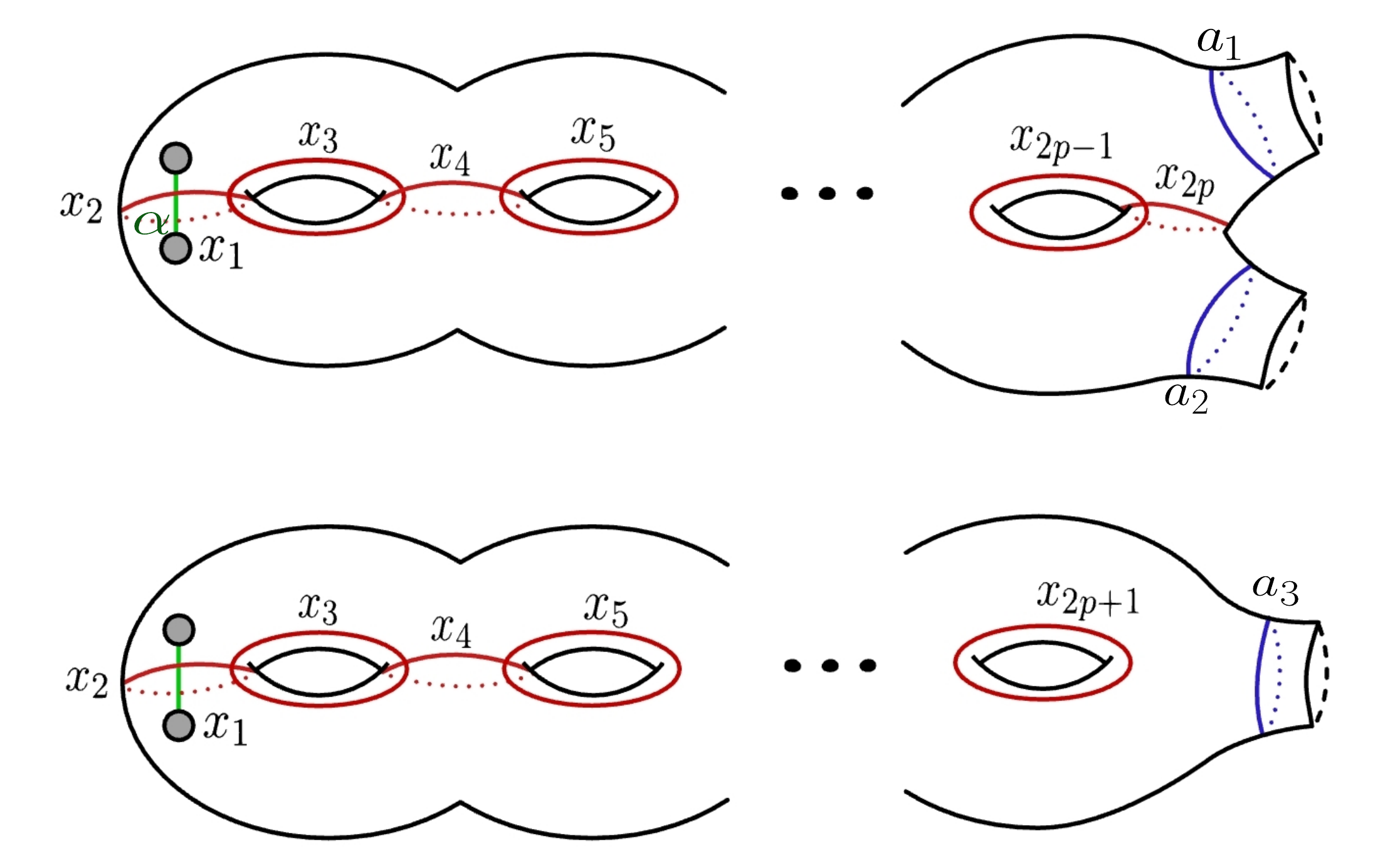}
    \caption{}
    \label{fig:tres-A}
\end{subfigure}
\hfill
\begin{subfigure}{0.48\textwidth}
    \centering
    \includegraphics[width=\linewidth]{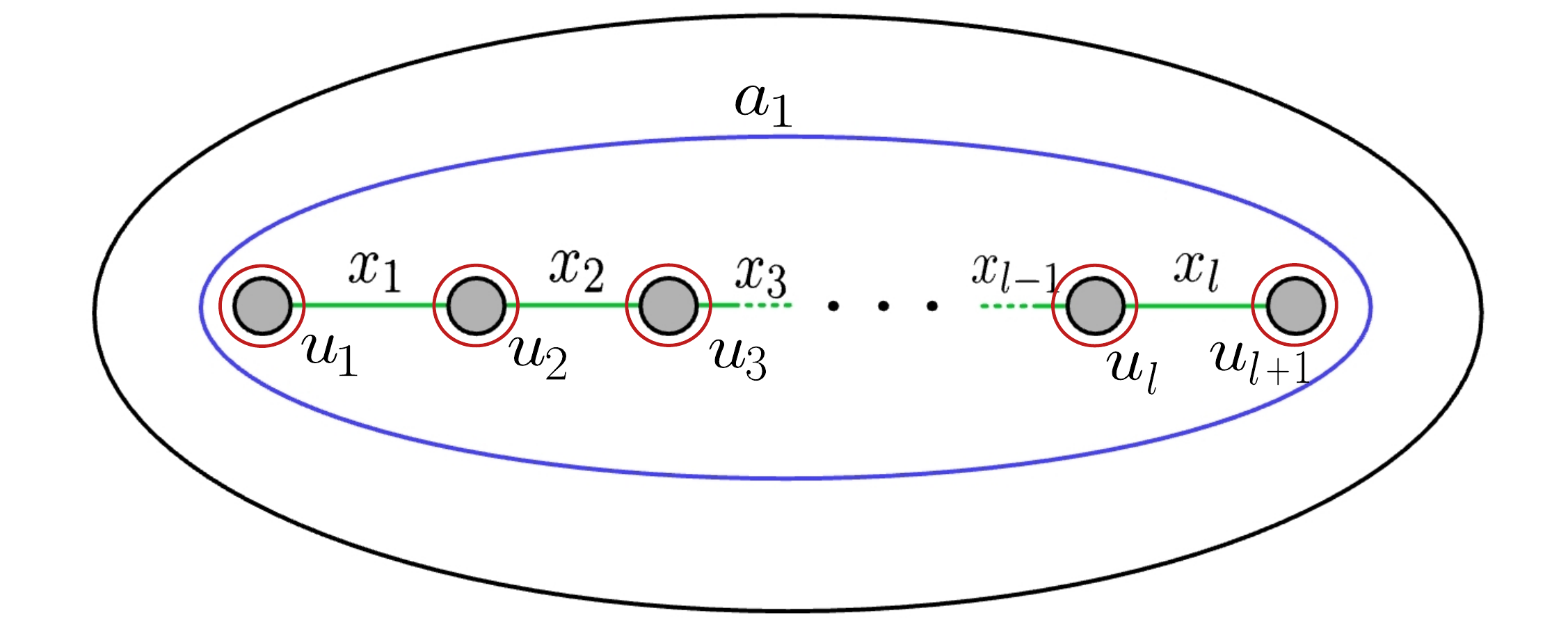}
    \caption{}
    \label{fig:tres-B}
\end{subfigure}

\vspace{0.5cm}

\begin{subfigure}{0.48\textwidth}
    \centering
    \includegraphics[width=\linewidth]{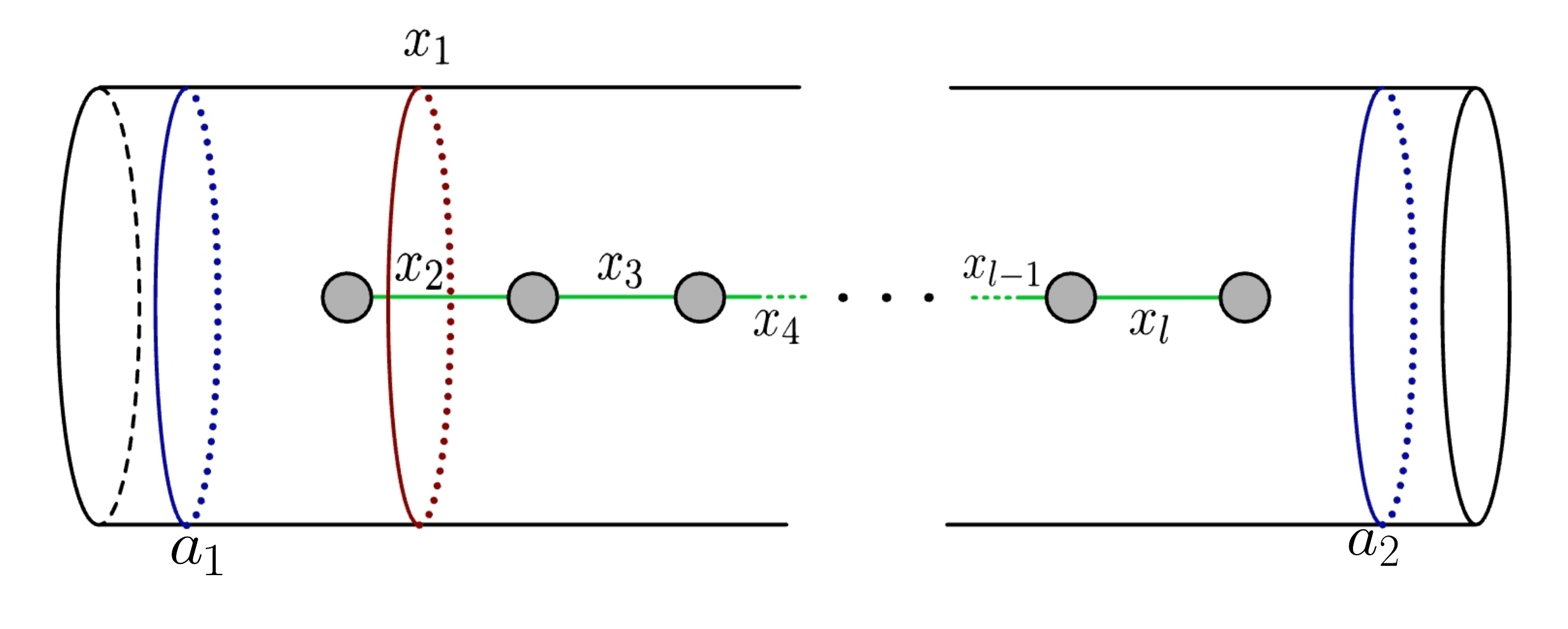}
    \caption{}
    \label{fig:tres-C}
\end{subfigure}

\caption{Homomorphisms $\rho:A(\Gamma)\rightarrow \Map_o(S)$. The darkened circles represent boundary components.}
\label{fig:tres}
\end{figure}

\begin{lem}\label{PVRepresentations2}
The following maps define representations of Artin groups:

\begin{enumerate}[label=(\arabic*),leftmargin=*]

\item The map $\rho:A(B_l)\rightarrow\Map_o(S)$ defined by mapping the generator $x_1$ to the boundary twist $x_1$ in Figure \ref{fig:tres}(\subref{fig:tres-A}), and each $x_i$ for $i\geq 2$ to the Dehn twist $x_i$ in Figure \ref{fig:tres}(\subref{fig:tres-A}). Moreover, the following equalities hold, where $a_1$, $a_2$ and $a_3$ are the Dehn twists in Figure \ref{fig:tres}(\subref{fig:tres-A}):

$$\rho(\Delta(B_{2p}))=a_1a_2,\quad\rho(\Delta^2(B_{2p+1}))=a_3.$$

\item The map $\rho:A(A_l)\rightarrow\Map_o(S)$ defined by mapping each generator $x_i$ to the boundary swap $x_i$ in Figure \ref{fig:tres}(\subref{fig:tres-B}). Moreover, the following equality holds, where $a_1$ and the $u_i$ are the Dehn twists in Figure \ref{fig:tres}(\subref{fig:tres-B}):

$$\rho(\Delta^2(A_{l}))=a_1u_1^{-1}u_2^{-1}...\,u_{l+1}^{-1}.$$

\item The map $\rho:A(B_l)\rightarrow\Map_o(S)$ defined by mapping the generator $x_1$ to the Dehn twist $x_1$ in Figure \ref{fig:tres}(\subref{fig:tres-C}) and each $x_i$ for $i\geq 2$ to the boundary twist $x_i$ in Figure \ref{fig:tres}(\subref{fig:tres-C}) is a representation of $A(B_l)$. Moreover, the following equality holds, where $a_1$ and $a_2$ are the Dehn twists in Figure \ref{fig:tres}(\subref{fig:tres-C}):

$$\rho(\Delta(B_{l})) = a_1^{l-1} a_2. $$

\end{enumerate}    
\end{lem}

\begin{proof}[\sc Proof of Lemma {\rm \ref{PVRepresentations2}}] 

Boundary swaps and Dehn twists satisfy the relations given by the graph labels because of Lemmas \ref{lemswap1}, \ref{lemswap2} and \ref{relDehnHalf}, hence the maps define homomorphisms. We only need to determine the image of the powers of the fundamental elements in the lemma. 

\textbf{Computing $\rho(\Delta(B_{2p})) $ and
$\rho(\Delta^2(B_{2p+1}))$ in homomorphism (1).} We explain the reasoning for $2p$, as the case $2p+1$ is analogous. Take the homomorphism $\rho:A(B_{2p})\rightarrow\Map_o(S)$, and for the surface $S$, take the map $\In:\Map_o(S)\rightarrow \Map(S')$ defined at the beginning of Section \ref{Sect2}, where $S'$ has as many punctures as $S$ has boundary components.

The images $\In(\rho(\Delta(B_{2p})))$ and $\In(a_1 a_2)$ are equal because of the homomorphism (2) in Lemma \ref{PVRepresentations}. Hence the difference between $\rho(\Delta(B_{2p}))$ and $a_1 a_2$ in $\Map(S)$ must be a product $u_1^{k_1}u_2^{k_2}$ of the Dehn twists $\{u_1,u_2\}$ generating $U_{S}$, as those are the elements in the kernel of $\In(\cdot)$. Particularly the product must be of the form $u_1^ku_2^k$, as $\rho(\Delta(B_{2p}))^{-1}a_1 a_2$ is in the center of $\Map_o(S)$, and that can only happen if $k_1=k_2$.

We can determine $k$ by studying the image of the arc $\alpha$ that defines $x_1$ in Figure \ref{fig:tres}(\subref{fig:tres-A}) by $\rho(\Delta(B_{2p}))$: the arc $\alpha$ has a different image up to homotopy for each value of $k$, hence its image determines $k$. We use the fact that $\Delta(B_l)=x_lx_{l-1}...x_2x_1x_2...x_{l-1}x_l\Delta(B_{l-1})$ (see \cite[Proposition 2.9]{Labruere-Paris}), and hence $$\Delta(B_l)=(x_lx_{l-1}...x_2x_1x_2...x_{l-1}x_l)(x_{l-1}...x_2x_1x_2...x_{l-1})...(x_2x_1x_2)x_1.$$

Any Dehn twist $x_i$ for $i\neq2$ preserves $\alpha$, and so does the product $x_2x_1x_2$ because of the calculations in the proof of Lemma \ref{lemswap2}. Hence $\rho(\Delta(B_{2p}))$ preserves $\alpha$, as does $a_1a_2$, which means $k=0$, and $\rho(\Delta(B_{2p})) = a_1 a_2$.

\textbf{Computing $\rho(\Delta^2(A_{l}))$:} By a similar reasoning as in the previous point, we deduce that: $\rho(\Delta^2(A_{l}))=a_1u_1^ku_2^k...u_{l+1}^k$. We will use that $$\Delta(A_l)=(x_lx_{l-1}...x_2x_1)(x_{l}...x_2)...(x_lx_{l-1})x_l=(x_1x_2...x_{l-1}x_l)(x_1...x_{l-1})...(x_1x_2)x_1,$$  because of \cite[Subsection 2.1, page 238]{Garside}. By regrouping the product, we have that:
$$\Delta^2(A_{l})=\underbrace{(x_l...x_1)(x_{l}...x_2)}_{p_4}\underbrace{(x_{l}...x_3)...(x_lx_{l-1})x_l}_{p_3}\underbrace{(x_1...x_l)(x_1...x_{l-1})}_{p_2}\underbrace{(x_1...x_{l-2})...(x_1x_2)x_1}_{p_1}.$$

As in the previous point, we determine $k$ by studying the image of the arc $\alpha$ in Figure \ref{MovA} by $\rho(\Delta^2(A_{l}))$. The image $\rho(\Delta^2(A_{l}))(\a)$ will uniquely determine $k$, since for every distinct value of $k$, $a_1u_1^ku_2^k...u_{l+1}^k(\a)$ is a different arc.

The product $p_1=(x_1...x_{l-2})...(x_1x_2)x_1$ leaves $\a$ invariant, while $p_2=(x_1...x_l)(x_1...x_{l-1})$ takes $\a$ to $\b$ (see Figure \ref{MovA}). Then $p_3=(x_l...x_3)...(x_lx_{l-1})x_l$ leaves $\b$ invariant, and $p_4=(x_l...x_1)(x_{l}...x_2)$ takes $\b$ to $\c$. The arc $\c$ is $u_l^{-1}u_{l+1}^{-1}(\a)$, hence $k=-1$, and $\rho(\Delta^2(A_{l}))=a_1u_1^{-1}u_2^{-1}...u_{l+1}^{-1}$.

\begin{figure}[H]
\includegraphics[width=11.7cm]{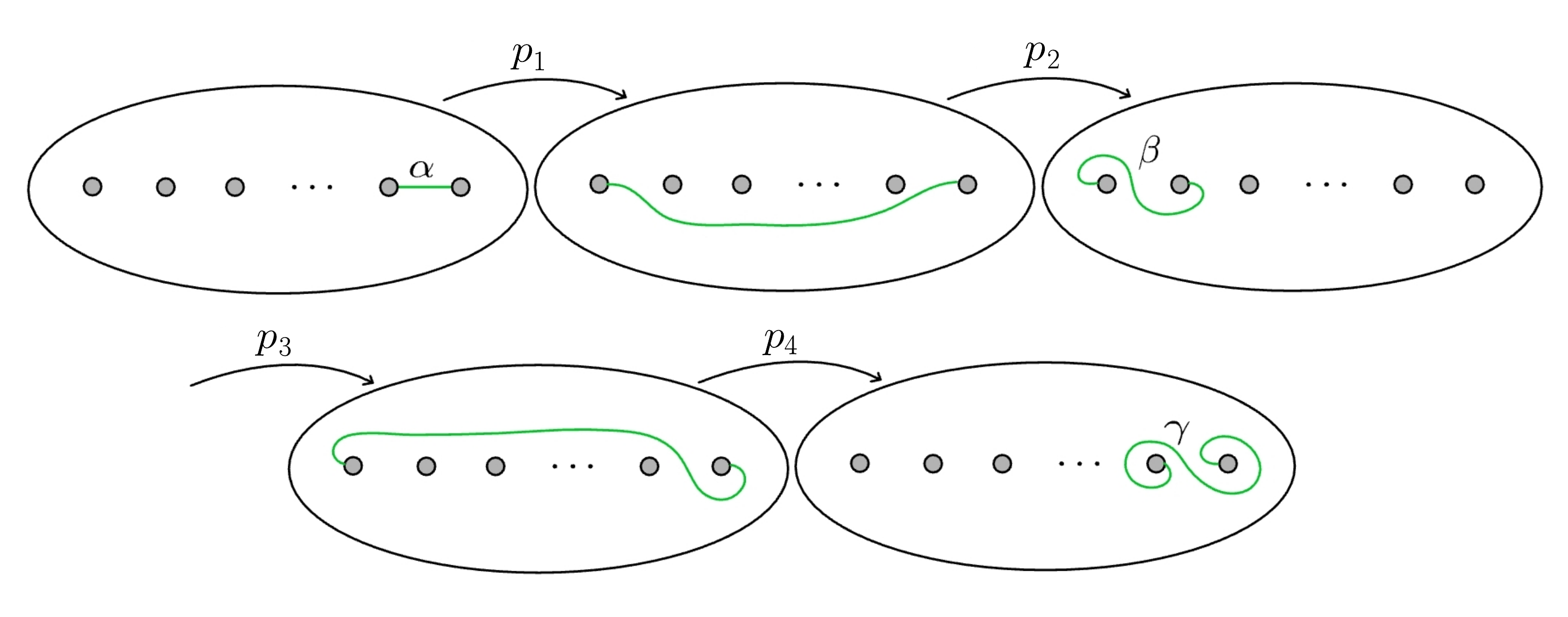}
\centering
\caption{}
\label{MovA}
\end{figure}

\textbf{Computing $\rho(\Delta(B_{l}))$:} Again: $\rho(\Delta(B_{l}))a_1^{-l+1}a_2^{-1}=u_1^ku_2^k...u_l^k$. We use that $\Delta(B_l)=x_lx_{l-1}...x_2x_1x_2...x_{l-1}x_l\Delta(B_{l-1})$. By iterating and regrouping, we have that:
$$\Delta(B_l)=\underbrace{x_lx_{l-1}...x_2}_{p_4}\underbrace{x_1x_2...x_{l-1}}_{p_3}\underbrace{x_lx_{l-1}...x_2}_{p_2}\underbrace{x_1x_2...x_{l-1}}_{p_1}\Delta(B_{l-2})$$

we determine $k$ by studying the image of the arc $\alpha$ in Figure \ref{MovB}. The map $\rho(\Delta(B_{l-2}))$ leaves $\a$ invariant as it has disjoint support, and $$(x_lx_{l-1}...x_2x_1x_2...x_{l-1}x_l)(x_{l-1}...x_2x_1x_2...x_{l-1})=p_4p_3p_2p_1$$ also preserves $\a$ (see Figure \ref{MovB}). Hence $k=0$, and $\rho(\Delta(B_{l})) = a_1^{l-1} a_2.$.

\begin{figure}[H]
\includegraphics[width=13.7cm]{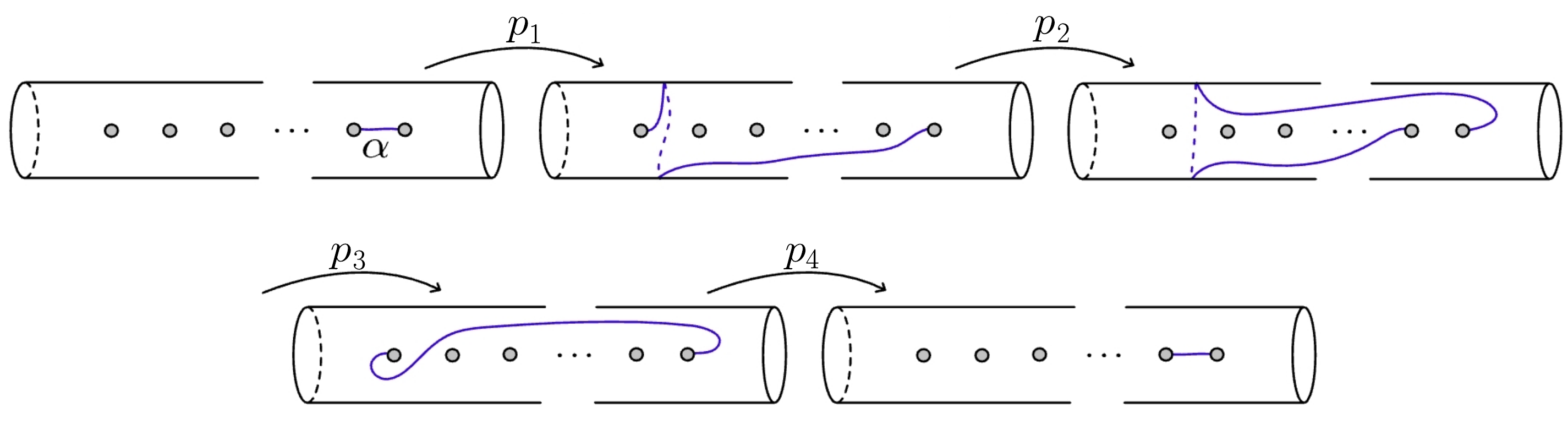}
\centering
\caption{}
\label{MovB}
\end{figure}

\end{proof}

\subsection{The presentation} We now have the tools to compute a presentation of the boundary permuting mapping class group.

\begin{theor}\label{ThMain}
Let $\Psi=\Psi(g,n)$ be the graph in Figure \ref{Graph2}. Then $\Map_o(S)$ has a presentation $$\Map_o(S)=\langle X_{A(\Psi)}|R_{A(\Psi)}\cup R_{S}\rangle,$$ where $ X_{A(\Psi)}$ is the set of generators of the Artin group $A(\Psi)$, $R_{A(\Psi)}$ its set of relations, and $R_{S}$ is the following set of relations:\\

\begin{tabular}{p{32pt} c l}

\rm (R1) \quad & $\Delta^4(y_1,y_2,y_3,z)=\Delta^2
(x_0,y_1,y_2,y_3,z),$ & \quad  if $g\ge 2,$\\

\rm (R2) \quad & $\Delta^2(y_1,y_2,y_3,y_4,y_5,z)=
\Delta(x_0,y_1,y_2,y_3,y_4,y_5,z)$ & \quad if $g\ge 3$,\\

\rm (R3) \quad & $\Delta(x_0,x_1,y_1,b_1)=
\Delta^2(x_1,y_1,b_1)$ & \quad if $n\ge 2$,\\

\rm (R4) \quad & $u_1\Delta(x_0,x_1,y_1,y_2,y_3,z)=
\Delta^2(x_1,y_1,y_2,y_3,z)$ & \quad if $n\ge 1$ and $g\ge 2$,\\

\rm (R5a) \quad & $ x_0^{2g-n-2}\Delta(x_1,b_1,\dots,
b_{n-1})=\Delta^2(z,y_2,\dots,y_{2g-1})$ & \quad if $g\ge 2$,\\

\rm (R5b) \quad & $x_0^n=\Delta(x_1,b_1,\dots, 
b_{n-1})$ & \quad if $g=1$,\\

\rm (R5c) \quad & $\Delta^4 (x_0,y_1)=\Delta^2(b_1,\dots, 
b_{n-1})$ & \quad if $g=1$,\\
 
\rule{0pt}{13pt}\makebox[0pt][l] { for  $i\in\{1,...,n-1\}$:} \\

\rm (Ci) \quad & $b_iu_i=u_{i+1}b_i,$ & \quad\\

\rm (Di) \quad & $u_{i}b_i=b_{i}u_{i+1}.$ & \quad\\

\end{tabular}

\end{theor}

\begin{figure}[H]
\begin{tikzpicture}
  \node at (-1.7, 0) {$\Psi({g,n})$};
  \filldraw (0,-0.2) circle (2pt) node[left] {$y_1$};
  \filldraw (1,-0.2) circle (2pt) node[above] {$y_2$};
  \filldraw (2,-0.2) circle (2pt) node[above] {$y_3$};
  \filldraw (3,-0.2) circle (2pt) node[above] {$y_4$};
  \filldraw (4.4,-0.2) circle (2pt) node[above] {$y_{2g-1}$};
  \filldraw (0,-1) circle (2pt) node[left] {$x_0$};
  \filldraw (0,1) circle (2pt) node[left] {$x_1$};
  \filldraw (1.5,1) circle (2pt) node[below] {$b_1$};
  \filldraw (3,1) circle (2pt) node[below] {$b_2$};
  \filldraw (4.5,1) circle (2pt) node[below] {$b_{n-1}$};
  \filldraw (2,-1) circle (2pt) node[left] {$z$};
  \filldraw (0.75,1.6) circle (2pt) node[above] {$u_1$};
  \filldraw (2.25,1.6) circle (2pt) node[above] {$u_2$};
  \filldraw (5.25,1.6) circle (2pt) node[above] {$u_n$};

  \draw (0.75,1.6) -- node[above]{4} (1.5,1);
  \draw (2.25,1.6) -- node[above]{4} (1.5,1);
  \draw (2.25,1.6) -- node[above]{4} (3,1);
  \draw (5.25,1.6) -- node[above]{4} (4.5,1);
  
  \draw (0,-0.2) -- (1,-0.2);
  \draw (1,-0.2) -- (2,-0.2);
  \draw (2,-0.2) -- (3,-0.2);
  \draw (3,-0.2) -- (3.3,-0.2) ;
  \draw[dashed] (3.3,-0.2) -- (4.1,-0.2);
  \draw (4.1,-0.2) -- (4.4,-0.2);
  \draw (0,1) -- node[below]{4} (1.5,1);
  \draw (1.5,1) -- (3,1);
  \draw (3,1) -- (3.3,1);
  \draw[dashed] (3.3,1) -- (4.2,1);
  \draw (4.2,1) -- (4.5,1);
  \draw (0,0) -- (0,1);
  \draw (0,0) -- (0,-1);
  \draw (2,-0.2) -- (2,-1);
  \draw[line width=1.3pt, dash pattern=on 0pt off 4pt, line cap=round] (3.5,1.6) -- (4,1.6);
\end{tikzpicture}
\centering
\caption{}
\label{Graph2}
\end{figure}

\begin{proof}[\sc Proof of Theorem {\rm \ref{ThMain}}] 

 We want to write $\Map_o(S)$ as a quotient of $A(\Psi)$ by a series of relations. Consider the map: $$\pi:A(\Psi)\rightarrow\Map_o(S)$$ defined by sending each generator $\{ x_0,x_1,z,y_1,...,y_{2g-1},b_1,...,b_{n-1},u_1,...,u_n\}$ of the Artin group to the mapping class with the same name in $\Map_o(S)$ (see Figure \ref{Generators}). Checking that generators of $\Map_o(S)$ satisfy the relations of the Artin group $A(\Psi)$ implies that the map $\pi$ is an epimorphism.

Recall from Subsection \ref{2.3} the exact sequence $$1\rightarrow U_{S}\rightarrow\Map_o(S)\overset{\In}{\rightarrow}\Map(S')\rightarrow1.$$ According to Lemma \ref{Lema-present}, if we choose a preimage $\tilde{x}\in \In^{-1}(x)$ for each $x\in X_H$ and write $$ \tilde{X}_\Map:=\{\tilde{x}:x\in X_\Map\},$$ the group $\Map_o(S)$ has a presentation with generators $X_{\Map_o}$, and relations $R_{U}\cup R_1\cup R_2,$ with $$R_1=\{\tilde{r}w_r^{-1}:r\in R_{\Map}\}\textnormal{,}$$ $$R_2=\{\tilde{x}y\tilde{x}^{-1}v(x,y)^{-1}:\tilde{x}\in \tilde{X}_{\Map} \textnormal{ and } y\in X_{U}\},$$ 

 where if $r=x_1^{\varepsilon_1}...\,x_l^{\varepsilon_l}\in R_\Map$, then $\tilde{r}=\tilde{x_1}^{\varepsilon_1}...\,\tilde{x_l}^{\varepsilon_l}\in \Map_o(S)$.
 
We will compute $R_2$, then check that generators of $\Map_o(S)$ satisfy the relations of the Artin group $A(\Psi)$ (which consist of $R_{U}$, and relations from the Artin group $A(\Gamma)$), and lastly compute $$R'_1=\{\tilde{r}w_r^{-1}:r \textnormal{ is a relation in Theorem \ref{Teo-L-P}}\}\subset R_1.$$ 

The fact that $\pi$ is an epimorphism together with Lemma \ref{Lema-present} will prove that the quotient of $A(\Psi)$ by $R_{U}\cup R_1'\cup R_2$ is isomorphic to $\Map_o(S)$.

\textbf{Computing $R_2$:}

We consider each conjugate $\tilde{x}u_i\tilde{x}$ where $u_i$ a Dehn twist in $X_{U}=\{ u_1,...,u_n\}$ by each \\$\tilde{x}\in\{ x_0,x_1,z,y_1,...,y_{2g-1},b_1,...,b_{n-1}\}$, and write it as a word over $X_{U}$. The Dehn twists in $X_{U}$ commute with every element in $\{ x_0,x_1,z,y_1,..,y_{2g-1},b_1,...,b_{n-1}\}$ except for the $b_i$, so we only have to consider those. Following an arc tracking argument as in the proof of Lemma \ref{PVRepresentations2}, we get that $b_iu_ib_i^{-1}=u_{i+1}$ and  $b_iu_{i+1}b_i^{-1}=u_i$. Hence we get two families of relations:

\hspace{45pt}\rule{0pt}{13pt}\makebox[0pt][l] {For  $i\in\{1,...,n-1\}$:}
\vspace{10pt}
$$ \textnormal{(Ci)} \quad  b_iu_ib_i^{-1}=u_{i+1},\quad \quad \quad \textnormal{(Di)} \quad  b_iu_{i+1}b_i^{-1}=u_{i}.$$

\textbf{Relations of $A(\Psi)$:}

We will check that the relations of the Artin group $A(\Psi)$ are satisfied by the Dehn twists and boundary swaps in $X_{\Map_o}$. Elements in $X_{\Map_o}$ which have representatives with disjoint support commute, and the relations between the Dehn twists are satisfied because of Lemma \ref{relDehnHalf}, while the relation between each $b_i$ and $b_{i+1}$ is satisfied because of Lemma \ref{lemswap1}, and the relation between $x_1$ and $b_1$ is satisfied because of Lemma \ref{lemswap2}. Relations between the $b_i$ and the $u_i$ are a consequence of the relations (Ci) and (Di), as we can check:$$b_iu_ib_iu_i=b_iu_iu_{i+1}b_i=b_iu_{i+1}u_ib_i=u_ib_iu_ib_i.$$

\textbf{Computing $R_1'$:}

We consider each relation in Theorem \ref{Teo-L-P} as an element of $\Map_o(S)$, and write it as a word over $X_{U}=\{ u_1,...,u_n\}$.

\begin{itemize}[leftmargin=*]

\item (R1): $\Delta^4(y_1,y_2,y_3,z)\Delta^{-2}
(x_0,y_1,y_2,y_3,z)$.

By Lemma \ref{PVRepresentations} (1) we have homomorphisms

\begin{center}
\begin{tikzpicture}
    \draw[->] (-1,0) -- node[above]{$\rho_2$} (1.5,0.35);

    \draw[->] (-1,1) -- node[above]{$\rho_1$} (1.5,0.65);
    
    \node at (-1.6,1) {$A(A_4)$};
    \node at (-1.6,0) {$A(A_5)$};
    \node at (4,0.5) {$\Map(S)\subset \Map_o(S)$};
\end{tikzpicture}
\end{center}

with images $$\rho_1 (\Delta^4(y_1,y_2,y_3,z))=a_1, \quad \rho_2(\Delta^2
(x_0,y_1,y_2,y_3,z))=c_1c_2,$$ where $a_1$, $c_1$ and $c_2$ are the Dehn twists in Figure \ref{Rel1}. But $c_2$ is a Dehn twist along a null-homotopic curve, and $a_1=c_1$, hence $$\rho_1(\Delta^{-4}(y_1,y_2,y_3,z))\rho_2(\Delta^2
(x_0,y_1,y_2,y_3,z))=\id$$ in $\Map_o(S)$ and we obtain the relation $\Delta^4(y_1,y_2,y_3,z)=\Delta^{2}
(x_0,y_1,y_2,y_3,z).$

\begin{figure}[H]
\includegraphics[width=11cm]{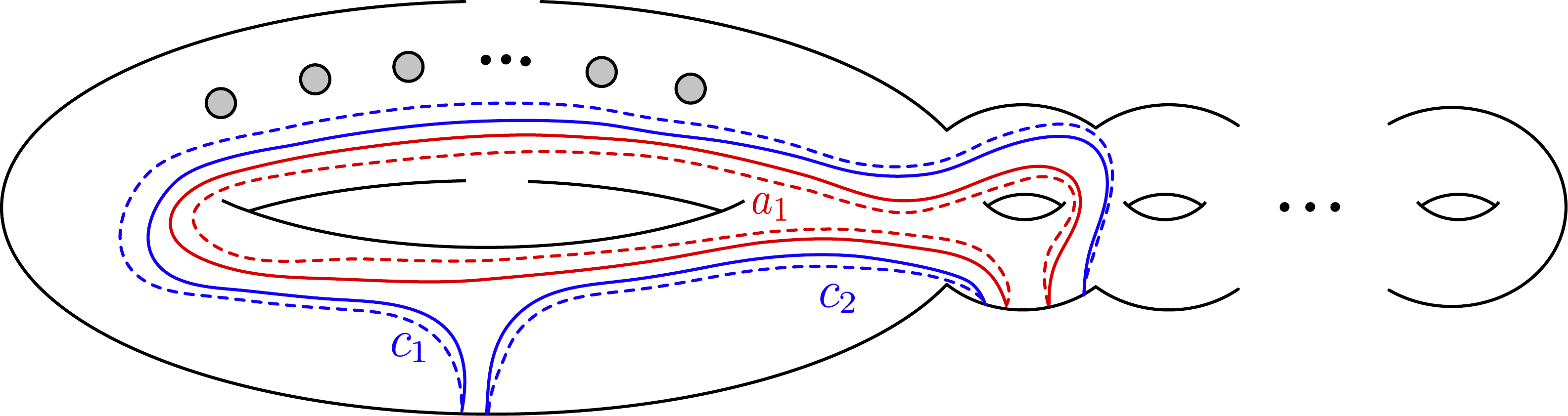}

\centering
\caption{}
\label{Rel1}
\end{figure}

\item (R2): $\Delta^2(y_1,y_2,y_3,y_4,y_5,z)
\Delta^{-1}(x_0,y_1,y_2,y_3,y_4,y_5,z)$.

By Lemma \ref{PVRepresentations} (4) we have homomorphisms

\begin{center}
\begin{tikzpicture}
    \draw[->] (-1,0) -- node[above]{$\rho_2$} (1.5,0.35);

    \draw[->] (-1,1) -- node[above]{$\rho_1$} (1.5,0.65);
    
    \node at (-1.6,1) {$A(E_7)$};
    \node at (-1.6,0) {$A(E_6)$};
    \node at (4,0.5) {$\Map(S)\subset \Map_o(S)$};
\end{tikzpicture}
\end{center}
with images $$\rho_{1}(\Delta^2(y_1,y_2,y_3,y_4,y_5,z))=a_1, \quad \rho_2(\Delta(x_0,y_1,y_2,y_3,y_4,y_5,z))=c_1c_2^2,$$ where $a_1$, $c_1$ and $c_2$ are the Dehn twists in Figure \ref{Rel2}. But $c_2$ is a Dehn twist along a null-homotopic curve, and $a_1=c_1$, hence $$\rho_{1}(\Delta(y_1,y_2,y_3,y_4,y_5,z))
\rho_{2}(\Delta^{-1}(x_0,y_1,y_2,y_3,y_4,y_5,z))=\id$$ in $\Map_o(S)$ and we obtain the relation $\Delta(y_1,y_2,y_3,y_4,y_5,z)=\Delta(x_0,y_1,y_2,y_3,y_4,y_5,z)).$

\begin{figure}[H]
\includegraphics[width=11cm]{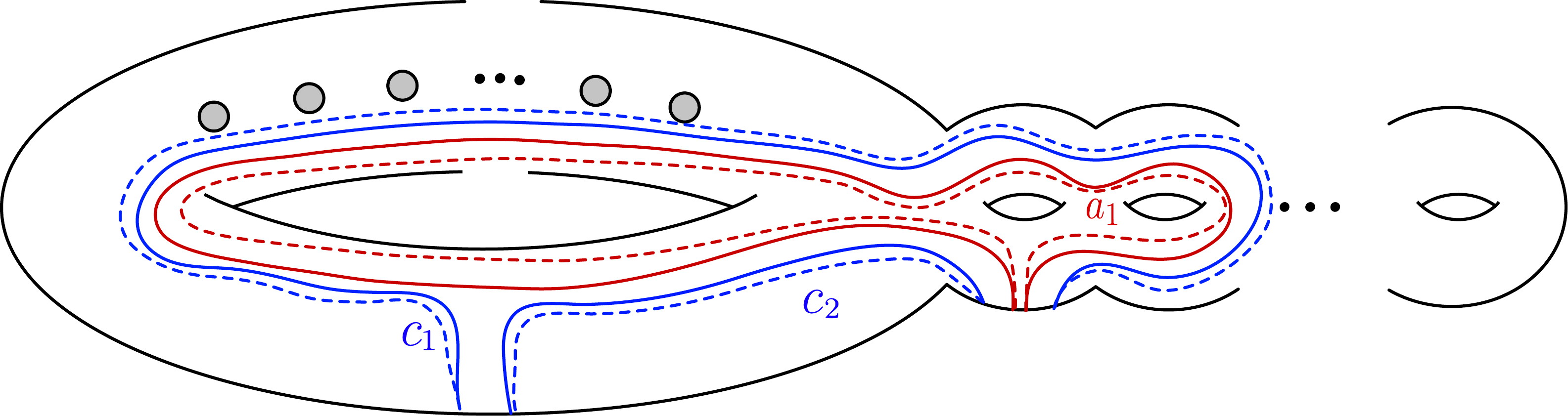}

\centering
\caption{}
\label{Rel2}
\end{figure}

\item (R3): $\Delta(x_0,x_1,y_1,h_1)
\Delta^{-2}(x_1,y_1,h_1)$.

By Lemma \ref{PVRepresentations2} (1) we have homomorphisms

\begin{center}
\begin{tikzpicture}
    \draw[->] (-1,0) -- node[above]{$\rho_2$} (1.5,0.35);

    \draw[->] (-1,1) -- node[above]{$\rho_1$} (1.5,0.65);
    
    \node at (-1.6,1) {$A(B_4)$};
    \node at (-1.6,0) {$A(B_3)$};
    \node at (2.75,0.5) {$ \Map_o(S)$};
\end{tikzpicture}
\end{center}
with images $$\rho_{1}(\Delta(x_0,x_1,y_1,b_1))=a_1, \quad \rho_2(\Delta^{2}(x_1,y_1,b_1))=c_1c_2,$$ where $a_1$, $c_1$ and $c_2$ are the Dehn twists in Figure \ref{Rel3}. The twist $c_1$ is trivial while $c_2=a_1$, hence $$\rho_{1}(\Delta(x_0,x_1,y_1,b_1))
\rho_{2}(\Delta^{-2}(x_1,y_1,b_1))=\id$$ in $\Map_o(S)$ and we obtain the relation $\Delta(x_0,x_1,y_1,b_1)=
\Delta^{2}(x_1,y_1,b_1).$

\begin{figure}[H]
\includegraphics[width=11cm]{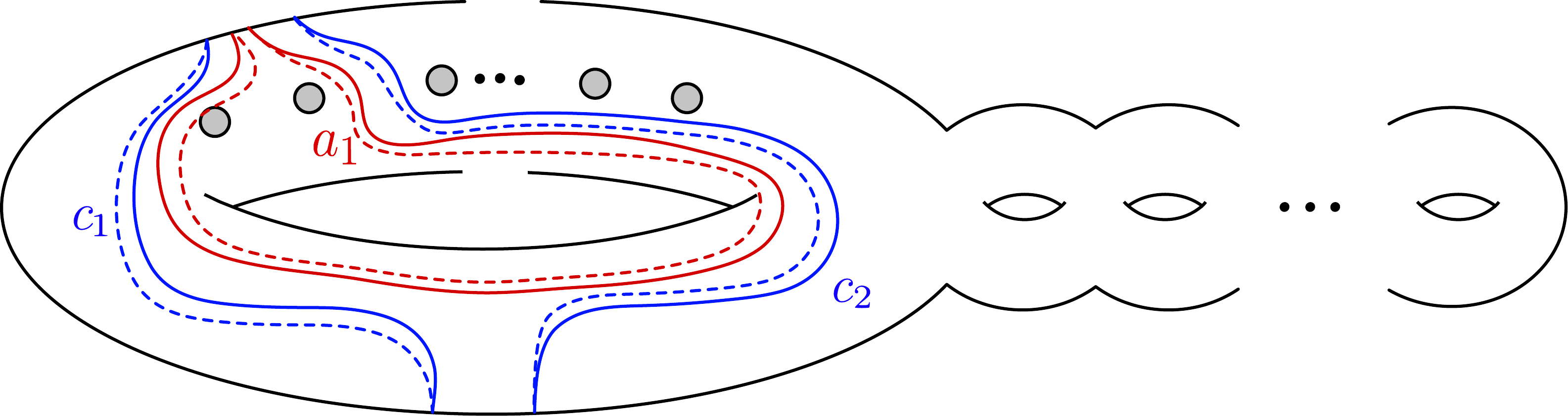}

\centering
\caption{}
\label{Rel3}
\end{figure}

\item (R4): $\Delta(x_0,x_1,y_1,y_2,y_3,z)
\Delta^{-2}(x_1,y_1,y_2,y_3,z)$.

By Lemma \ref{PVRepresentations} (3) we have homomorphisms

\begin{center}
\begin{tikzpicture}
    \draw[->] (-1,0) -- node[above]{$\rho_2$} (1.5,0.35);

    \draw[->] (-1,1) -- node[above]{$\rho_1$} (1.5,0.65);
    
    \node at (-1.6,1) {$A(D_6)$};
    \node at (-1.6,0) {$A(D_5)$};
    \node at (4,0.5) {$\Map(S)\subset \Map_o(S)$};
\end{tikzpicture}
\end{center}
with images $$\rho_{1}(\Delta(x_0,x_1,y_1,y_2,y_3,z))=a_2a_1^{2g+1}, \quad \rho_2(\Delta^2(x_1,y_1,y_2,y_3,z))=c_2c_3c_1^{2g+2},$$ where the $a_i$ and the $c_i$ are the Dehn twists in Figure \ref{Rel4}. The twist $c_2$ is trivial, $c_3=a_2$, and $c_1=a_1=u_1$, hence $$\rho_{1}(\Delta(x_0,x_1,y_1,y_2,y_3,z))
\rho_{2}(\Delta^{-2}(x_1,y_1,y_2,y_3,z))=u_1^{-1}$$ in $\Map_o(S)$ and we obtain the relation $\Delta(x_0,x_1,y_1,y_2,y_3,z)
u_1=\Delta^2(x_1,y_1,y_2,y_3,z).$

\begin{figure}[H]
\includegraphics[width=11cm]{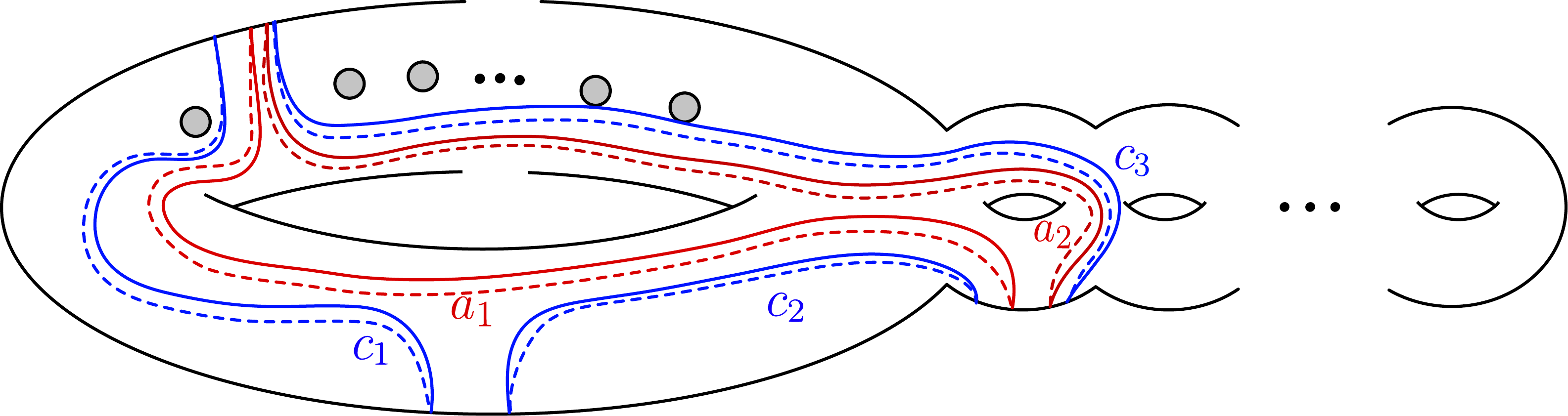}

\centering
\caption{}
\label{Rel4}
\end{figure}

\item (R5a): $ x_0^{2g-n-2}\Delta(x_1,h_1,\dots,
h_{n-1})\Delta^{-2}(z,y_2,\dots,y_{2g-1})$.

By Lemma \ref{PVRepresentations2} (3) and Lemma \ref{PVRepresentations} (3) we have homomorphisms

\begin{center}
\begin{tikzpicture}
    \draw[->] (-1,0) -- node[below]{$\rho_2$} (2,0);

    \draw[->] (-1,1) -- node[above]{$\rho_1$} (2,1);

    \node at (3.1,0.5) {$\cup$};
    \node at (-1.6,1) {$A(B_n)$};
    \node at (-1.9,0) {$A(D_{2g-1})$};
    \node at (3.1,1) {$\Map_o(S)$};
    \node at (3.1,0) {$\Map(S)$};
\end{tikzpicture}
\end{center} 

with images $$\rho_{1}(\Delta(x_1,b_1,\dots,
b_{n-1}))=a_1^{n-1}a_2, \quad \rho_2(\Delta^{2}(z,y_2,\dots,y_{2g-1}))=c_1^{2g-3}c_2,$$ where $a_1$, $a_2$, $c_1$ and $c_2$ are the Dehn twists in Figure \ref{Rel5}. The twists satisfy $c_1=a_1=x_0$, and $c_2=a_2$, hence $$x_0^{2g-n-2}\rho_1(\Delta(x_1,b_1,\dots,
b_{n-1}))\rho_2(\Delta^{-2}(z,y_2,\dots,y_{2g-1}))=\id$$ in $\Map_o(S)$ and we obtain the relation $ x_0^{2g-n-2}\Delta(x_1,b_1,\dots,
b_{n-1})=\Delta^{2}(z,y_2,\dots,y_{2g-1})$.

\begin{figure}[H]
\includegraphics[width=11cm]{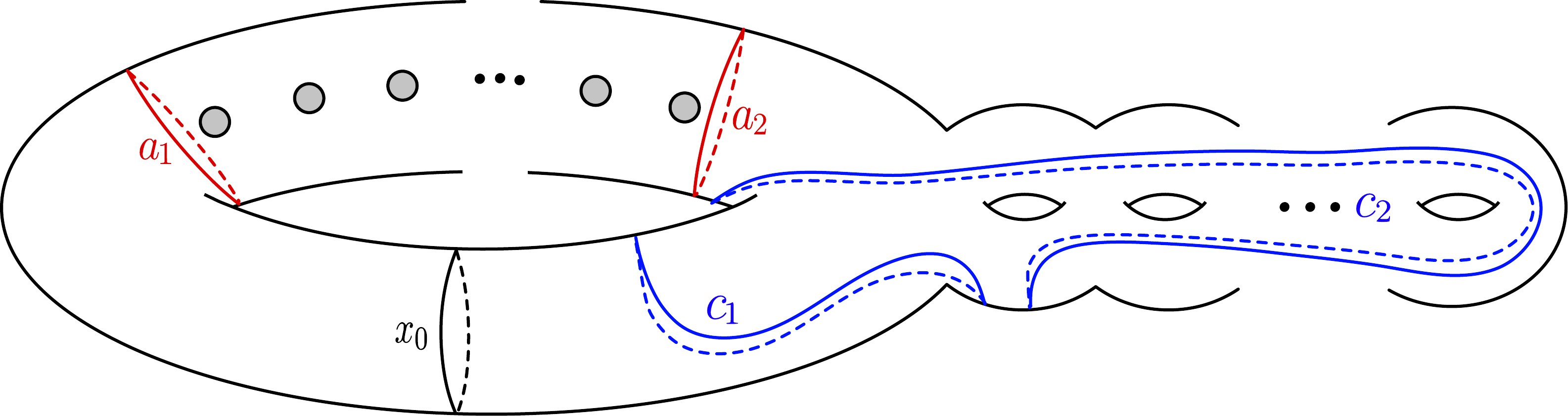}

\centering
\caption{}
\label{Rel5}
\end{figure}

\item (R5b): $x_0^n\Delta^{-1}(x_1,h_1,\dots, 
h_{n-1})$.

By Lemma \ref{PVRepresentations2} (3) we have a homomorphism

\begin{center}
\begin{tikzpicture}

    \draw[->] (-1,1) -- node[above]{$\rho$} (2,1);

    \node at (-1.6,1) {$A(B_n)$};

    \node at (3.1,1) {$\Map_o(S)$};

\end{tikzpicture}
\end{center} 

with image $$\rho(\Delta(x_1,b_1,\dots,
b_{n-1}))=a_1^{n-1}a_2,$$ where $a_1$ and $a_2$ are the Dehn twists in Figure \ref{Rel6}, which satisfy $a_1=a_2=x_0$, and hence 
$$\rho(\Delta(x_1,b_1,\dots,
b_{n-1}))=x_0^{n}$$ in $\Map_o(S)$, thus we obtain the relation $x_0^n=\Delta^{1}(x_1,b_1,\dots, 
b_{n-1})$.

\begin{figure}[H]
\includegraphics[width=8.5cm]{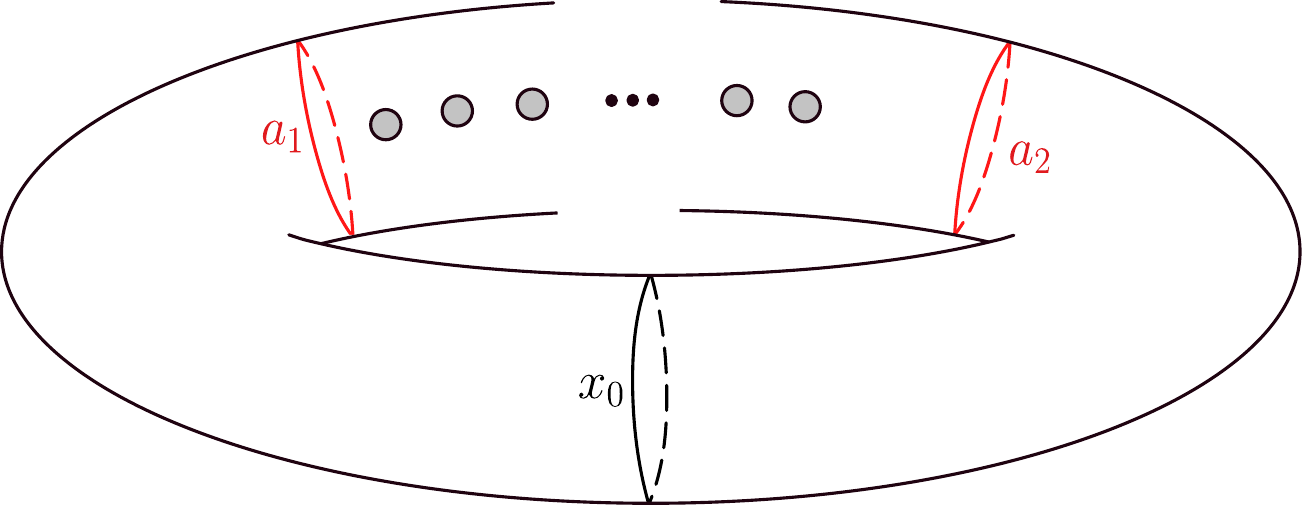}

\centering
\caption{}
\label{Rel6}
\end{figure}

\item (R5c): $\Delta^4 (x_0,y_1)\Delta^{-2}(h_1,\dots, 
h_{n-1})$. 

By Lemma \ref{PVRepresentations2} (2) and Lemma \ref{PVRepresentations} (1) we have homomorphisms

\begin{center}
\begin{tikzpicture}
    \draw[->] (-1,0) -- node[below]{$\rho_2$} (2,0);

    \draw[->] (-1,1) -- node[above]{$\rho_1$} (2,1);

    \node at (3.1,0.5) {$\cap$};
    \node at (-1.6,1) {$A(A_{2})$};
    \node at (-1.8,0) {$A(A_{n-1})$};
    \node at (3.1,1) {$\Map(S)$};
    \node at (3.1,0) {$\Map_o(S)$};
\end{tikzpicture}
\end{center} 

with images $$\rho_{1}(\Delta^4 (x_0,y_1))=a_1, \quad \rho_2(\Delta^{-2}(b_1,\dots, 
b_{n-1}))=c_1,$$ where $a_1$ and $c_1$ are the Dehn twists in Figure \ref{Rel7}, which satisfy $c_1=a_1$, and hence $$\rho_1(\Delta^4 (x_0,y_1))\rho_2(\Delta^{-2}(b_1,\dots, b_{n-1}))=\id$$ in $\Map_o(S)$, thus we obtain the relation $\Delta^4 (x_0,y_1)=\Delta^{2}(b_1,\dots, b_{n-1})$.

\begin{figure}[H]
\includegraphics[width=8.5cm]{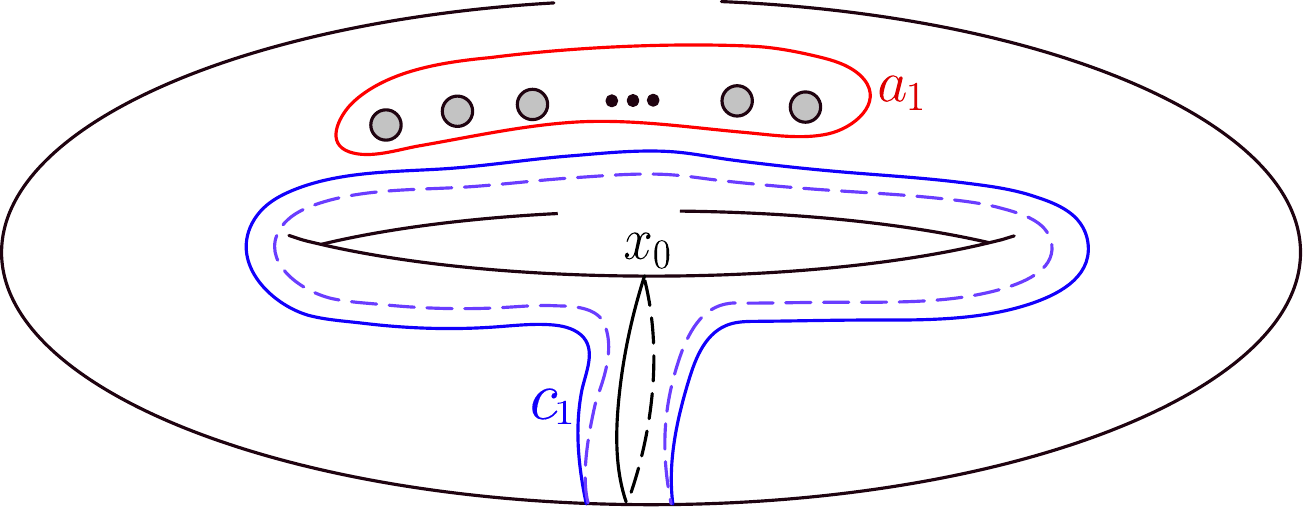}

\centering
\caption{}
\label{Rel7}
\end{figure}

We conclude that $\Map_o(S)$ can be expressed as the quotient $A(\Psi)$ by the sets of relations $R_2$ and $R_1'$, which finishes the proof.

\end{itemize}

\end{proof}

\section{Explicit presentation of $\B_{1,r}(O,Y)$}\label{Sect4}

Armed with the presentation of $\Map_o$ given by Theorem \ref{ThMain}, we will compute presentations of $\B_{1,r}(O, Y)$ for $r\geq 3$, and with $Y$ and $O$ tori, which are the content of Theorem \ref{Presentation1}. 

Because of Lemma \ref{fudom}, we can take as a fundamental domain $W_{\K_{1,r}^{3}(O,Y)}$ for the action of $\B_{1,r}(O, Y)$ on $\K$ the subcomplex: 

\begin{minipage}{0.4\textwidth}
  \centering
\begin{tikzpicture}
  \filldraw (0,0) circle (2pt) node[left] {$\v_0$};
  \filldraw (1,0.5) circle (2pt) node[right] {$\v_1$};
  \filldraw (0,1) circle (2pt) node[left] {$\v_2$};
  \filldraw (1,1.5) circle (2pt) node[right] {$\v_3$};

  \draw (0,0) --  (1,0.5);
  \draw (0,0) --  (0,1);
  \draw (1,0.5) --  (1,1.5);
  \draw (1,0.5) --  (0,1);
  \draw (0,1) --  (1,1.5);

\end{tikzpicture}
\end{minipage}
\hspace{-30pt}
\begin{minipage}{0.6\textwidth}
where $\v_i=[(S_{\v_i},\id)]$, for $S_{\v_i}$ the suited surface in Figure \ref{exp1}.  
\end{minipage}

\begin{figure}[H]  
\centering
\includegraphics[width=10cm]{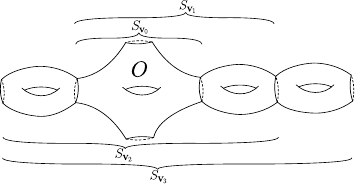}
\vspace{0.2cm}
\caption{A representation of the surfaces $S_{\v_i}$ for $r=4$. The union of $S_{\v_{i}}$ and a piece yields $S_{\v_{i+1}}$, while the union of $S_{\v_{i}}$ and two pieces results in $S_{\v_{i+2}}$. This determines the edges of $W_{\K_{1,4}^{3}(O,Y)}$. }
\label{exp1}
\end{figure}

The stabilizer $V_i$ of each vertex $\v_i\in W_{\K_{1,r}^{3}(O,Y)}$ has a presentation $V_i=\langle X_i|R_i\rangle$ given in Theorem \ref{ThMain}, and by Theorem \ref{presentationB} we know:
$$\B_{1,r}(O,Y)=\langle\cup X_i|(\cup R_i)\cup(\cup R_{i,j})\rangle, \textnormal{ whith  } \; R_{i,j}=\{o(s)t(s)^{-1}|s\in X_{i,j}\},$$
where each $o(s)$ is an expression of $s$ as a word in $X_i$, $t(s)$ is an expression of $s$ as a word in $X_j$, and the $X_{i,j}$ are sets of generators of the edge stabilizers. We will first give a generating set for $\B_{1,r}(O,Y)$, and then compute the sets $R_{i,j}$.

\subsection{Vertex stabilizers}\label{vertex0}
By Lemma \ref{typesstab}, the stabilizer $V_i$ of each vertex $\v_i$ is isomorphic to $\Map_o(S_{\v_i})$, which by Theorem \ref{ThMain} has a presentation $\langle X_i|R_i\rangle$ with 
$$X_0:=\{x_{0,0},x_{0,1},y_{0,1}, b_{0,1},b_{0,2},...,b_{0,r-1},u_{0,1},u_{0,2},...,u_{0,r}\},$$
$$X_i:=\{x_{i,0},x_{i,1},y_{i,1},y_{i,2},...,y_{i,2i+1}, z_i, b_{i,1},b_{i,2},...,b_{i,r-1},u_{i,1},u_{i,2},...,u_{i,r}\},\quad \textnormal{for }i\geq1,$$ and the relations described in Theorem \ref{ThMain}. Each $X_i$ is represented in Figure \ref{gen0}, with the curves defining Dehn twists and the arcs defining boundary swaps. The following is an immediate consequence of Theorem \ref{presentationB}:

\begin{cor}\label{GenB0}
The union $\underset{ 0\leq i\leq3}{\bigcup}{X_i}$ is a generating set of $\B_{1,r}(O, Y)$, where $Y$ and $O$ are tori and $r\geq3$.
    
\end{cor}

\begin{figure}[H]  
\centering
\includegraphics[width=14.5cm]{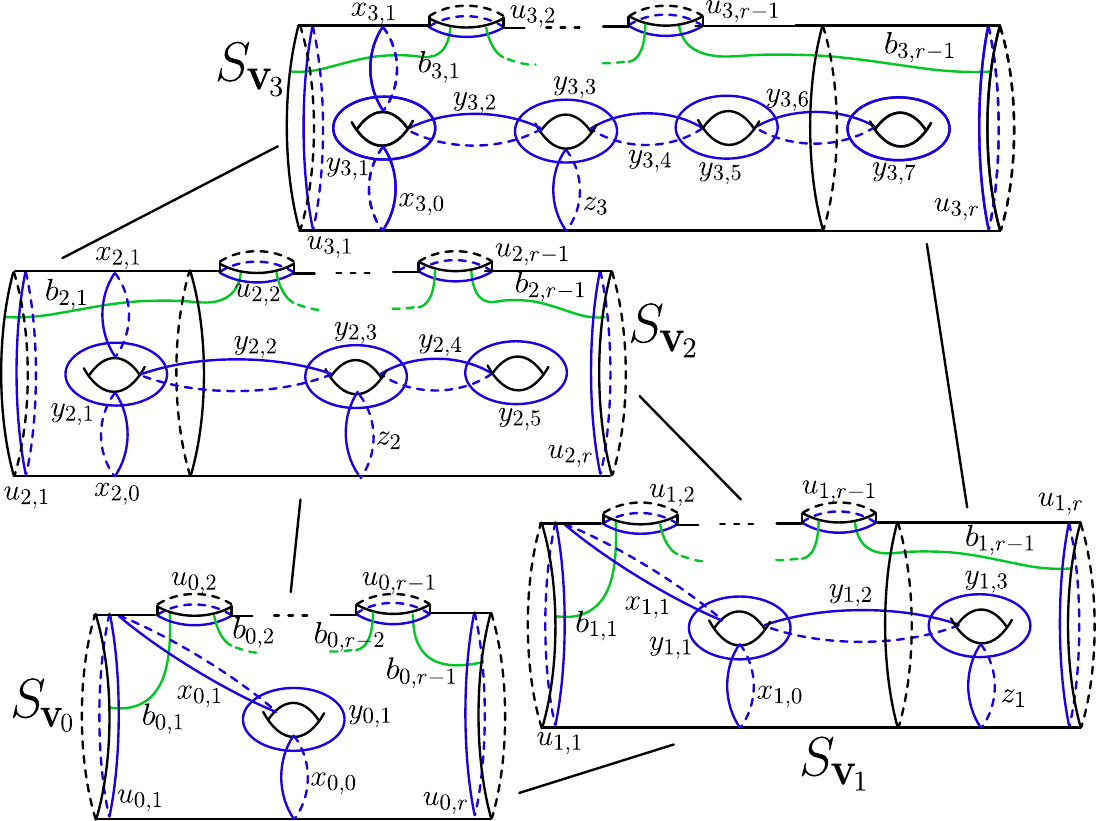}
\vspace{0.2cm}
\caption{Generators of $\B_{1,r}(O, Y)$. We represent each set $X_i$ as curves and arcs in $S_{\v_i}$.}
\label{gen0}
\end{figure}

It thus remains to give a complete set of relations, which are the ones afforded by Lemma \ref{presentationB}. To this end, we compute the sets $R_{i,j}$. 

\subsection{Edge stabilizers of the form $E_{i,i+1}$}\label{4.2} We will now compute the sets of relations $R_{i,i+1}$.

\subsubsection{Computing $R_{0,1}$}  By Lemma \ref{typesstab} (2) and Lemma \ref{lemSPMCG-A1}, the Dehn twists and boundary swaps $r_i$ in Figure \ref{1-20} are a set of generators of $E_{0,1}$. We express each of them as a word in $X_0$ and $X_1$ and hence determine the elements of $R_{0,1}$:

\hspace{45pt}\begin{minipage}{0.3\textwidth}
  \centering
  \begin{itemize}
      \item $r_1: x_{0,0}=x_{1,0},$
      \item $r_2: x_{0,1}=x_{1,1},$
      \item $r_3: y_{0,1}=y_{1,1}.$
  \end{itemize}

\end{minipage}
\hfill
\begin{minipage}{0.5\textwidth}
\begin{itemize}
\item For $j\in \{1,...,r-1\}$:
  \begin{center}
  \begin{itemize}
      \item[] $r_{u_j}: u_{0,j}=u_{1,j}.$
  \end{itemize}
  \end{center}
\item For $j\in \{1,...,r-2\}$:
  \begin{center}
  \begin{itemize}
      \item[] $r_{b_j}: b_{0,j}=b_{1,j}.$
  \end{itemize}
  \end{center}
\end{itemize}
  
\end{minipage}

\begin{figure}[H]
\includegraphics[height=3.5  cm]{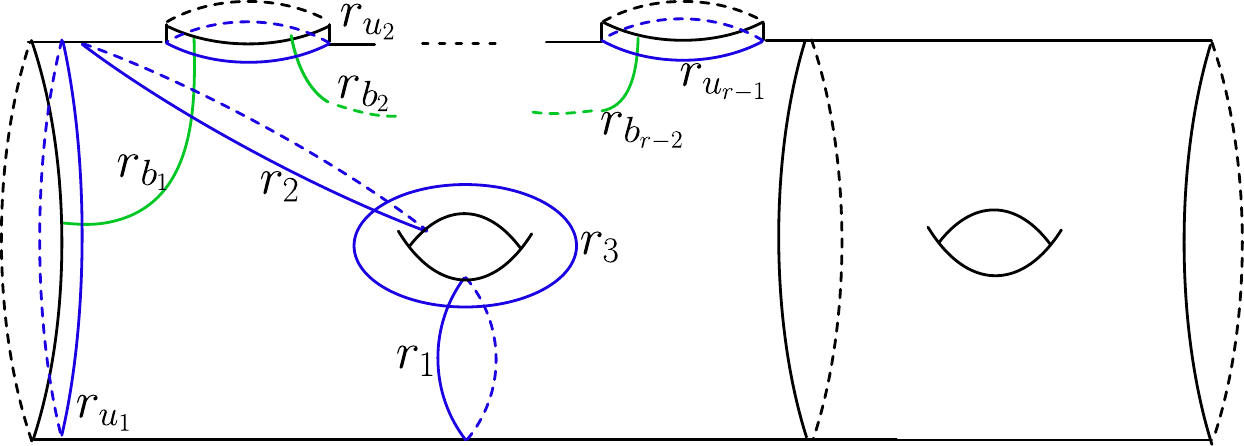}
\centering  
\caption{}
\label{1-20}
\end{figure}

\subsubsection{Computing $R_{1,2}$}\label{6.2.2} Because of Lemma \ref{typesstab} part (2) and Lemma \ref{lemSPMCG-A2}, the Dehn twists and boundary swaps $r_i$ in Figure \ref{1-20} are a set of generators of $E_{1,2}$. We express each of them as a word in $X_1$ and $X_2$ and hence determine $R_{1,2}$:

\hspace{45pt}\begin{minipage}{0.3\textwidth}
  \centering
  \begin{itemize}
      \item $r_1: x_{2,0}=z_{3},$
      \item $r_4: y_{2,1}=y_{3,3},$
      \item $r_5: y_{2,2}=y_{3,4},$
      \item $r_6: y_{2,3}=y_{3,5},$
  \end{itemize}

\end{minipage}
\hfill
\begin{minipage}{0.5\textwidth}
\begin{itemize}
\item For $j\in \{2,...,r\}$:
  \begin{center}
  \begin{itemize}
      \item[] $r_{u_j}: u_{1,j}=u_{2,j}.$
  \end{itemize}
  \end{center}
\item For $j\in \{2,...,r-1\}$:
  \begin{center}
  \begin{itemize}
      \item[] $r_{b_j}: b_{1,j}=b_{2,j}.$
  \end{itemize}
  \end{center}
\end{itemize}
  \end{minipage}

\begin{figure}[H]
\includegraphics[width=10cm]{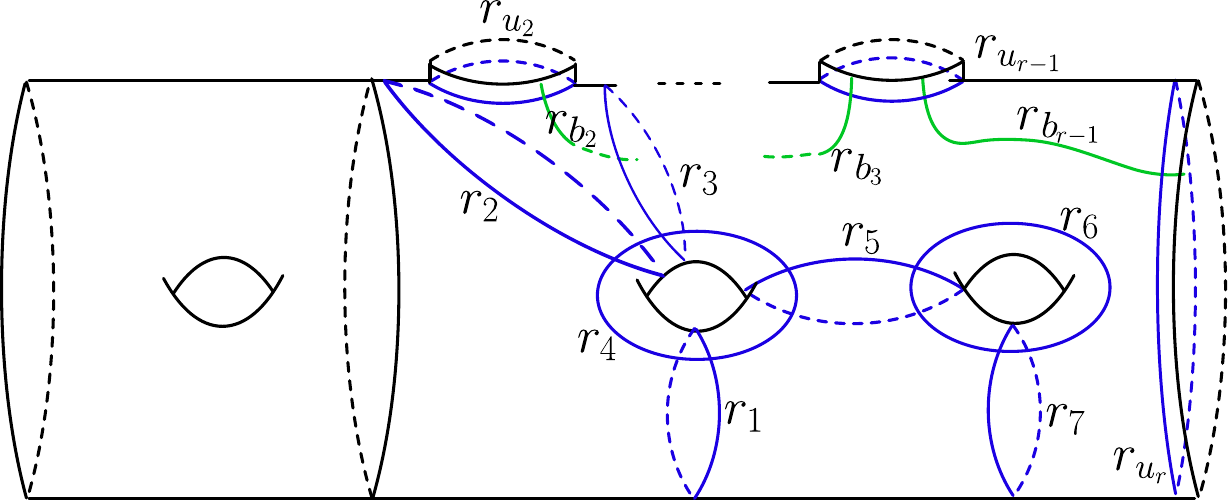}
\centering
\caption{}
\label{2-30}
\end{figure}

The Dehn twists $r_2$, $r_3$ and $r_7$ are not elements of $X_2$, hence we figure out expressions for them in terms of elements in $X_2$, in order to compute the corresponding relations.

\textbf{Computing $r_7$:} The twist $r_7$ is $z_1$ viewed as an element of $X_1$. We will write it as a product of elements in $X_2$. Let $\mathfrak{m}_i$ be the multitwist $$\mathfrak{m}_i=y_{i,5}y_{i,4}y_{i,3}z_{i}y_{i,2}y_{i,1}y_{i,3}y_{i,2}y_{i,4}y_{i,3}y_{i,5}y_{i,4}z_{i}y_{i,3}y_{i,2}y_{i,1}.$$
According to the calculations of Humphries in \cite[Proof of Theorem, page 45]{HumpGen}, one can obtain the Dehn twist $r_7$ by conjugating $x_{2,0}$ by $\mathfrak{m}$, yielding
$r_7=\mathfrak{m}_2x_{2,0}\mathfrak{m}_2^{-1}$. Hence we obtain the relation: $$r_7: z_1=\mathfrak{m}_2x_{2,0}\mathfrak{m}_2^{-1}.$$

\textbf{Computing $r_2$:} The twist $r_2$ is $x_{1,1}$ viewed as an element of $X_1$. On the other hand, using the homomorphism $\rho$ of $A(D_4)$ in Lemma \ref{PVRepresentations} (3) for the group generated by $\{x_{2,1},x_{2,0},y_{2,1},y_{2,2}\}$, we know that $\Delta(x_{2,1},x_{2,0},y_{2,1},y_{2,2})=u_{2,1}x_{1,1}z_2$, and we have the relation:
$$r_2:\Delta(x_{2,1},x_{2,0},y_{2,1},y_{2,2})=u_{2,1}x_{1,1}z_2.$$

\textbf{Computing $r_3$:} For $r_3$, by Lemmas \ref{PVRepresentations} (1) and \ref{PVRepresentations2} (3) applied to the groups generated by $\{x_{2,0},y_{2,1},y_{2,2}\}$ and $\{x_{1,1},b_{2,1}\}$, we deduce that \begin{equation}\label{eqr3}
    r_3=\Delta(x_{1,1},b_{2,1})\Delta^{-2}(x_{2,0},y_{2,1},y_{2,2})z_2.\end{equation} Given that $x_{1,1}$ can be expressed as a product of elements in $X_2$ because of relation $r_2$, and every other factor in Equation \ref{eqr3} is an element of $X_2$, there is no need to add $r_3$ to $R_{1,2}$ since $r_3$ can be expressed as a product of elements in $X_2$.

\subsubsection{Computing $R_{2,3}$} Using the same argument as for $R_{0,1}$ adapted to the the generators $X_2$ and $X_3$ given by Lemma \ref{lemSPMCG-A1}, we obtain:

    \begin{minipage}{0.3\textwidth}
  \centering
  \begin{itemize}[label=$\bullet$]
      \item $r_1: x_{2,0}=x_{3,0},$
      \item $r_2: x_{2,1}=x_{3,1},$
      \item $r_3: y_{2,1}=y_{3,1},$
      \item $r_4: y_{2,2}=y_{3,2},$
      \item $r_5: y_{2,3}=y_{3,3},$
  \end{itemize}

\end{minipage}
\hfill
\begin{minipage}{0.3\textwidth}
  \centering
  \begin{itemize}[label=$\bullet$]
      \item $r_6: y_{2,4}=y_{3,4},$
      \item $r_7: y_{2,5}=y_{3,5},$
      \item $r_{8}: z_{2}=z_3.$
  \end{itemize}

\end{minipage}
\hspace{-7pt}
\begin{minipage}{0.38\textwidth}
\begin{itemize}
\item For $j\in \{1,...,r-1\}$:
  \begin{center}
  \begin{itemize}
      \item[] $r_{u_j}: u_{2,j}=u_{3,j}.$
  \end{itemize}
  \end{center}
\item For $j\in \{1,...,r-2\}$:
  \begin{center}
  \begin{itemize}
      \item[] $r_{b_j}: b_{2,j}=b_{3,j}.$
  \end{itemize}
  \end{center}
\end{itemize}
\end{minipage}

\subsection{Edge stabilizers of the form $E_{i,i+2}$}\label{4.3}
By Lemma \ref{typesstab}, we know $E_{i,i+2}\simeq \Map_o^{\{A_1,A_2\}}(S_{\v_i})$. 

Given any element $\t\in  \Map_o^{\{A_1,A_2\}}(S_{\v_i})$ that swaps $A_1$ and $A_2$, we can find a set of generators $X_{i,i+2}$ of $E_{i,i+2}$ containing $\t$, and where every other element $x\in X_{i,i+2}$ fixes both $A_1$ and $A_2$, hence $x\in\Map_o^{A_1}(S_{\v_i})\cap\Map_o^{A_2}(S_{\v_i})$. Such a set of generators can be achieved by choosing any generating set $X$ containing $\t$, and substituting any generator $x\in X$ which swaps $A_1$ and $A_2$ for $x\t$. The following immediate inclusions: $$\Map_o^{A_1}(S_{\v_i})\cap\Map_o^{A_2}(S_{\v_i})\subset\Map_o^{A_1}(S_{\v_i})\simeq E_{i,i+1}$$ 

\begin{center}$\Map_o^{A_1}(S_{\v_i})\cap\Map_o^{A_2}(S_{\v_i})\subset\Map_o^{A_2}(S_{\v_i}\cup P_1)\simeq E_{i,i+1}$ 

\end{center} 
where $P_1$ is the piece adjacent to $S_{\v_i}$ containing $A_1$, imply that any relation in $R_{i,i+2}$ will be redundant given both $R_{i,i+1}$ and $R_{i+1,i+2}$, except for the relation coming from $\t$.

Hence it suffices to find a relation $r_{i,i+2}$ corresponding to an element that permutes the two pieces in $S_{\v_{i+2}}\setminus S_{\v_i}$. To this end, we will focus on finding an element of $\Map_o^{\{A_1,A_2\}}(S_{\v_i})$ that swaps $A_1$ and $A_2$.

A candidate to swapping $A_1$ and $A_2$ are $180^\circ$ rotations of $S_{\v_i}$. We will first determine how to write a $180^\circ$ rotation of the surface $S'$ with $n$ punctures but no boundary components as a product of Dehn twists and half twists; and then we will deduce a homomorphism of $S_{\v_i}$ that swaps $A_1$ and $A_2$.

\subsubsection{Rotation of $S'$.}\label{4.3.1} 

We will write the rotation $\phi\in \Map(S')$ of $180^\circ$ with axis the vertical line in  Figure \ref{fig:dos}(\subref{fig:dos-A}) as a product of Dehn twists and half twists, which in turn will allow us to find a homomorphism of $S_{\v_i}$ that swaps $A_1$ and $A_2$. The map $\phi$ fixes the curve $\alpha$ in Figure \ref{fig:dos}(\subref{fig:dos-A}), hence, according to \cite[Proposition 3.20]{Farb-Margalit}, we can compute separately:
\begin{enumerate}
    \item The $180^\circ$ rotation $\phi_D$ of $D$ in Figure \ref{fig:dos}(\subref{fig:dos-A}),
    \item The $180^\circ$ rotation $\phi_Q$ of $Q$ in Figure \ref{fig:dos}(\subref{fig:dos-A}),
\end{enumerate}
and then write 
\vspace{-5pt}$$\phi=\phi_D\phi_Q T_\a^k$$ for $T_\a$ the Dehn twist around the curve $\a$ in Figure \ref{fig:dos}(\subref{fig:dos-A}), and $k$ an integer.

\textbf{Rotation $\phi_Q$:} For $g\geq 2$ we use the homomorphism $\rho$ for $A_{2g+1}$ in Lemma \ref{PVRepresentations} (1) for the group generated by the Dehn twists $d_1,...,d_{2g+1}$ in Figure \ref{fig:dos}(\subref{fig:dos-A}). Because of \cite[Section 2, Lemma 2 (ii)]{Garside}, we know that $\Delta(d_1,...,d_{2g+1})=(d_1,...,d_{2g+1})(d_1,...,d_{2g})...(d_1d_2)d_1$ satisfies $$\Delta(d_1,...,d_{2g+1})d_i\Delta(d_1,...,d_{2g+1})^{-1}=d_{2g+2-i}.$$ By \cite[Fact 3.6, Fact 3.7]{Farb-Margalit}, this means that $\rho(\Delta(d_1,...,d_{2g+1}))$ takes the curve defining $d_i$ to the curve defining $d_{2g+2-i}$, and vice versa. Hence the composition $\phi_Q^{-1}\rho(\Delta(d_1,...,d_{2g+1}))$ fixes each curve defining the $d_i$. By the Alexander method (see \cite[Subsection 2.3]{Farb-Margalit}) this means that $\phi_Q^{-1}\rho(\Delta(d_1,...,d_{2g+1}))$ either swaps $Q_1$ and $Q_2$ in Figure \ref{fig:dos}(\subref{fig:dos-B}) or is the identity. But $Q_1$ has a puncture, which $Q_2$ does not, hence $\phi_Q^{-1}\rho(\Delta(d_1,...,d_{2g+1}))$ cannot swap them so we deduce $\phi_Q=\rho(\Delta(d_1,...,d_{2g+1}))$.

\begin{figure}[htbp]
\centering

\begin{subfigure}{0.48\textwidth}
    \centering
    \includegraphics[width=\linewidth]{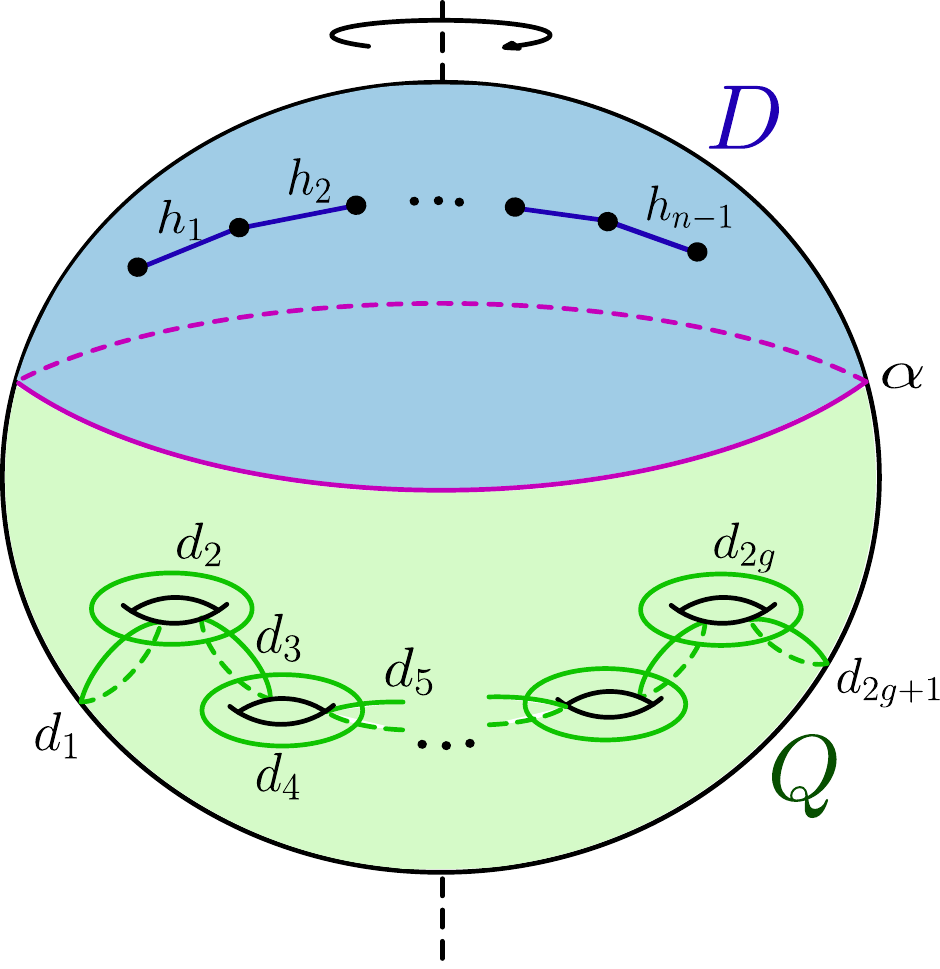}
    \caption{}
    \label{fig:dos-A}
\end{subfigure}
\hfill
\begin{subfigure}{0.48\textwidth}
    \centering
    
    \includegraphics[width=\linewidth]{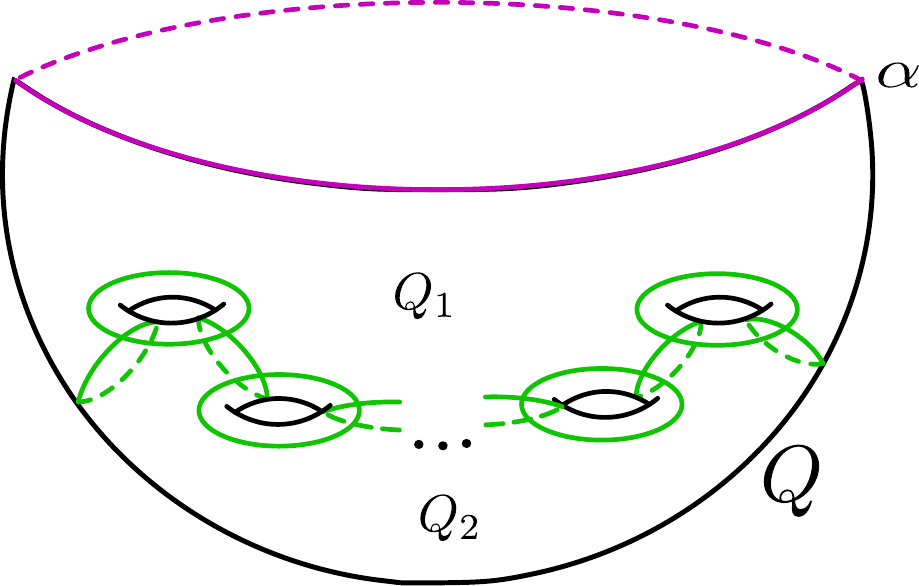}
    \vspace{6pt}
    \caption{}
    \label{fig:dos-B}
\end{subfigure}

\caption{}
\label{fig:dos}
\end{figure}

The particular case where $g=1$ differs in that there is no chain of curves as in Figure \ref{fig:dos}(\subref{fig:dos-B}) with an odd number of curves, since we only have $d_1$ and $d_2$. However, in this case it is easy to see that $\phi_Q=\rho(\Delta(d_1,d_2))^2$ is the desired rotation. 

\textbf{Rotation $\phi_D$:} We will use the homomorphism $\rho$ of $A_{n-1}$ in Lemma \ref{PVRepresentations} (5) for the group generated by the half twists $h_1,...,h_{n-1}$ in Figure \ref{fig:dos}(\subref{fig:dos-A}). Because of \cite[Section 2, Lemma 2 (ii)]{Garside}, we know that $$\Delta(h_1,...,h_{n-1})h_i\Delta(h_1,...,h_{n-1})^{-1}=h_{n-i}.$$

Half twists are uniquely defined by an arc, and the proof of \cite[Fact 3.7]{Farb-Margalit} works for half twists, hence $\rho(\Delta(h_1,...,h_{n-1}))$ takes the arc defining $h_i$ to the arc defining $h_{n-i}$, and vice versa.  Take $\phi_D$ the rotation of $D$ of $180^\circ$. The composition $\phi_D^{-1}\rho(\Delta(h_1,...,h_{n-1}))$ fixes each arc defining the $h_i$, by the Alexander method (see \cite[Subsection 2.3]{Farb-Margalit}) this implies that $\phi_D^{-1}\rho(\Delta(h_1,...,h_{n-1}))=\id$, and $\phi_D=\rho\Delta(h_1,...,h_{n-1}))$.

We  now compute the rotation of $S'$ for $g\geq 2$. Take the product $$\phi_Q^{-1}\phi_D=\rho(\Delta(d_1,...,d_{2g+1}))^{-1}\rho(\Delta(h_1,...,h_{n-1})).$$ We know that $\phi=\phi_Q^{-1}\phi_DT_\a^{k}$ for some $k$ and we also know that $\phi^2=\id$. Because of Subsection \ref{2.3}, we know that $\rho(\Delta(d_1,...,d_{2g+1}))^2=T_\a=\rho(\Delta(h_1,...,h_{n-1}))^2$, hence $$\id=(\phi_Q^{-1}\phi_DT_\a^{k})^2=\rho(\Delta(d_1,...,d_{2g+1}))^{-2}\rho(\Delta(h_1,...,h_{n-1}))^2T_\a^{2k}=T_\a^{-1}T_\a^1T_\a^{2k},$$ which means $k=0$ and $$\phi=\rho(\Delta(d_1,...,d_{2g+1}))^{-1}\rho(\Delta(h_1,...,h_{n-1})).$$

For the case where $g=1$, an analogous reasoning yields that $$\phi=\rho(\Delta(d_1,d_{2}))^{-2}\rho(\Delta(h_1,...,h_{n-1})).$$

We now explicitly compute the relations derived from the rotations.

\subsubsection{Computing $r_{0,2}$} The product $$\phi_0=\rho(\Delta(x_{0,0}, y_{0,1}))^{-2}\rho(\Delta(b_{0,1},b_{0,2},...,b_{0,r-1}))$$ defines an element of $\Map_o(S_{\v_0})$, which is not a $180^\circ$ rotation since $$\phi_0^2=\rho(\Delta(x_{0,0},y_{0,1}))^{-4}\rho(\Delta(b_{0,1},b_{0,2},...,b_{0,r-1}))^2=u_{0,1}^{-1}u_{0,2}^{-1}...u_{0,r}^{-1}\neq \id$$ because of Lemma \ref{PVRepresentations} (1) and (5). Since $\In(\phi_0)$ is a $180^\circ$ rotation of $S'_{\v_0}$ (the result of capping $S_{\v_0}$), it swaps the punctures coming from $A_1$ and $A_2$, thus $\phi_0$ exchanges $A_1$ and $A_2$ in Figure \ref{tubo0}.

The product $$\phi_2=\rho(\Delta(x_{2,0},y_{2,1},y_{2,2},y_{2,3},y_{2,4},y_{2,5},z_1))^{-1}\rho(\Delta(b_{2,1},b_{2,2},...,b_{2,r-1}))$$ defines an element of $\Map^o(S_{\v_2})$, which also exchanges $A_1$ and $A_2$. Both $\phi_0$ and $\phi_2$ are expressed as a product of generators in $\underset{ 0\leq i\leq3}{\bigcup}{X_i}$.

\vspace{-0.2cm}
\begin{figure}[H]  
\centering
\includegraphics[width=11.5cm]{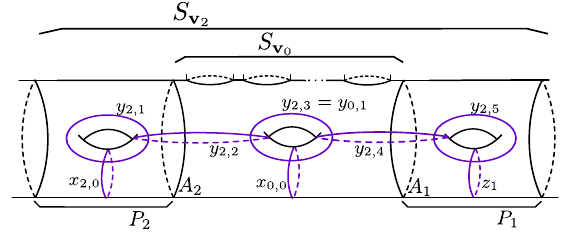}
\vspace{-0.2cm}
\caption{}
\label{tubo0}
\end{figure}

The maps $\phi_0$ and $\phi_2$ can be extended to elements of $\B_{d,r}(O, Y)$ that act rigidly away from $S_{\v_2}$. Because of how we defined $\phi_0$ and $\phi_2$ as preimages through $\In(\cdot)$ of the $180^\circ$ rotations, their extensions must define the same elements up to Dehn twisting around the boundary components, hence
 \begin{equation}\label{eq0}
    \phi_2u_{2,1}^{k_1}u_{2,2}^{k_2}...u_{2,r}^{k_r}=\phi_0u_{0,1}^{l_1}u_{0,2}^{l_2}...u_{0,r}^{l_r}.\end{equation} for some integers $l_1,...,l_r$ and $k_1,...,k_r$. Taking the squares yields:

\begin{equation}\label{eq02}
\begin{split}
\rho(\Delta(x_{0,0},y_{0,1}))^{-4}\rho(\Delta(b_{0,1},b_{0,2},...,b_{0,r-1}))^2u_{0,1}^{l_1+l_r}u_{0,2}^{l_2+l_{r-1}}...u_{0,r}^{l_r+l_1}=\\\rho(\Delta(x_{2,0},y_{2,1},y_{2,2},y_{2,3},y_{2,4},y_{2,5},z_1))^{-2}\rho(\Delta(b_{2,1},b_{2,2},...,b_{2,r-1}))^2u_{2,1}^{k_1+k_r}...u_{2,r}^{k_r+k_1}, 
\end{split}
\end{equation}
which by Lemmas \ref{PVRepresentations} (1) and \ref{PVRepresentations2} (2) can be rewritten as
$$ T_{\a_0}^{-1}T_{\a_0}u_{0,1}^{-1}...u_{0,r}^{-1}u_{0,1}^{l_1+l_r}u_{0,2}^{l_2+l_{r-1}}...u_{0,r}^{l_r+l_1}=T_{\a_2}^{-1}T_{\a_2}u_{2,1}^{-1}...u_{2,r}^{-1}u_{2,1}^{k_1+k_r}u_{2,2}^{k_2+k_{r-1}}...u_{2,r}^{k_r+k_1}.$$

From this, we can deduce a set of sufficient conditions for \eqref{eq02} to be satisfied: we know that, for $j\in \{2,r-1\}$, we have $u_{0,j}=u_{2,j}$, hence if  \begin{itemize}
    \item $l_1=k_1=1$,
    \item $l_2=l_3=...=l_r=k_2=...=k_r=0$
\end{itemize}
then Equation \ref{eq02} is satisfied. 

These conditions are in fact also sufficient for \eqref{eq0} to be satisfied, as for those values we get $$(\phi_0u_{0,1})^2|_{\S_{d,r}(O, Y)\setminus S_{\v_0}}=\id=(\phi_2u_{2,1})^2|_{\S_{d,r}(O, Y)\setminus S_{\v_0}},$$ which means that $$\phi_0u_{0,1}(P_1)=P_2, \quad\phi_0u_{0,1}(P_2)=P_1,$$ 
hence $\phi_0u_{0,1}(\cdot)$ defines a homeomorphism that swaps $P_1$ and $P_2$, and $$\phi_2u_{2,1}(P_1)=P_2, \quad \phi_2u_{2,1}(P_2)=P_1,$$ hence $\phi_2u_{2,1}(\cdot)$ defines a homeomorphism that swaps $P_1$ and $P_2$. 

Recall that the preferred rigid structure of Definition \ref{rigstru} was arbitrarily chosen, and in fact the choice has not affected any of the arguments in the paper up to this point. We can then retroactively set the preferred rigid structure so that $\phi_0u_{0,1}=\phi_2u_{2,1}$. This yields the following relation: $$r_{1,3}:\phi_0u_{0,1}=\phi_2u_{2,1},
$$ with 
\vspace{-7pt}$$
\phi_0=\rho(\Delta(x_{0,0},y_{0,1}))^{-2}\rho(\Delta(b_{0,1},b_{0,2},...,b_{0,r-1})),
$$$$
\phi_2=\rho(\Delta(x_{2,0},y_{2,1},y_{2,2},y_{2,3},y_{2,4},y_{2,5},z_1))^{-1}\rho(\Delta(b_{2,1},b_{2,2},...,b_{2,r-1})).$$

\begin{rem}
    In fact, any presentation of $\S_{d,r}(O,Y)$ fixes a preferred rigid structure, since relations between generators determine how the pieces join together.
\end{rem}

\subsubsection{Computing $r_{1,3}$} Using the Humphries calculations in \cite[Proof of Theorem, page 45]{HumpGen}, set $t_3$ to be: $$t_3=\mathfrak{n}_3z_{3}\mathfrak{n}_3^{-1},$$ with 
\vspace{-18pt}\begin{align}
     \mathfrak{n}_3= y_{3,7}y_{3,6}y_{3,5}z_{1}y_{3,4}y_{3,3}y_{3,5}y_{3,4}\cdot\nonumber\\\cdot y_{3,6}y_{3,5}y_{3,7}y_{3,6}z_{1}y_{3,5}y_{3,4}y_{3,3}\nonumber.
    \end{align}

Define: 
\vspace{-14pt}$$\phi_1=\Delta(x_{1,0},y_{1,1},...,y_{1,3},z_1)^{-1}\Delta(b_{1,1},...,b_{1,r}),$$
$$\phi_3=\Delta(x_{3,0},y_{3,1},...,y_{3,7},t_3)^{-1}\Delta(b_{3,1},...,b_{3,r}).$$
By an analogous reasoning as for $r_{0,2}$, we get the following relation: $$r_{1,3}:\phi_{1}u_{1,1}=\phi_3u_{3,1}.$$

\subsection{The presentation} We collect the generators in Subsection \ref{vertex0} and every set of relations in Subsections \ref{4.2} and \ref{4.3}, and Theorem \ref{ThMain}, which gives rise to the following presentation of $\B_{1,r}(O, Y)$.

\begin{theor}\label{Presentation1}
 Let $O$ and $Y$ be tori. The group $\B_{1,r}(O, Y)$ has a presentation with generators the Dehn twists $$\left(\overset{3}{\underset{i=1}{\cup}}\{x_{i,0},x_{i,1},y_{i,1},y_{i,2},...,y_{i,2i+1}, z_i, u_{i,1},u_{i,2},...,u_{i,r}\}\right){\cup}\{x_{0,0},x_{0,1},y_{0,1}, u_{0,1},u_{0,2},...,u_{0,r}\}$$ and the boundary swaps $$\overset{3}{\underset{i=0}{\cup}}\{ b_{i,1},b_{i,2},...,b_{i,r-1},\},$$ and the relations are:

 \begin{enumerate}
     \item Relations from the vertex stabilizer $V_0$:

\begin{itemize}[leftmargin=0pt]
\item Braid relations.  Let  $(a,b)\in \{(x_{0,1},y_{0,1}),(x_{0,0},y_{0,1}),\\(b_{0,1},b_{0,2}),...,(b_{0,r-2},b_{0,r-1})\}$:

\hspace{2pt} \textnormal{(A$_{a,b}$)}\quad  $aba=bab$.

\item 4-length relations. Let  $(a,b)\in \{(x_{0,1},b_{0,1}),(u_{0,1},b_{0,1}),...,(u_{0,r-1},b_{0,r-1}),\\(u_{0,2},b_{0,1}),...,(u_{0,r},b_{0,r-1})\}$:

\hspace{2pt}\textnormal{(A$_{a,b}$)} \quad  $abab=baba$.

\item Commutation. Let  $(a,b)$ be a pair of generators in $X_0$ that do not satisfy the braid or 4-length relations. Then:

\hspace{2pt} \textnormal{(A$_{a,b}$)}\quad  $ab=ba$.

\item Other relations.     
    
\begin{tabular}{p{30pt} c l}

\rm (R03) & $(x_{0,0}x_{0,1}y_{0,1}b_{0,1})^4=
(x_{0,1}y_{0,1}b_{0,1})^6$,\\

\rm (R05b) & $x_0^r=(x_{0,1}b_{0,1}\dots b_{0,r-1})^r$,\\

\rm (R05c) & $ (x_{0,0}y_{0,1})^8=(b_1\dots 
b_{r-1})^r$,\\
 
\rule{0pt}{13pt}\makebox[0pt][l] { For  $j\in\{1,...,r-1\}$:} \\
\rm (C01j) \quad & $b_{0,j}u_{0,j}=u_{0,j+1}b_{0,j},$ & \quad\\
\rm (C02j) \quad & $u_{0,j}b_{0,j}=b_{0,j}u_{0,j+1}.$ & \quad\\
\end{tabular}
\end{itemize}

\item Relations from the vertex stabilizer $V_i$, for $i\in \{1,2,3\}$:

\begin{itemize}[leftmargin=0pt]
\item Braid relations.  Let  $(a,b)\in \{(x_{i,1},y_{i,1}),(x_{i,0},y_{i,1}),(z_{i},y_{i,3}),\\(b_{i,1},b_{i,2}),...,(b_{i,r-2},b_{i,r-1}),(y_{i,1},y_{i,2}),...,(y_{i,2i},y_{i,2i+1})\}$:

\hspace{2pt} \textnormal{(A$_{a,b}$)}\quad  $aba=bab$.

\item 4-length relations. Let  $(a,b)\in \{(x_{i,1},b_{i,1}),(u_{i,1},b_{i,1}),...,(u_{i,r-1},b_{i,r-i}),\\(u_{i,2},b_{i,1}),...,(u_{i,r},b_{i,r-1})\}$:

\hspace{2pt} \textnormal{(A$_{a,b}$)}\quad  $abab=baba$.

\item Commutation. Let  $(a,b)$ be a pair of generators in $X_i$ that do not satisfy the braid or 4-length relations. Then:

\hspace{2pt} \textnormal{(A$_{a,b}$)}\quad  $ab=ba$.

\item Other relations.
    
\begin{tabular}{p{30pt} c l}

\rm (Ri1) \quad & $(y_{i,1}y_{i,2}y_{i,3}z_{i})^{10}=
(x_{i,0}y_{i,1}y_{i,2}y_{i,3}z_{i})^6$,\\

\rm (Ri2) \quad & $(y_{i,1}y_{i,2}y_{i,3}y_{i,4}y_{i,5}z_{i})^{12}=
(x_{i,0}y_{i,1}y_{i,2}y_{i,3}y_{i,4}y_{i,5}z_{i})^{15}$,\\

\rm (Ri3) \quad & $(x_{i,0}x_{i,1}y_{i,1}b_{i,1})^4=
(x_{i,1}y_{i,1}b_{i,1})^6$,\\

\rm (Ri4) \quad & $u_{i,1}(x_{i,0}x_{i,1}y_{i,1}y_{i,2}y_{i,3}z_{i})^5=
(x_{i,1}y_{i,1}y_{i,2}y_{i,3}z_{i})^8$,\\

\rm (Ri5a) \quad & $ x_{i,0}^{2i-r}(x_{i,1}b_{i,1}\dots
b_{i,r-1})^{r}=(z_{i}y_{i,2}\dots y_{i,2i+1})^{4i}$,\\
 
\rule{0pt}{13pt}\makebox[0pt][l] { For  $j\in\{1,...,r-1\}$:} \\
\rm (Ci1j) \quad & $b_{i,j}u_{i,j}=u_{i,j+1}b_{i,j},$ & \quad\\
\rm (Ci2j) \quad & $u_{i,j}b_{i,j}=b_{i,j}u_{i,j+1}.$ & \quad\\
\end{tabular}

\end{itemize}

\item Relations from edge stabilizers of the form $E_{i,i+1}$:

\hspace{0pt}\begin{minipage}{0.45\textwidth}
  \centering
  \begin{itemize}[label={},leftmargin=*]
      \item \textnormal{(S1)}$\quad x_{0,0}=x_{1,0},$
      \item \textnormal{(S2)}$\quad x_{0,1}=x_{1,1},$
      \item \textnormal{(S3)}$\quad y_{0,1}=y_{1,1},$
  \end{itemize}

\end{minipage}
\hfill
\begin{minipage}{0.55\textwidth}
For $j\in \{1,...,r-1\}$:
  \begin{center}
  \begin{itemize}[label={},leftmargin=*]
      \item \textnormal{(S4j)}$\quad u_{0,j}=u_{1,j},$
  \end{itemize}
  \end{center}
For $j\in \{1,...,r-2\}$:
  \begin{center}
  \begin{itemize}[label={},leftmargin=*]
      \item \textnormal{(S5j)}$\quad b_{0,j}=b_{1,j},$
  \end{itemize}
  \end{center}

\end{minipage}


\hspace{0pt}\begin{minipage}{0.45\textwidth}
  \centering
  \begin{itemize}[label={},leftmargin=*]
      \item \textnormal{(S6)}$\quad x_{2,0}=z_{3},$
      \item \textnormal{(S7)}$\quad y_{2,1}=y_{3,3},$
      \item \textnormal{(S8)}$\quad y_{2,2}=y_{3,4},$
      \item \textnormal{(S9)}$\quad y_{2,3}=y_{3,5},$
  \end{itemize}

\end{minipage}
\hfill
\begin{minipage}{0.55\textwidth}
For $j\in \{2,...,r\}$:
  \begin{center}
  \begin{itemize}[label={},leftmargin=*]
      \item \textnormal{(S10j)}$\quad u_{1,j}=u_{2,j},$
  \end{itemize}
  \end{center}
For $j\in \{2,...,r-1\}$:
  \begin{center}
  \begin{itemize}[label={},leftmargin=*]
      \item \textnormal{(S11j)}$\quad b_{1,j}=b_{2,j},$
  \end{itemize}
  \end{center}
  \end{minipage}

$$  \textnormal{(S12)}\quad (x_{2,1}x_{2,0}y_{2,1}y_{2,2})^3=u_{2,1}x_{1,1}z_2,$$\vspace{-0.63cm}
$$\textnormal{(S13)}\quad z_1=\mathfrak{m}_2x_{2,0}\mathfrak{m}_2^{-1},$$


 \hspace{-30pt}\begin{minipage}{0.3\textwidth}
  \centering
  \begin{itemize}[label={},leftmargin=*]
      \item $\textnormal{(S14)}\quad x_{2,0}=x_{3,0},$
      \item $\textnormal{(S15)}\quad x_{2,1}=x_{3,1},$
      \item $\textnormal{(S16)}\quad y_{2,1}=y_{3,1},$
      \item $\textnormal{(S17)}\quad y_{2,2}=y_{3,2},$
  \end{itemize}

\end{minipage}
\hfill
\begin{minipage}{0.3\textwidth}
  \centering
  \begin{itemize}[label={},leftmargin=*]
      \item $\textnormal{(S18)}\quad y_{2,3}=y_{3,3},$
      \item $\textnormal{(S19)}\quad y_{2,4}=y_{3,4},$
      \item $\textnormal{(S20)}\quad y_{2,5}=y_{3,5},$
      \item $\textnormal{(S21)}\quad z_{2}=z_3,$
  \end{itemize}

\end{minipage}
\hfill
\begin{minipage}{0.3\textwidth}
For $j\in \{1,...,r-1\}$:
  \begin{center}
  \begin{itemize}[label={},leftmargin=*]
      \item $\textnormal{(S22j)}\quad u_{2,j}=u_{3,j}.$
  \end{itemize}
  \end{center}
For $j\in \{1,...,r-2\}$:
  \begin{center}
  \begin{itemize}[label={},leftmargin=*]
      \item $\textnormal{(S23j)}\quad b_{2,j}=b_{3,j}.$
  \end{itemize}
  \end{center}
  
\end{minipage}


\item Relations from edge stabilizers of the form $E_{i,i+2}$: $$\textnormal{(T1)}\quad \phi_0u_{0,1}=\phi_2u_{2,1},\quad \quad \textnormal{(T2)}\quad \phi_{1}u_{1,1}=\phi_3u_{3,1}.$$

 \end{enumerate}
\end{theor}

\section{Explicit presentation of $B_{2,1}(O,Y)$}\label{Sect5}

Next, we compute a presentation of $\B_{2,1}(O, Y)$  with $Y$ and $O$ tori, given in Theorem \ref{Presentation2}. The strategy follows the same spirit as that of Section \ref{Sect4}.

Because of Lemma \ref{fudom}, we can take as a fundamental domain $W_{\K_{2,1}^{6}(O,Y)}$ for the action of $\B_{2,1}(O, Y)$ on $\K$ the subcomplex: 

\begin{minipage}{0.7\textwidth}
  \centering
\begin{tikzpicture}
  \filldraw (0,0) circle (2pt) node[left] {$\mathbf{v}_1$};
  \filldraw (1,0.5) circle (2pt) node[right] {$\mathbf{v}_2$};
  \filldraw (0,1) circle (2pt) node[left] {$\mathbf{v}_3$};
  \filldraw (1,1.5) circle (2pt) node[right] {$\mathbf{v}_4$};
  \filldraw (0,2) circle (2pt) node[left] {$\mathbf{v}_5$};
  \filldraw (1,2.5) circle (2pt) node[right] {$\mathbf{v}_6$};

  \draw (0,0) --  (1,0.5);
  \draw (0,0) --  (0,1);
  \draw (1,0.5) --  (1,1.5);
  \draw (1,0.5) --  (0,1);
  \draw (0,1) --  (0,2);
  \draw (0,1) --  (1,1.5);
  \draw (1,1.5) --  (1,2.5);
  \draw (1,1.5) --  (0,2);
  \draw (0,2) --  (1,2.5);

\end{tikzpicture}
\end{minipage}
\hspace{-0.25\textwidth}
\begin{minipage}{0.3\textwidth}
where $\v_i=[(S_{\v_i},\id)]$, for $S_{\v_i}$ the suited surface in Figure \ref{exp2}.
\end{minipage}

\begin{figure}[H]  
\centering
\includegraphics[width=13.5cm]{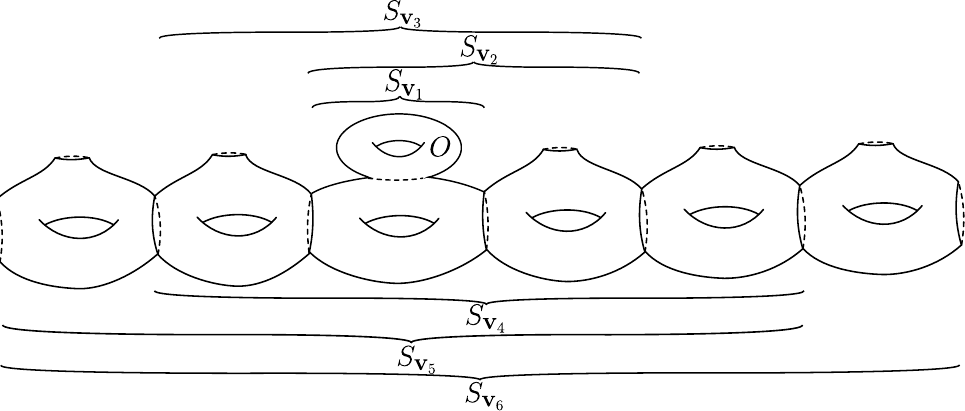}
\vspace{0.2cm}
\caption{The union of $S_{\mathbf{v}_{i}}$ and a piece yields $S_{\mathbf{v}_{i+1}}$, while the union of $S_{\mathbf{v}_{i}}$ and two pieces results in $S_{\mathbf{v}_{i+2}}$. This determines the edges of $W_{\K_{2,1}^{6}(O,Y)}$. }
\label{exp2}
\end{figure}

As in Section \ref{Sect4}, by Theorem \ref{presentationB} we know:
$$\B_{2,1}(O,Y)=\langle\cup X_i|(\cup R_i)\cup(\cup R_{i,j})\rangle, \textnormal{ whith  } \; R_{i,j}=\{o(s)t(s)^{-1}|s\in X_{i,j}\},$$
where each $o(s)$ is an expression of $s$ as a word in $X_i$, $t(s)$ is an expression of $s$ as a word in $X_j$, the $X_{i,j}$ are sets of generators of the edge stabilizers, and each $V_i=\langle X_i|R_i\rangle$.

The calculations in this Section differ from the star surface case (Section \ref{Sect4}) in that the different number of boundaries in each piece demands individual arguments to be made for each value of $d$, particularly, the number of generators and relations in each $V_i$ grows with $d$. As mentioned in Subsection \ref{MainResults}, this method can be used to compute an explicit presentation of $\B_{d,r}(O,Y)$ for any surface $O$, a torus $Y$, and any $d$ and $r$ such that either $d \geq 2$ or $r \geq 3$.

\subsection{Vertex stabilizers}\label{vertex}

By Lemma \ref{typesstab}, the stabilizer $V_i$ of each vertex $\v_i$ is isomorphic to $\Map_o(S_{\v_i})$, which by Theorem \ref{ThMain} has a presentation $\langle X_i|R_i\rangle$ with $$X_i:=\{x_{i,0},x_{i,1},y_{i,1},y_{i,2},...,y_{i,2i+1}, z_i, b_{i,1},b_{i,2},...,b_{i,i},u_{i,1},u_{i,2},...,u_{i,i+1}\},$$ and the relations described in Theorem \ref{ThMain}. Each $X_i$ is represented in Figure \ref{gen}, with the curves defining Dehn twists and the arcs defining boundary swaps. The following is an immediate consequence of Theorem \ref{presentationB}:

\begin{cor}\label{GenB}
The union $\underset{ 1\leq i\leq6}{\bigcup}{X_i}$ is a set of generators of $\B_{2,1}(O, Y)$, where $Y$ and $O$ are tori.
    
\end{cor}

\begin{figure}[H]  
\centering
\includegraphics[width=14.5cm]{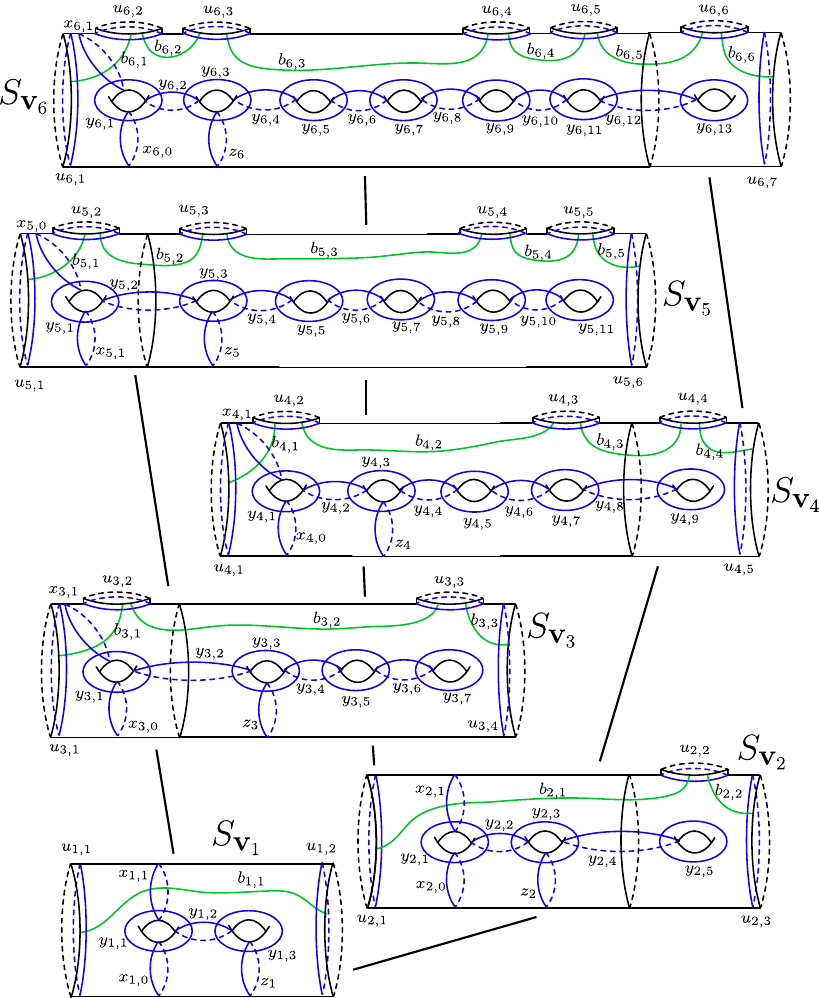}
\vspace{0.2cm}
\caption{Generators of $\B_{2,1}(O, Y)$. We represent each set $X_i$ as curves and arcs in $S_{\v_i}$. }
\label{gen}
\end{figure}

It remains to give a complete set of relations, which are the ones afforded by Lemma \ref{presentationB}. To this end, we compute the sets $R_{i,j}$. 

\subsection{Edge stabilizers of the form $E_{i,i+1}$}\label{5.2} We will now compute the sets of relations $R_{i,i+1}$.

\subsubsection{Computing $R_{1,2}$} The group $E_{1,2}\simeq \Map(S_{\mathbf{V}_1})$ by Lemma \ref{typesstab}, and the fact that both of the boundary components are fixed by mapping classes. Because of \cite[Proposition 2.10, Theorem 3.1]{Labruere-Paris}, the Dehn twists in Figure \ref{1-2} are a set of generators of $E_{1,2}$. We express each of them as a word in $X_1$ and $X_2$ and hence determine $R_{1,2}$: 

\begin{minipage}{0.3\textwidth}
  \centering
  \begin{itemize}
      \item $r_1: x_{1,0}=x_{2,0},$
      \item $r_2: x_{1,1}=x_{2,1},$
  \end{itemize}

\end{minipage}
\hfill
\begin{minipage}{0.3\textwidth}
  \centering
  \begin{itemize}
      \item $r_3: y_{1,1}=y_{2,1},$
      \item $r_4: y_{1,2}=y_{2,2},$
  \end{itemize}

\end{minipage}
\hfill
\begin{minipage}{0.3\textwidth}
  \begin{center}
  \begin{itemize}
      \item $r_5: y_{1,3}=y_{2,3},$
      \item $r_6: z_{1}=z_{2}.$
  \end{itemize}
  \end{center}
  
\end{minipage}

\begin{figure}[H]
\includegraphics[height=3.2cm]{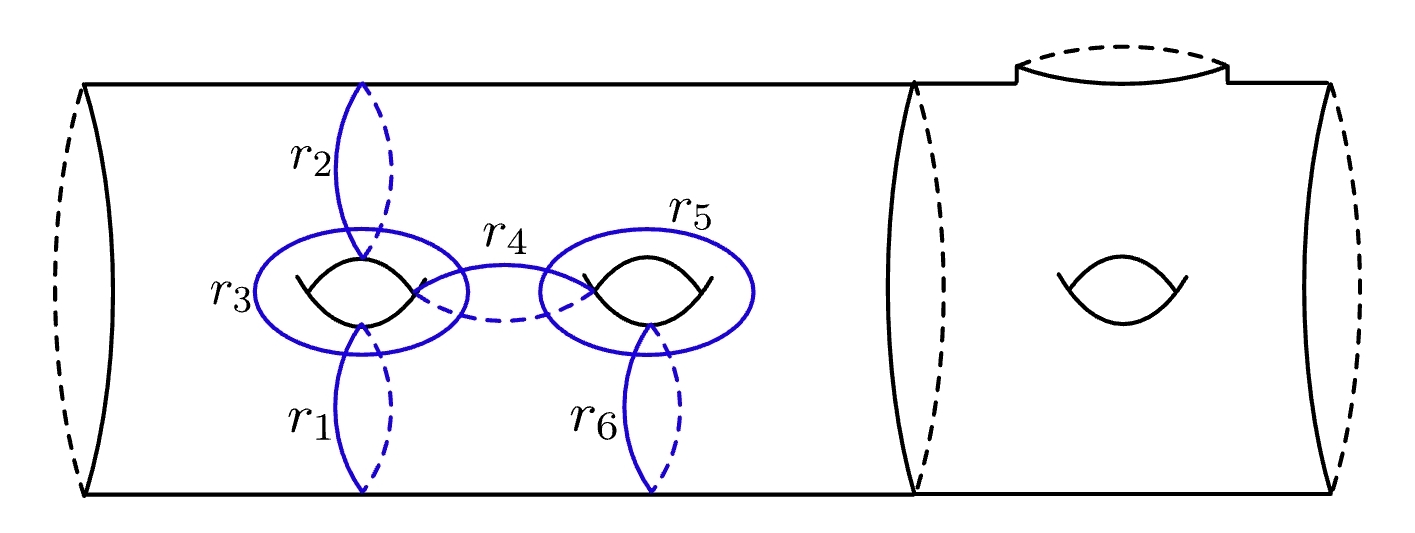}
\centering
\caption{}
\label{1-2}
\end{figure}

\subsubsection{Computing $R_{2,3}$}\label{7.2.2} By Lemmas \ref{typesstab} part (2) and \ref{lemSPMCG-A2}, the Dehn twists and boundary swaps in Figure \ref{2-3} are a set of generators of $E_{2,3}$. We compute $R_{2,3}$:

\begin{minipage}{0.3\textwidth}
  \centering
  \begin{itemize}
      \item $r_1: x_{2,0}=z_{3},$
      \item $r_3: y_{2,1}=y_{3,3},$
      \item $r_4: y_{2,2}=y_{3,4},$
  \end{itemize}

\end{minipage}
\hfill
\begin{minipage}{0.3\textwidth}
  \centering
  \begin{itemize}
      \item $r_5: y_{2,3}=y_{3,5},$
      \item $r_6: y_{2,4}=y_{3,6},$
      \item $r_7: y_{2,5}=y_{3,7},$
  \end{itemize}

\end{minipage}
\hfill
\begin{minipage}{0.3\textwidth}
  \begin{center}
  \begin{itemize}
      \item $r_{10}: b_{2,2}=b_{3,3},$
      \item $r_{11}: u_{2,2}=u_{3,3},$
      \item $r_{12}: u_{2,3}=u_{3,4},$
      
  \end{itemize}
  \end{center}
  
\end{minipage}

\begin{figure}[H]
\includegraphics[width=10cm]{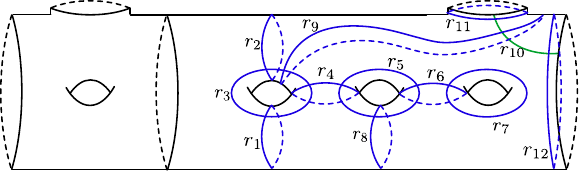}
\centering
\caption{}
\label{2-3}
\end{figure}

The Dehn twists $r_2$, $r_8$ and $r_9$ are not elements of $X_3$, hence we have to figure out expressions for them:

\textbf{Computing $r_8$:} The twist $r_8$ is $z_2$ viewed as an element of $X_2$. Let $\mathfrak{m}_i$ be the product of Dehn twists in Subsection \ref{6.2.2}.
According to the calculations of Humphries in \cite[Proof of Theorem, page 45]{HumpGen}, we have
$r_8=\mathfrak{m}_3x_{3,0}\mathfrak{m}_3^{-1}$. Hence we obtain the relation: $$r_8: z_2=\mathfrak{m}_3x_{3,0}\mathfrak{m}_3^{-1}.$$

\textbf{Computing $r_2$:} The twist $r_2$ is $x_{2,1}$ viewed as an element of $X_2$. On the other hand, by Lemma \ref{PVRepresentations2} (1) applied to the parabolic subgroup generated by $\{x_{t 3,1},b_{3,1}\}$, we get that $\Delta(x_{3,1},b_{3,1})=a_1x_{3,0}$, hence $a_1=\Delta(x_{3,1},b_{3,1})x_{3,0}^{-1}$. By Lemma \ref{PVRepresentations} (3) applied to the group generated by $\{a_1,x_{3,0},y_{3,1},y_{3,2}\}$, we know that $\Delta(a_1,x_{3,0},y_{3,1},y_{3,2})=a_2x_{2,1}z_3$, where $a_2=b_{3,1}^2u_1u_2$ (see Figure \ref{relespecial}). From these equalities we can deduce $\Delta(a_1,x_{3,0},y_{3,1},y_{3,2})=b_{3,1}^2u_1u_2x_{2,1}z_3$, and hence:
$$r_2:\Delta(\Delta(x_{3,1},b_{3,1})x_{3,0}^{-1},x_{3,0},y_{3,1},y_{3,2})=b_{3,1}^2u_{3,1}u_{3,2}x_{2,1}z_3.$$

\begin{figure}[H]  
\centering
\includegraphics[width=11cm]{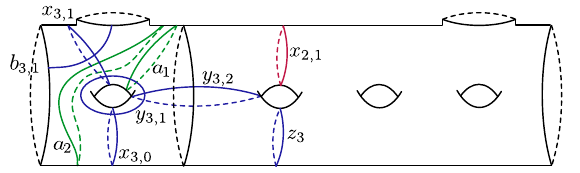}
\caption{}
\label{relespecial}
\end{figure}

\textbf{Computing $r_9$:} For $r_9$, by Lemmas \ref{PVRepresentations} (3) and \ref{PVRepresentations2} (3) applied to the groups generated by $\{x_{3,0},x_{3,1},y_{3,1},y_{3,2}\}$ and $\{x_{2,1},b_{3,2}\}$, we deduce that \begin{equation}\label{eqr9}
    r_9=\Delta(x_{2,1},b_{3,2})\Delta^{-1}(x_{3,0},x_{3,1},y_{3,1},y_{3,2})z_3u_{3,1}.\end{equation} Given that $x_{2,1}$ can be expressed as a product of elements in $X_3$ because of relation $r_2$, and every other factor in Equation \ref{eqr9} is an element of $X_3$, there is no need to add $r_9$ to $R_{2,3}$ since $r_9$ can be expressed as a product of elements in $X_3$.

\subsubsection{Computing $R_{i,i+1}$ for $i\geq3$.}\label{7.2.3} We now display the remaining sets of relations and omit the computations, as the determination of $R_{3,4}$ and $R_{5,6}$ is analogous to that of $R_{1,2}$, with generators given by Lemma \ref{lemSPMCG-A1}; while the calculations involved in $R_{4,5}$ are analogous to those in $R_{2,3}$, with the generators of $E_{4,5}$ determined by Lemma \ref{lemSPMCG-A2}.

\begin{itemize}[label=-]
    \item Relations $R_{3,4}:$
    
    \hspace{-1.8cm}
    \begin{minipage}{0.3\textwidth}
  \centering
  \begin{itemize}[label=$\bullet$]
      \item $r_1: x_{3,0}=x_{4,0},$
      \item $r_2: x_{3,1}=x_{4,1},$
      \item $r_3: y_{3,1}=y_{4,1},$
      \item $r_4: y_{3,2}=y_{4,2},$
      \item $r_5: y_{3,3}=y_{4,3},$
  \end{itemize}

\end{minipage}
\hfill
\begin{minipage}{0.3\textwidth}
  \centering
  \begin{itemize}[label=$\bullet$]
      \item $r_6: y_{3,4}=y_{4,4},$
      \item $r_7: y_{3,5}=y_{4,5},$
      \item $r_8: y_{3,6}=y_{4,6},$
      \item $r_9: y_{3,7}=y_{4,7},$
      \item $r_{10}: z_{3}=z_4,$
  \end{itemize}

\end{minipage}
\hfill
\begin{minipage}{0.3\textwidth}
  \begin{center}
  \begin{itemize}[label=$\bullet$]
      \item $r_{11}: b_{3,1}=b_{4,1},$
      \item $r_{12}: b_{3,2}=b_{4,2},$
      \item $r_{13}: u_{3,1}=u_{4,1},$
      \item $r_{14}: u_{3,2}=u_{4,2},$
      \item $r_{15}: u_{3,3}=u_{4,3}.$
  \end{itemize}
  \end{center}
  
\end{minipage}

\item Relations $R_{4,5}:$
    
    \hspace{-1.8cm}
    \begin{minipage}{0.3\textwidth}
  \centering
  \begin{itemize}[label=$\bullet$]
      \item $r_1: x_{4,0}=z_{5},$
      \item $r_3: y_{4,1}=y_{5,3},$
      \item $r_4: y_{4,2}=y_{5,4},$
      \item $r_5: y_{4,3}=y_{5,5},$
      \item $r_6: y_{4,4}=y_{5,6},$
      \item $r_7: y_{4,5}=y_{5,7},$
  \end{itemize}

\end{minipage}
\hfill
\begin{minipage}{0.3\textwidth}
  \centering
  \begin{itemize}[label=$\bullet$]
      \item $r_8: y_{4,6}=y_{5,8},$
      \item $r_9: y_{4,7}=y_{5,9},$
      \item $r_{10}: y_{4,8}=y_{5,10},$
      \item $r_{11}: y_{4,9}=y_{5,11},$
      \item $r_{13}: b_{4,2}=b_{5,3},$
      \item $r_{14}: b_{4,3}=b_{5,4},$
  \end{itemize}

\end{minipage}
\hfill
\begin{minipage}{0.3\textwidth}
  \begin{center}
  \begin{itemize}[label=$\bullet$]
      \item $r_{15}: b_{4,4}=b_{5,5},$
      \item $r_{16}: u_{4,2}=u_{5,3},$
      \item $r_{17}: u_{4,3}=u_{5,4},$
      \item $r_{18}: u_{4,4}=u_{5,5},$
      \item $r_{19}: u_{4,5}=u_{5,6},$
  \end{itemize}
  \end{center}
  
\end{minipage}
$$\bullet \;\; r_2: \Delta(\Delta(x_{5,1},b_{5,1})x_{5,0}^{-1},x_{5,0},y_{5,1},y_{5,2})=b_{5,1}^2u_{5,1}u_{5,2}x_{4,1}z_5,$$\vspace{-0.6cm}$$\bullet \;\; r_{12}: z_4=\mathfrak{m}_5x_{5,0}\mathfrak{m}_5^{-1}.$$

    \item Relations $R_{5,6}:$
    
    \hspace{-1.8cm}
    \begin{minipage}{0.3\textwidth}
  \centering
  \begin{itemize}[label=$\bullet$]
      \item $r_1: x_{5,0}=x_{6,0},$
      \item $r_2: x_{5,1}=x_{6,1},$
      \item $r_3: y_{5,1}=y_{6,1},$
      \item $r_4: y_{5,2}=y_{6,2},$
      \item $r_5: y_{5,3}=y_{6,3},$
      \item $r_6: y_{5,4}=y_{6,4},$
      \item $r_7: y_{5,5}=y_{6,5},$
      \item $r_8: y_{5,6}=y_{6,6},$
  \end{itemize}

\end{minipage}
\hfill
\begin{minipage}{0.3\textwidth}
  \centering
  \begin{itemize}[label=$\bullet$]
      \item $r_9: y_{5,7}=y_{6,7},$
      \item $r_{10}: y_{5,8}=y_{6,8},$
      \item $r_{11}: y_{5,9}=y_{6,9},$
      \item $r_{12}: y_{5,10}=y_{6,10},$
      \item $r_{13}: y_{5,11}=y_{6,11},$
      \item $r_{14}: z_{5}=z_{6},$
      \item $r_{15}: b_{5,1}=b_{6,1},$
      \item $r_{16}: b_{5,2}=b_{6,2},$
  \end{itemize}

\end{minipage}
\hfill
\begin{minipage}{0.3\textwidth}
  \begin{center}
  \begin{itemize}[label=$\bullet$]
      \item $r_{17}: b_{5,3}=b_{6,3},$
      \item $r_{18}: b_{5,4}=b_{6,4},$
      \item $r_{19}: u_{5,1}=u_{6,1},$
      \item $r_{20}: u_{5,2}=u_{6,2},$
      \item $r_{21}: u_{5,3}=u_{6,3},$
      \item $r_{22}: u_{5,4}=u_{6,4},$
      \item $r_{23}: u_{5,5}=u_{6,5}.$
  \end{itemize}
  \end{center}
  
\end{minipage}

\end{itemize}

\subsection{Edge stabilizers of the form $E_{i,i+2}$}\label{5.3}
As it was done in Subsection \ref{4.3}, it suffices to find a relation $r_{i,i+2}$ corresponding to an element that permutes the two pieces in $S_{\v_{i+2}}\setminus S_{\v_i}$.

Set $t_i$ to be the product recursively defined as:
\begin{itemize}
    \item $t_0=z_1$,
    \item $t_2=\mathfrak{m}_2x_{2,0}\mathfrak{m}_2^{-1},$
    \item $t_i=t_{i-1}$ if $i$ odd,
    \item $t_i=\mathfrak{n}_it_{i-4}\mathfrak{n}_i^{-1}$ if $i$ even, with \begin{align}
     \mathfrak{n}_i= y_{i,2i+1}y_{i,2i}y_{i,2i-1}t_{i-2}y_{i,2i-2}y_{i,2i-3}y_{i,2i-1}y_{i,2i-2}\cdot\nonumber\\\cdot y_{i,2i}y_{i,2i-1}y_{i,2i+1}y_{i,2i}t_{i-2}y_{i,2i-1}y_{i,2i-2}y_{i,2i-3}\nonumber.
    \end{align}
\end{itemize}
Define $$\phi_i=\Delta(x_{i,0},y_{i,1},...,y_{i,2i+1},t_i)^{-1}\Delta(b_{i,1},...,b_{i,i}).$$ By an analogous reasoning as for $r_{1,3}$, we get the following family of relations: for $ 2\leq i\leq 4,$ $$r_{i,i+2}:\phi_{i+2}u_{i+2,1}u_{i+2,2}=\phi_iu_{i,1}.$$

\subsection{The presentation} We collect the generators in Subsection \ref{vertex} and every set of relations in Subsections \ref{5.2} and \ref{5.3}, and Theorem \ref{ThMain}, which gives rise to the following presentation of $\B_{2,1}(O, Y)$.

\begin{theor}\label{Presentation2}
 Let $O$ and $Y$ be tori. The group $\B_{2,1}(O, Y)$ has a presentation with generators the Dehn twists $$\overset{6}{\underset{i=1}{\cup}}\{x_{i,0},x_{i,1},y_{i,1},y_{i,2},...,y_{i,2i+1}, z_i,u_{i,1},u_{i,2},...,u_{i,i+1}\}$$ and the boundary swaps $$\overset{6}{\underset{i=1}{\cup}}\{ b_{i,1},b_{i,2},...,b_{i,i}\},$$ and the relations are:

\begin{enumerate}
     \item  Relations from the vertex stabilizers:
    For $i\in \{1,...,6\}$,

\begin{itemize}[leftmargin=0pt]
\item Braid relations.  Let  $(a,b)\in \{(x_{i,1},y_{i,1}),(x_{i,0},y_{i,1}),(z_{i},y_{i,3}),\\(b_{i,1},b_{i,2}),...,(b_{i,i-1},b_{i,i}),(y_{i,1},y_{i,2}),...,(y_{i,2i},y_{i,2i+1})\}$:

\hspace{2pt} \textnormal{(A$_{a,b}$)}\quad  $aba=bab$.

\item 4-length relations. Let  $(a,b)\in \{(x_{i,1},b_{i,1}),(u_{i,1},b_{i,1}),...,(u_{i,i},b_{i,i}),\\(u_{i,2},b_{i,1}),...,(u_{i,i+1},b_{i,i})\}$:

\hspace{2pt} \textnormal{(A$_{a,b}$)}\quad  $abab=baba$.

\item Commutation. Let  $(a,b)$ be a pair of generators in $X_i$ that do not satisfy the braid or 4-length relations. Then:

\hspace{2pt} \textnormal{(A$_{a,b}$)}\quad  $ab=ba$.

\item Other relations.

\begin{tabular}{p{30pt} c l}
\rm(R1) \quad & $(y_{i,1}y_{i,2}y_{i,3}z_{i})^{10}=
(x_{i,0}y_{i,1}y_{i,2}y_{i,3}z_{i})^6$,\\

\rm(R2) \quad & $(y_{i,1}y_{i,2}y_{i,3}y_{i,4}y_{i,5}z_{i})^{12}=
(x_{i,0}y_{i,1}y_{i,2}y_{i,3}y_{i,4}y_{i,5}z_{i})^{15}$,\\

\rm(R3) \quad & $(x_{i,0}x_{i,1}y_{i,1}b_{i,1})^4=
(x_{i,1}y_{i,1}b_{i,1})^6$,\\

\rm(R4) \quad & $u_{i,1}(x_{i,0}x_{i,1}y_{i,1}y_{i,2}y_{i,3}z_{i})^5=
(x_{i,1}y_{i,1}y_{i,2}y_{i,3}z_{i})^8$,\\

\rm(R5a) \quad & $ x_{i,0}^{2g-n-2}(x_{i,1}b_{i,1}\dots
b_{i,n-1})^{n}=(z_{i}y_{i,2}\dots y_{i,2g-1})^{4g-4}$,\\
 
\rule{0pt}{13pt}\makebox[0pt][l] { For  $j\in\{1,...,i\}$:} \\
\rm(C1j) \quad & $b_{i,j}u_{i,j}=u_{i,j+1}b_{i,j},$ & \quad\\
\rm(C2j) \quad & $u_{i,j}b_{i,j}=b_{i,j}u_{i,j+1}.$ & \quad\\
\end{tabular}
\end{itemize}
\item Relations from edge stabilizers of the form $E_{i,i+1}$:

\hspace{-30pt}\begin{minipage}{0.3\textwidth}
  \centering
  \begin{itemize}[label={},leftmargin=*]
      \item $\textnormal{(S1)}\quad  x_{1,0}=x_{2,0},$
      \item $\textnormal{(S2)}\quad  x_{1,1}=x_{2,1},$
  \end{itemize}

\end{minipage}
\hfill
\begin{minipage}{0.3\textwidth}
  \centering
  \begin{itemize}[label={},leftmargin=*]
      \item $\textnormal{(S3)}\quad  y_{1,1}=y_{2,1},$
      \item $\textnormal{(S4)}\quad  y_{1,2}=y_{2,2},$
  \end{itemize}

\end{minipage}
\hfill
\begin{minipage}{0.3\textwidth}
  \begin{center}
  \begin{itemize}[label={},leftmargin=*]
      \item $\textnormal{(S5)}\quad  y_{1,3}=y_{2,3},$
      \item $\textnormal{(S6)}\quad  z_{1}=z_{2},$
  \end{itemize}
  \end{center}
  
\end{minipage}


\hspace{-30pt}\begin{minipage}{0.3\textwidth}
  \centering
  \begin{itemize}[label={},leftmargin=*]
      \item $\textnormal{(S7)}\quad x_{2,0}=z_{3},$
      \item $\textnormal{(S8)}\quad y_{2,1}=y_{3,3},$
      \item $\textnormal{(S9)}\quad y_{2,2}=y_{3,4},$
  \end{itemize}

\end{minipage}
\hfill
\begin{minipage}{0.3\textwidth}
  \centering
  \begin{itemize}[label={},leftmargin=*]
      \item $\textnormal{(S10)}\quad y_{2,3}=y_{3,5},$
      \item $\textnormal{(S11)}\quad y_{2,4}=y_{3,6},$
      \item $\textnormal{(S12)}\quad y_{2,5}=y_{3,7},$
  \end{itemize}

\end{minipage}
\hfill
\begin{minipage}{0.3\textwidth}
  \begin{center}
  \begin{itemize}[label={},leftmargin=*]
      \item $\textnormal{(S13)}\quad b_{2,2}=b_{3,3},$
      \item $\textnormal{(S14)}\quad u_{2,2}=u_{3,3},$
      \item $\textnormal{(S15)}\quad u_{2,3}=u_{3,4},$
      
  \end{itemize}
  \end{center}
  
\end{minipage}

$$\textnormal{(S16)}\quad((x_{3,1}b_{3,1})^2x_{3,0}^{-1}x_{3,0}y_{3,1}y_{3,2})^3=b_{3,1}^2u_{3,1}u_{3,2}x_{2,1}z_3.$$\vspace{-0.6cm}$$\textnormal{(S17)}\quad z_2=\mathfrak{m}_3x_{3,0}\mathfrak{m}_3^{-1},$$


\hspace{-30pt}\begin{minipage}{0.3\textwidth}
  \centering
  \begin{itemize}[label={},leftmargin=*]
      \item $\textnormal{(S18)}\quad x_{3,0}=x_{4,0},$
      \item $\textnormal{(S19)}\quad x_{3,1}=x_{4,1},$
      \item $\textnormal{(S20)}\quad y_{3,1}=y_{4,1},$
      \item $\textnormal{(S21)}\quad y_{3,2}=y_{4,2},$
      \item $\textnormal{(S22)}\quad y_{3,3}=y_{4,3},$
  \end{itemize}

\end{minipage}
\hfill
\begin{minipage}{0.3\textwidth}
  \centering
  \begin{itemize}[label={},leftmargin=*]
      \item $\textnormal{(S23)}\quad  y_{3,4}=y_{4,4},$
      \item $\textnormal{(S24)}\quad  y_{3,5}=y_{4,5},$
      \item $\textnormal{(S25)}\quad  y_{3,6}=y_{4,6},$
      \item $\textnormal{(S26)}\quad y_{3,7}=y_{4,7},$
      \item $\textnormal{(S27)}\quad  z_{3}=z_4,$
  \end{itemize}

\end{minipage}
\hfill
\begin{minipage}{0.3\textwidth}
  \begin{center}
  \begin{itemize}[label={},leftmargin=*]
      \item $\textnormal{(S28)}\quad  b_{3,1}=b_{4,1},$
      \item $\textnormal{(S29)}\quad  b_{3,2}=b_{4,2},$
      \item $\textnormal{(S30)}\quad  u_{3,1}=u_{4,1},$
      \item $\textnormal{(S31)}\quad  u_{3,2}=u_{4,2},$
      \item $\textnormal{(S32)}\quad  u_{3,3}=u_{4,3},$
  \end{itemize}
  \end{center}
  
\end{minipage}


\hspace{-30pt}\begin{minipage}{0.3\textwidth}
  \centering
  \begin{itemize}[label={},leftmargin=*]
      \item $\textnormal{(S33)}\quad x_{4,0}=z_{5},$
      \item $\textnormal{(S34)}\quad y_{4,1}=y_{5,3},$
      \item $\textnormal{(S35)}\quad y_{4,2}=y_{5,4},$
      \item $\textnormal{(S36)}\quad y_{4,3}=y_{5,5},$
      \item $\textnormal{(S37)}\quad y_{4,4}=y_{5,6},$
      \item $\textnormal{(S38)}\quad y_{4,5}=y_{5,7},$
  \end{itemize}

\end{minipage}
\hfill
\begin{minipage}{0.3\textwidth}
  \centering
  \begin{itemize}[label={},leftmargin=*]
      \item $\textnormal{(S39)}\quad y_{4,6}=y_{5,8},$
      \item $\textnormal{(S40)}\quad y_{4,7}=y_{5,9},$
      \item $\textnormal{(S41)}\quad y_{4,8}=y_{5,10},$
      \item $\textnormal{(S42)}\quad y_{4,9}=y_{5,11},$
      \item $\textnormal{(S43)}\quad b_{4,2}=b_{5,3},$
      \item $\textnormal{(S44)}\quad b_{4,3}=b_{5,4},$
  \end{itemize}

\end{minipage}
\hfill
\begin{minipage}{0.3\textwidth}
  \begin{center}
  \begin{itemize}[label={},leftmargin=*]
      \item $\textnormal{(S45)}\quad b_{4,4}=b_{5,5},$
      \item $\textnormal{(S46)}\quad u_{4,2}=u_{5,3},$
      \item $\textnormal{(S47)}\quad u_{4,3}=u_{5,4},$
      \item $\textnormal{(S48)}\quad u_{4,4}=u_{5,5},$
      \item $\textnormal{(S49)}\quad u_{4,5}=u_{5,6},$
  \end{itemize}
  \end{center}
  
\end{minipage}
$$\textnormal{(S50)}\quad \Delta(\Delta(x_{5,1},b_{5,1})x_{5,0}^{-1},x_{5,0},y_{5,1},y_{5,2})=b_{5,1}^2u_{5,1}u_{5,2}x_{4,1}z_5,$$\vspace{-0.6cm}$$\textnormal{(S51)}\quad z_4=\mathfrak{m}_5x_{5,0}\mathfrak{m}_5^{-1},$$


\hspace{-30pt}\begin{minipage}{0.3\textwidth}
  \centering
  \begin{itemize}[label={},leftmargin=*]
      \item $\textnormal{(S52)}\quad x_{5,0}=x_{6,0},$
      \item $\textnormal{(S53)}\quad x_{5,1}=x_{6,1},$
      \item $\textnormal{(S54)}\quad y_{5,1}=y_{6,1},$
      \item $\textnormal{(S55)}\quad y_{5,2}=y_{6,2},$
      \item $\textnormal{(S56)}\quad y_{5,3}=y_{6,3},$
      \item $\textnormal{(S57)}\quad y_{5,4}=y_{6,4},$
      \item $\textnormal{(S58)}\quad y_{5,5}=y_{6,5},$
      \item $\textnormal{(S59)}\quad y_{5,6}=y_{6,6},$
  \end{itemize}

\end{minipage}
\hfill
\begin{minipage}{0.3\textwidth}
  \centering
  \begin{itemize}[label={},leftmargin=*]
      \item $\textnormal{(S60)}\quad y_{5,7}=y_{6,7},$
      \item $\textnormal{(S61)}\quad y_{5,8}=y_{6,8},$
      \item $\textnormal{(S62)}\quad y_{5,9}=y_{6,9},$
      \item $\textnormal{(S63)}\quad y_{5,10}=y_{6,10},$
      \item $\textnormal{(S64)}\quad y_{5,11}=y_{6,11},$
      \item $\textnormal{(S65)}\quad z_{5}=z_{6},$
      \item $\textnormal{(S66)}\quad b_{5,1}=b_{6,1},$
      \item $\textnormal{(S67)}\quad b_{5,2}=b_{6,2},$
  \end{itemize}

\end{minipage}
\hfill
\begin{minipage}{0.3\textwidth}
  \begin{center}
  \begin{itemize}[label={},leftmargin=*]
      \item $\textnormal{(S68)}\quad b_{5,3}=b_{6,3},$
      \item $\textnormal{(S69)}\quad b_{5,4}=b_{6,4},$
      \item $\textnormal{(S70)}\quad u_{5,1}=u_{6,1},$
      \item $\textnormal{(S71)}\quad u_{5,2}=u_{6,2},$
      \item $\textnormal{(S72)}\quad u_{5,3}=u_{6,3},$
      \item $\textnormal{(S73)}\quad u_{5,4}=u_{6,4},$
      \item $\textnormal{(S74)}\quad u_{5,5}=u_{6,5}.$
  \end{itemize}
  \end{center}
  
\end{minipage}

\item Relations from edge stabilizers of the form $E_{i,i+2}$: For $i\in \{1,...,4\}$, $$\textnormal{(Ti)}\quad \phi_{i+2}u_{i+2,1}u_{i+2,2}=\phi_iu_{i,1}.$$

 \end{enumerate}
\end{theor}

\bibliographystyle{plain}
\bibliography{name.bib}

\end{document}